\documentclass[10pt]{amsart}
\usepackage{amscd}
\usepackage{amsmath}
\usepackage{amssymb}
\usepackage{latexsym}
\usepackage{lscape}
\usepackage{xypic}
\usepackage{url}

\newtheorem{thm}{Theorem}[section]
\newtheorem{prop}[thm]{Proposition}
\newtheorem{lem}[thm]{Lemma}

\newtheorem{lem-def}[thm]{Lemma-Definition}
\newtheorem{cor}[thm]{Corollary}
\theoremstyle{remark}

\newtheorem{rmk}[thm]{Remark}
\theoremstyle{definition}
\newtheorem{dfn}[thm]{Definition}

\numberwithin{equation}{subsection}

\input xy
\xyoption{all}

\newcommand{\quash}[1]{}  
\newcommand{\nc}{\newcommand}
\nc{\on}{\operatorname}

\newcommand{\frakg}{{\mathfrak g}}

\newcommand{\frakn}{{\mathfrak n}}

\newcommand{\frakX}{{\mathfrak X}}
\newcommand{\frakY}{{\mathfrak Y}}
\newcommand{\frakZ}{{\mathfrak Z}}

\newcommand{\bA}{{\mathbb A}}
\newcommand{\bB}{{\mathbb B}}
\newcommand{\bC}{{\mathbb C}}
\newcommand{\bD}{{\mathbb D}}

\newcommand{\bF}{{\mathbb F}}
\newcommand{\bG}{{\mathbb G}}

\newcommand{\bP}{{\mathbb P}}
\newcommand{\bQ}{{\mathbb Q}}

\newcommand{\bS}{{\mathbb S}}

\newcommand{\bX}{{\mathbb X}}

\newcommand{\bZ}{{\mathbb Z}}

\newcommand{\mA}{{\mathcal A}}

\newcommand{\mC}{{\mathcal C}}

\newcommand{\mE}{{\mathcal E}}
\newcommand{\mF}{{\mathcal F}}
\newcommand{\mG}{{\mathcal G}}
\newcommand{\mH}{{\mathcal H}}

\newcommand{\mL}{{\mathcal L}}
\newcommand{\mM}{{\mathcal M}}

\newcommand{\mO}{{\mathcal O}}

\newcommand{\mQ}{{\mathcal Q}}

\newcommand{\mT}{{\mathcal T}}

\newcommand{\mX}{{\mathcal X}}
\newcommand{\mY}{{\mathcal Y}}

\nc{\al}{{\alpha}} \nc{\be}{{\beta}} \nc{\ga}{{\gamma}}
\nc{\ve}{{\varepsilon}} \nc{\Ga}{{\Gamma}} \nc{\la}{{\lambda}}
\nc{\La}{{\Lambda}}

\nc{\ad}{{\on{ad}}}

\nc{\aff}{{\on{aff}}}
\nc{\Aff}{{\mathbf{Aff}}}
\newcommand{\Aut}{{\on{Aut}}}
\nc{\Bun}{{\on{Bun}}}

\nc{\der}{{\on{der}}}

\nc{\diag}{{\on{diag}}}
\newcommand{\End}{{\on{End}}}
\nc{\Fl}{{\mF\ell}}
\newcommand{\Gal}{{\on{Gal}}}
\newcommand{\Gr}{{\on{Gr}}}
\newcommand{\Hom}{{\on{Hom}}}
\nc{\IC}{{\on{IC}}}
\newcommand{\id}{{\on{id}}}
\nc{\Id}{{\on{Id}}}

\nc{\Ind}{{\on{Ind}}}
\nc{\inv}{{\on{Inv}}}
\nc{\Iso}{{\on{Isom}}}
\newcommand{\Lie}{{\on{Lie}}}
\newcommand{\Pic}{{\on{Pic}}}
\newcommand{\pr}{{\on{pr}}}
\nc{\Ql}{{\overline{\bQ}_\ell}}
\nc{\res}{{\on{res}}}
\newcommand{\Res}{{\on{Res}}}

\newcommand{\Spec}{{\on{Spec}\ \!}}

\nc{\tr}{{\on{tr}}}

\nc{\mmu}{\mu_\bullet}

\newcommand{\GL}{{\on{GL}}}
\nc{\GSp}{{\on{GSp}}} \nc{\GU}{{\on{GU}}} \nc{\SL}{{\on{SL}}}
\nc{\SU}{{\on{SU}}} \nc{\SO}{{\on{SO}}}

\nc{\Hk}{{\on{Hk}}}
\nc{\lhk}{\Hk^{\on{loc}}}

\nc{\pf}{{p^{-\infty}}}
\nc{\bb}{{\mathbf{b}}}

\newcommand{\bGr}{{\overline{\Gr}}}
\nc{\wGr}{{\widetilde{\Gr}}}
\nc{\Sat}{{\on{Sat}}}

\def\xcoch{\mathbb{X}_\bullet}
\def\xch{\mathbb{X}^\bullet}

\title{Affine Grassmannians and the geometric Satake in mixed characteristic}
\author{Xinwen Zhu}  \email{xzhu@caltech.edu} 

\begin{document}
\maketitle{}
\begin{abstract}We endow the set of lattices in $\bQ_p^n$ with a reasonable algebro-geometric structure. As a result, we prove the representability of affine Grassmannians and establish the geometric Satake equivalence in mixed characteristic. We also give an application of our theory to the study of Rapoport-Zink spaces. 
\end{abstract}

\tableofcontents

\section*{Introduction}
\subsection{Mixed characteristic affine Grassmannians}
\subsubsection{}
Let $F$ be a non-archimedean local field, i.e., $F$ is either $\bF_q((t))$ or a finite extension of $\bQ_p$, with $\mO\subset F$ its ring of integers. Let $V=F^n$ denote an $n$-dimensional $F$-vector space. A lattice of $V$ is a finitely generated $\mO$-submodule $\La$ of $V$ such that $\La\otimes F=V$. For example $\La_0=\mO^n$ is a lattice in $V$, and every other lattice in $V$ can be translated to $\La_0$ by a linear automorphism of $V$. Therefore the set of lattices in $V$ can be identified with the set $\GL_n(F)/\GL_n(\mO)$. 

For various applications in number theory, representation theory and algebraic geometry, it is highly desirable to realize this set as the set of ($k$-)points of some (reasonable, infinite dimensional) algebraic variety defined over $k$, where $k$ denotes the residue field of $F$. If $F=k((t))$, such algebro-geometric object, called the affine Grassmannian, does exist, and plays a fundamental role in geometric representation theory and in the study of moduli spaces of vector bundles on algebraic curves. However, a reasonable algebro-geometric structure on the set $\GL_n(\bQ_p)/\GL_n(\bZ_p)$ is not available for many years, although some attempts have been made (\cite{Ha,Kr,CKV}). The first goal of this paper is to give a solution of this problem to some extent. We will call this new algebro-geometric object the mixed characteristic affine Grassmannian.

\subsubsection{}
To explain the ideas, let us first recall the equal  characteristic story (see, e.g. \cite{BL} for details). First one can make sense of the notion of a family of lattices in $k((t))^n$: for a $k$-algebra $R$, a lattice in $R((t))^n$ is a finitely generated projective $R[[t]]$-submodule $\La$ of $R((t))^n$ such that $\La\otimes_{R[[t]]}R((t))=R((t))^n$. The affine Grassmannian $\Gr^\flat$\footnote{In this paper, objects defined in the equal characteristic setting are usually written with $\flat$ in their supscripts. But this does not mean that they are the tilts of the corresponding mixed characteristic objects as in Scholze's theory of perfectoid spaces.} is defined as the moduli space that assigns every $k$-algebra $R$ the set of lattices in $R((t))^n$. In particular, the set of $k$-points of $\Gr^\flat$ is just $\GL_n(k((t)))/\GL_n(k[[t]])$. 

Given a lattice $\La$ in $R((t))^n$, there always exists some big integer $N$ such that 
\begin{equation}\label{subfun}
t^NR[[t]]^n\subset \La\subset t^{-N}R[[t]]^n.
\end{equation}
So $\Gr^{\flat}$ is the union of subfunctors $\Gr^{\flat,(N)}$ consisting of those lattices satisfying \eqref{subfun}.
The key fact is that via the map
$$\La\mapsto t^{-N}R[[t]]^n/\La,$$ $\Gr^{\flat,(N)}(R)$ is identified with the set of quotient $R[[t]]$-modules of $t^{-N}R[[t]]^n/t^NR[[t]]^n$ that are projective as $R$-modules. Then it is not hard to see that $\Gr^{\flat,(N)}$ is represented by a closed subscheme of the usual Grassmannian variety.

\subsubsection{}
There is an obvious guess of the moduli problem that the mixed characteristic affine Grassmannian $\Gr$ should represent: it should associate to every $k$-algebra $R$ the set 
\begin{equation}\label{intro:Gr}
\big\{\mbox{finite projective } W(R)\mbox{-submodules } \La \mbox{ of } W(R)[1/p]^n \mbox{ such that } \La[1/p]=W(R)[1/p]^n \big\},
\end{equation} 
where $W(R)$ is the ring of Witt vectors for $R$. 
Unfortunately, this definition is unreasonable\footnote{Alternatively, one could try to define $\Gr(R)$ as the set of pairs $(\La,\beta)$, where $\La$ is a finite projective $W(R)$-module and $\beta: \La\otimes W(R)[1/p]\simeq W(R)[1/p]^n$ is an isomorphism. But we still do not know whether this is reasonable.}, as the ring of the Witt vectors for a non-perfect ring $R$ is not well-behaved. E.g. $p$ could be a zero divisor of $W(R)$ if $R$ is non-reduced, so $\La_{0R}=W(R)^n$ may not be a submodule of $W(R)[1/p]^n$. On the other hand, note that
\begin{enumerate}
\item if $R$ is a perfect $k$-algebra, then $W(R)$ is well-behaved.
\item The values of a scheme $X$ at perfect rings $R$ determine $X$ up to perfection\footnote{The category of perfect $k$-schemes is a full subcategory of the category of presheaves on the category of perfect $k$-algebras. See Lemma \ref{descent}.}. 
\item The (\'{e}tale) topology of a scheme (e.g. the $\ell$-adic cohomology) does not change when passing to the perfection.
\end{enumerate}

Therefore, we restrict the naive moduli problem \eqref{intro:Gr} to the category of perfect $k$-algebras. This defines a presheaf on this category\footnote{In fact, Kreidl \cite{Kr} proved that this is an fpqc sheaf.}, and the best question one can ask is: whether this functor is represented by a(n inductive limit of) perfect $k$-scheme(s). Our first main theorem gives a positive answer to a slightly weaker version of the question.

\begin{thm}\label{intro:main}
The above functor can be written as an increasing union of subfunctors $\Gr=\underrightarrow\lim X_i $, where each $X_i$ is the perfection of some proper algebraic space defined over $k$ and $X_i\to X_{i+1}$ is a closed embedding. 
\end{thm}

Perfect $k$-schemes/algebraic spaces are almost never of finite type over $k$. But as stated in Theorem \ref{intro:main}, each $X_i$ appearing above is  in fact the perfection of some proper algebraic space over $k$\footnote{It is expected that these $X_i$s are the perfections of projective varieties over $k$. See Appendix \ref{rep as sch} for further discussions. But knowing that they are algebraic spaces is sufficient for all the applications we have in mind.}\footnote{That $X_i$s are the perfections of projective varieties was proved very recently by Bhatt-Scholze \cite{BS}.}. We do not know how to canonically construct these algebraic spaces without passing to the perfection. But this does not bother us. We can still study their topological properties. In particular, we can define the $\ell$-adic derived category on $X_i$, the notion of perverse sheaves, and etc. 

As soon as the representability of $\Gr$ is known, the representability of mixed characteristic affine Grassmannians and affine flag varieties for general reductive groups follows by the same argument as in equal characteristic situation. See \S~\ref{gen aff flag}.

\subsubsection{} Now we explain some ideas behind the proof of Theorem \ref{intro:main}. As in the equal  characteristic situation, it is enough to prove the representability of the subfunctor $\Gr^{(N)}$ of $\Gr$ defined by a condition similar to \eqref{subfun}. With a little further work, one then reduces to prove the representability of the following functor
\[\bGr_N=\left\{\La\in \Gr \mid \La\subset \La_{0R} \mbox{ such that } \wedge^n\La= p^N\wedge^n\La_{0R}\right\},\]
where $\La_{0R}=W(R)^n$, and $\wedge^n(-)$ denotes the $n$th wedge product.
Now, the essential difficulty is that the quotient $\La_{0R}/\La$ is not an $R$-module so the previous strategy to embed this functor into the Grassmannian fails. In fact, a basic question is whether there exists a non-trivial line bundle on $\bGr_N$. Note that in equal characteristic, such a line bundle exists, known as the determinant line bundle $\mL^\flat_{\det}$ on $\bGr^\flat_N$.  Its fiber over a lattice $\La\subset R[[t]]^n$ is the top exterior wedge of $R[[t]]^n/\La$ regarded as an $R$-module. This construction certainly fails in mixed characteristic (see Appendix \ref{det line} for more discussions).

Therefore, we proceed in another way. Our observation is that after adding level structures, $\bGr_N$ is represented by an affine scheme defined by matrix equations. More precisely, for each $h$, let $W_h(R)$ denote the ring of truncated Witt vectors of length $h$, and define
\[\bGr_{N,h}=\left\{(\La,\bar\epsilon)\mid \La\in \bGr_N, \bar\epsilon: W_{h}(R)^n\simeq \La/p^h\La\right\}.\]
This is an $L^h\GL_n$-torsor over $\bGr_{N}$, where $L^h\GL_n$ denotes the affine $k$-group scheme which is the perfection of  the Greenberg realization of $\GL_n$ over $\mO/p^h$ and which acts on $\bGr_{N,h}$ by changing the isomorphism $\bar\epsilon$. We will show that when $h>N$, $\bGr_{N,h}$ can be (non-canonically) identified  with the following scheme of pairs of matrices
$$\left\{(A,\gamma)\mid \gamma\in L^h\GL_n, A\in L^hM_{n}, \det A\in p^N(\mO/p^h\mO)^\times, A\gamma=A\right\},$$
where $M_n$ denotes the scheme of all $n\times n$ matrices, and $L^hM_n$ denotes the perfection of its Greenberg realization over $\mO/p^h$. In fact, $A$ is the matrix representing the map 
$$W_h(R)^n\stackrel{\bar \epsilon}{\simeq} \La/p^h\La\to \La_{0R}/p^h\La_{0R}=W_h(R)^n.$$

Therefore, $\bGr_N$ can be expressed as a quotient of an affine scheme by a free action of an affine group scheme. One can expect that such a quotient should exist, at least as an algebraic space over $k$. This is indeed the case here, but cannot follow directly from the general theory, because neither $\bGr_{N,h}$ nor $L^h\GL_n$ is of finite type. However, we manage to prove the
following result, which is enough to conclude Theorem \ref{intro:main}. 

\begin{thm}(See Theorem \ref{quotient}) Let $G$ be the perfection of an algebraic group over $k$ and let $X$ be the perfection of an affine scheme of finite type over $k$. Assume that $G$ acts on $X$ and that the action map $G\times X\to X\times X,\ (g,x)\mapsto (gx,x)$ is a closed embedding. Then the quotient $[X/G]$ (as a stack) is represented by the perfection of an algebraic space separated and of finite type over $k$. 
\end{thm}

\subsection{The geometric Satake}
There are a lot of applications of mixed characteristic affine Grassmannians. The most fundamental one is to establish the geometric Satake equivalence in mixed characteristic. Its equal characteristic counterpart is a result of works of Lusztig, Drinfeld, Ginzburg, Mirkovi\'{c}-Vilonen (cf. \cite{Lu,Gi1,BD,MV}).
In a forthcoming joint work with Liang Xiao \cite{XZ}, we will apply the mixed characteristic geometric Satake to the study of some arithmetic geometry of Shimura varieties\footnote{We refer to \cite{Laf} for some amazing applications of the equal  characteristic geometric Satake to the Langlands correspondence over function fields.}.

Let $G$ be a reductive group over $\mO$, the ring of integers of a $p$-adic field $F$, and let $\Gr_G$ denote its affine Grassmannian. As explain above, it makes sense to define the category of $L^+G$-equivariant perverse sheaves (with coefficients in $\Ql$) on $\Gr_G$, denoted by $\on{P}_{L^+G}(\Gr_G)$. As in the equal characteristic situation, this is a semisimple monoidal category with the monoidal structure given by Lusztig's convolution product of sheaves. In addition, one can still endow the hypercohomology functor $\on{H}^*(\Gr_G,-):\on{P}_{L^+G}(\Gr_G)\to \on{Vect}_{\Ql}$ with a canonical monoidal structure (although the methods of \cite{Gi1,BD,MV} do not work directly in our setting). 
Our second main theorem is the geometric Satake equivalence in this setting.

\begin{thm}\label{intro:geomSat}
The monoidal functor $\on{H}^*$ factors as the composition of an equivalence of monoidal categories from $\on{P}_{L^+G}(\Gr_G)$ to the category $\on{Rep}_{\Ql}(\hat{G})$ of finite dimensional representations of the Langlands dual group $\hat{G}$ over  $\Ql$ and the forgetful functor from $\on{Rep}_{\Ql}(\hat{G})$ to the category $\on{Vect}_{\Ql}$ of finite dimensional $\Ql$-vector spaces.
\end{thm}
The theorem in particular implies that there exist the commutativity constraints of $\on{P}_{L^+G}(\Gr_G)$ such that $\on{H}^*$ is a tensor functor. In equal characteristic, such constraints come from the interpretation of the convolution product as the fusion product (cf. \cite[\S 5]{MV} or \cite[\S 5.3.17]{BD}). As the fusion product currently does not exist in mixed characteristic\footnote{But we note that the recent work Scholze on diamonds opens a door to this direction. See \cite{SW}.}, it is probably surprising that we can still establish these constraints in the current setting. 

In fact, a construction of the commutativity constraints using a categorical version of the classical Gelfand's trick already appeared in \cite{Gi1}. It was then claimed in \cite[\S 5.3.8, \S 5.3.9]{BD} (but without proof) that (a modification of) Ginzburg's construction coincides with the construction via the fusion product. Therefore, we do have candidates of the commutativity constraints even in mixed characteristic. The problem is that it is not clear how to verify the properties they suppose to satisfy (e.g. the hexagon axiom), without using the fusion interpretation.

Our new observation is that the validity of these properties is equivalent to a numerical result for the affine Hecke algebra. Namely, in \cite{LV,Lu2} Lusztig and Vogan introduced, for a Coxeter system $(W,S)$ with an involution, certain polynomials $P^\sigma_{y,w}(q)$ similar to the usual Kazhdan-Lusztig polynomials $P_{y,w}(q)$ (\cite{KaLu}). Then it was conjectured in \cite{Lu2} that if $(W,S)$ is an affine Weyl group and $y,w$ are certain elements in $W$,
\[P_{y,w}^\sigma(q)=P_{y,w}(-q).\]
See \emph{loc. cit.} or \S \ref{comb formula} for more details. This conjecture is purely combinatoric, but its proof by Lusztig and Yun \cite{LY} is geometric, which in fact uses the equal characteristic geometric Satake! We then go in the opposite direction by showing that this formula implies that the above mentioned commutativity constraints are the correct ones. 

So our proof of Theorem \ref{intro:geomSat} uses the geometric Satake in equal characteristic. It is an interesting question to find a direct proof of the above combinatoric formula, which will yield a purely local proof of the geometric Satake, in both equal and mixed characteristic.

Along the way of our proof, we also establish the Mirkovi\'{c}-Vilonen theory in mixed characteristic. This is very useful to the study of affine Deligne-Lusztig varieties. See below.

\subsection{Dimension of affine Deligne-Lusztig varieties}
One motivation to introduce mixed characteristic affine Grassmannians is their relation to Rapoport-Zink spaces. Let $G$ be a connected reductive group over $\bQ_p$, and assume (for simplicity) that there is an extension of $G$ to a reductive group scheme over $\bZ_p$, still denoted by $G$. Let $k=\bar\bF_p$ and let $L=W(k)[1/p]$ denote the completion of the maximal unramified extension of $\bQ_p$. Let $\sigma$ denote the Frobenius automorphism of $L$. Let $b$ be a $\sigma$-conjugacy class of $G(L)$ and $\mu$ a geometric conjugacy class of one parameter subgroups of $G$ (a.k.a. a dominant coweight of $G$ with respect to some chosen Borel). When the triple $(G,b,\mu)$ (sometimes called a Rapoport-Zink datum) comes from a PEL-datum, Rapoport-Zink (cf. \cite{RZ}) defined a formal scheme $\breve\mM(G,b,\mu)$, locally of finite type over $W=W(k)$, as certain moduli problem of $p$-divisible groups. Recently, Kim (cf. \cite{Ki}) and Howard-Pappas (cf. \cite{HP}) generalized the definition of $\breve\mM(G,b,\mu)$ to include those Rapoport-Zink data of Hodge type and proved the representability of $\breve\mM(G,b,\mu)$ in the case $p>2$. In any case, a serious restriction is that $\mu$  must be \emph{minuscule}. Under this assumption, by the Dieudonn\'{e} theory one identifies the set of $k$-points of $\breve\mM(G,b,\mu)$ with the set
\begin{equation}
X_\mu(b)= \{g\in G(L)/G(W)\mid g^{-1}b\sigma(g)\in G(W)p^\mu G(W)\}.
\end{equation}
This identification endows $X_\mu(b)$ with an algebro-geometric structure, and therefore $X_\mu(b)$ is  sometimes called an affine Deligne-Lusztig variety (cf. \cite{R}). Note that the definition of $X_\mu(b)$ as a set makes sense for any triple $(G,b,\mu)$, but only for minuscule $\mu$, $X_\mu(b)$ may relate to the moduli of $p$-divisible groups. 

It has been hoped to endow $X_\mu(b)$ with an algebro-geometric structure without using $p$-divisible groups and the Dieudonn\'{e} theory. Now, the existence of the mixed characteristic affine Grassmannian $\Gr_G$ (for $G$ over a $p$-adic field $F$) allows us to realize $X_\mu(b)$ as a (locally) closed subset of $\Gr_G$\footnote{Note that \eqref{ADL2} shows that $X_\mu(b)$ itself is a kind of moduli space of mixed characteristic local Shtukas over $k$.} and therefore to give $X_\mu(b)$ a structure as an ind-perfect algebraic space. Note that in this new definition, there is no restriction on the cocharacter $\mu$. But when $(G,b,\mu)$ arises as a(n unramified) Rapoport-Zink datum of Hodge type as above, we have the following proposition, as a simple application of the equivalence of categories between $p$-divisible groups and $F$-crystals over a perfect ring in characteristic $p>0$ (a theorem of Gabber, see also \cite[\S 6]{La}).
\begin{prop}\label{intro:RZ}
Let $\overline\mM^\pf_\mu(b)$ denote the perfection of the special fiber of $\breve\mM(G,b,\mu)$. Then
there is a canonical isomorphism of spaces
$\overline\mM^\pf_\mu(b)\simeq X_\mu(b)$.
\end{prop}
Even if the primary interests are the study of the Rapoport-Zink spaces, having another definition of $X_\mu(b)$ gives us extra flexibility. For example, the new definition is group theoretical, so allows us to study $\breve\mM(G,b,\mu)$ by using root subgroups or Levi subgroups $G$, or passing to central isogenies of $G$.

In a forthcoming work \cite{XZ}, these extra flexibilities allow us to understand the irreducible components of certain RZ spaces. Here, we illustrate this idea by one simple example\footnote{An earlier example is the study of the connected components of $X_\mu(b)$ in \cite{CKV}, although the notion of connected components of $X_\mu(b)$ was defined in \emph{loc. cit.} in an ad hoc way, due to the lack of the representability of $X_\mu(b)$ at the time.}:
we prove the dimension formula of $X_\mu(b)$ as conjectured by Rapoport. 
\begin{thm}\label{intro:dim}
The ind-perfect algebraic space $X_\mu(b)$ is finite dimensional, and
\[\dim X_\mu(b)=\langle\rho,\mu-\nu_b\rangle-\frac{1}{2}\on{def}_G(b).\]
\end{thm}
We refer to \S~\ref{dim:ADLV} for unexplained notations. Thanks to Proposition~\ref{intro:RZ}, when $(G,b,\mu)$ is of Hodge type, we obtain the dimension formula of the corresponding Rapoport-Zink space\footnote{When the first draft version of the paper was completed, Hamacher released his preprint  \cite{Ham1} where Rapoport's conjecture for PEL type Rapoport-Zink spaces was solved.}.

Not surprisingly, after the machinery is set up, we can imitate the methods used in the equal characteristic situation (with one innovation): we can apply the arguments of \cite{GHKR} in the current setting and reduce to prove Theorem~\ref{intro:dim} for those $b$ that are the so-called superbasic $\sigma$-conjugacy classes. It was shown in \cite{GHKR,CKV} that if $G$ is of adjoint type, superbasic $\sigma$-conjugacy classes exist only when $G=\on{PGL}_n$ or $G=\Res_{E/F}\on{PGL}_n$, where $E/F$ is an unramified extension. The case when $G=\on{PGL}_n$ was treated in \cite{V} (in equal  characteristic but the same arguments apply to mixed characteristic as well). The case when $G=\Res_{E/F}\on{PGL}_n$ was treated in \cite{Ham2} in the equal  characteristic situation and then was adapted in \cite{Ham1} to deal with the corresponding Rapoport-Zink spaces. Our innovation here is a reduction of the $\Res_{E/F}\on{PGL}_n$ case to the $\on{PGL}_n$ case so one can invoke the results of \cite{V}  directly (see Proposition \ref{unramified to split}). It in particular gives a much shorter proof of the main result of \cite{Ham2} (assuming \cite{V}). We note that this reduction step uses the representability of $X_\mu(b)$ for non-minuscule $\mu$ (even we are just interested in proving Theorem \ref{intro:dim} for minuscule $\mu$) and also the semismallness of convolution maps of affine Grassmannians.

\subsection{Plan of the paper} We quickly discuss the organization of this paper. In \S~\ref{aff Grass}, we prove the representability of affine Grassmannians and affine flag varieties and discuss their first properties, first for $\GL_n$ in \S~\ref{aff Grass for GLn} and \S~\ref{Dem Res}, and then for general groups in \S~\ref{gen aff flag}. We establish the geometric Satake equivalence in \S~\ref{GS}. In particular, the semi-infinite geometry of the affine Grassmannian is discussed in \S~\ref{semiinfinite geometry} and the Tannakian structure on the category is established in \S~\ref{mon str of H} and \S~\ref{comm constr}. In \S~\ref{dim of RZ}, we prove the dimension formula for Rapoport-Zink spaces conjectured by Rapoport, as an application of our theory. The paper contains two appendices. Appendix \ref{generality of perf} discusses perfect schemes and perfect algebraic spaces in some generality, which is the setting we will work with in the paper. Appendix \ref{rep as sch} discusses some further questions on affine Grassmannians, including conjectures related to the representability of affine Grassmannians as \emph{schemes} and the deperfection of ``Schubert varieties" inside them. Finally, we discuss an example of these ``Schubert varieties" in our setting.
 
\subsection{Notations}\label{not}
We fix  a perfect field $k$ of characteristic $p>0$.
For a $k$-algebra $R$, let 
$$W(R)=\{ (a_0,a_1,\ldots)\mid a_i \in R\}$$ 
denote its ring of Witt vectors, and let $R\to W(R),\ x\mapsto [x]=(x,0,0,\ldots)$ be the Teichm\"{u}ller lifting.  We denote by $W_h(R)$ the ring of truncated Witt vectors of length $h$. If $R$ is perfect, $W_h(R)=W(R)/p^{h}W(R)$.

Let us write $\mO_0=W(k)$, $F_0=W(k)[1/p]$. Except in \S~\ref{ADLV==RZ}, $F$ denotes a totally ramified finite extension of $F_0$ and $\mO$ denotes the ring of integers of $F$. Let $\varpi$ denote a uniformizer of $\mO$. For a $k$-algebra $R$, let $W_{\mO}(R)=W(R)\otimes_{W(k)}\mO$ and $W_{\mO,n}(R)=W(R)\otimes_{W(k)}\mO/\varpi^n$. We write $W_{\mO}(\bar{k})=\mO_L$ and $L=\mO_L[1/p]$, which is the completion of the maximal unramified extension of $F$. We write
\[D_{F,R}:=\Spec W_\mO(R),\quad D_{F,R}^*:=\Spec W_\mO(R)[1/p].\]
Informally, $D_{F,R}$ (resp. $D_{F,R}^*$) can be thought as a family of discs (resp. punctured discs) parameterized by $\Spec R$. In the above notations, we will omit the subscript $F$ if $F=F_0$, and will omit the subscript $R$ if $R=k$.

\medskip

Unless otherwise stated, we will assume that $G$ is a smooth affine group scheme over $\mO$ with connected geometric fibers. We will denote by $\mE_0$ the trivial $G$-torsor. 

When $G$ is a split reductive group,  we will choose a Borel subgroup $B\subset G$ over $\mO$ and a split maximal torus $T\subset B$. Let $U\subset B$ denote the unipotent radical of $B$.
Sometimes, we denote by $\bar G,\bar B, \bar U, \bar T$ their reductions modulo $\varpi$.

Let $\xcoch$ denote the coweight lattice of $T$ and $\xch$ the weight lattice. Let $\xcoch^+$ denote the semi-group of dominant coweights with respect to the chosen $B$.  We denote by $2\rho\in\xch$ the sum of positive roots. Let $``\leq"$ be the following partial order on $\xcoch$: $\la\leq \mu$ if $\mu-\la$ is a linear combination of positive roots with coefficients in $\bZ_{\geq 0}$. For $\la\in\xcoch$, the image of $\varpi\in F^\times=\bG_m(F)$ under the map $\la :\bG_m\to T\to G$ is denoted by $\varpi^\la$.

The dual group of $G$ (over a field of characteristic zero) is denoted by $\hat{G}$. We equip it with a dual Borel $\hat B$ and a maximal torus $\hat T$ dual to $T$\footnote{In fact, if $\hat G$ is the Tannakian group constructed from the geometric Satake, it is automatically equipped with $\hat B$ and $\hat T$.}. For $\mu\in \xcoch^+$, let $V_\mu$ denote the irreducible representation of $\hat{G}$ with the highest weight $\mu$. For $\la\in\xcoch$, let $V_\mu(\la)$ denote the $\la$-weight subspace of $V_\mu$. 

\subsection{Acknowledgement}
An ongoing project with L. Xiao is the main motivation of the current paper. The author thanks him cordially for the collaboration. The author also thanks B. Bhatt, B. Conrad, V. Drinfeld, B. Elias, A. de Jong, X. He, J. Kamnitzer, L. Moret-Bailly, G. Pappas, P. Scholze and Z. Yun  for useful discussions and comments. He in particular would like to thank J. Kamnitzer for pointing out a serious mistake in an early draft of the paper. The author is partially supported by NSF grant DMS-1001280/1313894 and DMS-1303296/1535464 and the AMS Centennial Fellowship.

\section{Affine Grassmannians}\label{aff Grass}
In this section, we construct affine Grassmannians and affine flag varieties in mixed characteristic. We will work with a class of geometric objects called perfect algebraic spaces. We refer to Appendix \ref{generality of perf} for the necessary background for these objects.

\subsection{$p$-adic loop groups and affine Grassmannians}
In this subsection, we will define affine Grassmannians and state our first main theorem. We refer to \S~\ref{not} for the notations.
\subsubsection{} Let $\mX$ be a finite type $\mO$-scheme. According to Greenberg (\cite{Grb}), there are the following two presheaves on the category of affine $k$-schemes defined as
\[L_p^+\mX(R)=\mX(W_{\mO}(R)),\quad L_p^h\mX(R)=\mX(W_{\mO,h}(R)),\]
which are represented by schemes over $k$. In addition, $L_p^h\mX$ is of finite type over $k$, and $L_p^+\mX=\underleftarrow\lim L_p^h\mX$. If $\mX\subset \mY$ is open, then $L_p^+\mX\subset L_p^+\mY$ is open.
We denote their perfection by
\[L^{+}\mX= (L_p^+\mX)^\pf,\quad L^{h}\mX=(L_p^h\mX)^\pf,\]
and call them $p$-adic jet spaces.
The justification of the choice of the notations is that perfect objects  behave better and are more similar to their equal characteristic analogues. If $f:\mX\to\mY$ is an $\mO$-morphism, we denote by $L_p^+f:L^+_p\mX\to L^+_p\mY$ and $L^+f: L^+\mX\to L^+\mY$ the induced maps.

Let $X$ be an affine scheme over $F$. We define the $p$-adic loop space  $LX$ of $X$ as a perfect space by assigning a perfect $k$-algebra $R$ the set
\[L X(R)=X(W_{\mO}(R)[1/p]).\]
Every $F$-morphism $f:X\to Y$ induces a morphism $Lf:LX\to LY$. We do not define the object $L_pX$ as a presheaf on the category of all affine $k$-schemes. According to Lemma \ref{descent}, the following statement makes sense.

\begin{prop}\label{p-adic loop}
If $X$ is affine of finite type, then $LX$ is represented by an ind perfect schemes.
\end{prop}
\begin{proof}As soon as we go to the perfection, the proof is similar to the representability of the usual loop groups in the equal characteristic setting.

First, it is enough to consider the case $F=F_0=W(k)[1/p]$. If $X=\bA^1$, then $LX=\underrightarrow\lim (\bA^{\infty})^{\pf}$. This follows from the fact that every element in $W(R)[1/p]$ can be uniquely written as
\[x=\sum_{i\geq -N}p^i[x_i].\]
Second, $X=X_1\times X_2$, then $L(X_1\times X_2)=LX_1\times LX_2$ so $L\bA^n$ is representable. Finally, if $Z\subset \bA^n$ is a closed embedding, then $LZ\to L\bA^n$ is a closed embedding. Indeed we can write
\[[x]+[y]=\sum p^j[\Sigma_j(x,y)^{1/p^j}],\]
where $\Sigma_j(X,Y)$s are certain polynomials with two variables $X,Y$, of homogeneous degree $p^j$.
Now assume that
\[\mO_Z= F[t_1,\ldots,t_n]/(f_1,\ldots,f_m).\]
It is easy to see that if $f(t_1,\ldots,t_n)$ is a polynomial with coefficients in $F$, then
\[f(\sum p^i[x_{1i}],\ldots, \sum p^i[x_{ni}]) = \sum p^j [f^{(j)}(x_{ml})^{1/p^j}].\]
where $f^{(j)}$s are some polynomials in terms of the variables $X_{ml}, m=1,\ldots,n,l\in\bZ$. Then $L_pZ$ is defined in $L_p\bA^n$ by the equations $f_{r}^{(j)}, (f_{r}^{(j)})^{1/p},(f_{r}^{(j)})^{1/p^2},\ldots$.
\end{proof}

The above arguments also give the following lemma.
\begin{lem}(i) Let $\mX$ be an affine scheme of finite type over $\mO$ and let $X=\mX\otimes_\mO F$. Then $L^+\mX\subset LX$ is a closed subscheme.

(ii) If $X\to Y$ is a closed embedding, then $LX\to LY$ is a closed embedding.
\end{lem}

Now, let $\mX=G$ be a smooth affine group scheme over $\mO$. We write $G^{(0)}=G$ and define the $h$th congruence group scheme of $G$ over $\mO$ , denoted by $G^{(h)}$, as the dilatation of $G^{(h-1)}$ along the unit. There is a natural map $G^{(h)}\to G$ which identifies
\begin{equation}\label{congruent}
L^+G^{(h)}=\ker (L^+G\to L^hG). 
\end{equation}
Note, however, that $L_p^+G^{(h)}\neq \ker(L_p^+G\to L_p^hG)$.

\subsubsection{}
Let $G$ be a smooth affine group scheme over $\mO$. We define the affine Grassmannian of $G$ as the perfect space
\[\Gr_G:= [LG/L^+G].\]
See \S~\ref{nonsense pf sp} for the notation. Explicitly, for a perfect $k$-algebra $R$, $\Gr_G(R)$ is the set of pairs $(P,\phi)$, where $P$ is an $L^+G$-torsor on $\Spec R$ and $\phi: P\to LG$ is an $L^+G$-equivariant morphism.
Similar to the equal characteristic situation, there is the following interpretation of $\Gr_G$. Recall that we denote by $\mE_0$ the trivial $G$-torsor on $\mO$.
\begin{lem}\label{moduli interpretation of Gr}
We have
\[\Gr_G(R)=\left\{(\mE,\beta)\ \left|\ \begin{split}&\mE \mbox{ is a } G\mbox{-torsor on } D_{F,R} , \mbox{ and }\\
& \beta:\mE|_{D^*_{F,R}}\simeq  \mE_0|_{D^*_{F,R}}\mbox{ is a trivialization }\end{split}\right.\right\}.\]
\end{lem}
\begin{proof}Let us temporarily denote the functor defined by the right hand side by $\Gr'_G$. We define a new presheaf $L'G$ as
\[L'G(R)\simeq\left\{(\mE,\beta, \epsilon)\ \left|\ \begin{split}&(\mE,\beta)\in \Gr'_G(R) \\
& \epsilon:\mE_0\simeq \mE \mbox{ is a trivialization}\end{split}\right.\right\}.\]
We claim that: (a) $L'G$ is an $L^+G$-torsor over $\Gr'_G$; and (b) there is an $L^+G$-equivariant isomorphism $LG=L'G$. Then it follows that $\Gr_G=\Gr'_G$.

Let $\mE$ be a $G$-torsor on $D_{F,R}$. Since $G$ is smooth, after replacing $R$ by its an \'etale cover, we may assume that $\mE$ is trivial when restricted to $\Spec W_\mO(R)/\varpi$. Then it is trivial on $D_{F,R}$, again by the smoothness of $G$. In other words, \'etale locally on $R$ a trivialization $\epsilon$ as in the definition of $L'G$ always exists. Claim (a) follows. The isomorphism in Claim (b) is given by $A\mapsto (\mE_0,A,\on{id})$ with the inverse map given by $(\mE,\beta,\epsilon)\mapsto A:=\beta\epsilon$.
\end{proof}

According to Lemma \ref{descent}, one can ask whether $\Gr_G$ is represented by a(n ind)-perfect scheme.
Our main theorem of this section gives a positive answer to a slightly weaker version of this question.
\begin{thm}\label{aff: main}
The affine Grassmannian $\Gr_G$ is represented by a separated ind-pfp ind-perfect algebraic space. If $G$ is reductive over $\mO$, then $\Gr_G$ is ind-perfectly proper.
\end{thm}

Again, similar to the equal characteristic situation, one can reduce the proof of this theorem to the case $G=\GL_n$ and $F=F_0$ (see Proposition \ref{gen G}).
So in the next two subsections, we will focus on the $\GL_n$ case first. 

\subsection{The affine Grassmannian for $\GL_n$}\label{aff Grass for GLn}
We denote $\Gr_{\GL_n}$ by $\Gr$ in this subsection for simplicity. We will introduce some closed subspace $\bGr_N\subset \Gr$, which can be realized as a quotient of a finite dimensional affine scheme by an action of a finite dimensional affine group scheme. It then follows from Theorem \ref{quotient} that $\bGr_N$ (and therefore $\Gr$) is representable.
\subsubsection{}\label{rel pos latt}
Let $R$ be a perfect $k$-algebra. As usual we will identify $\GL_n$-torsors on $D_{R}=\Spec W(R)$ with finite projective $W(R)$-modules. 
So we can rewrite the moduli problem $\Gr$ as follows. Let $\mE_0=W(R)^n$ denote the rank $n$ free $W(R)$-module. Then for a perfect $k$-algebra $R$, we have
\[\Gr(R)=\left\{(\mE,\beta)\ \left|\ \begin{split}&\mE \mbox{ is a rank } n \mbox{ projective } W(R)\mbox{-module} , \\
& \beta:\mE[1/p]\to  \mE_0[1/p]\mbox{ is an isomorphism }\end{split}\right.\right\}.\]
Note that via the inclusion $\mE\subset \mE[1/p]\stackrel{\beta}{\simeq} \mE_0[1/p]=W(R)[1/p]^n$, we can think of $\mE$ as a lattice in $W(R)[1/p]^n$. Therefore the above moduli interpretation coincides with the one given by \eqref{intro:Gr}. We will use these two points of views interchangeably.

Recall that a finite projective $W(R)$-module is the same as a finite rank locally free crystal on $\Spec R$.  Due to this reason, for two finite projective $W(R)$-modules $\mE_1$ and $\mE_2$,
an isomorphism $\beta: \mE_1[1/p]\simeq \mE_2[1/p]$ will be called a quasi-isogeny. Sometimes, we write it as
\begin{equation*}\label{qiso}
\beta: \mE_1\dasharrow \mE_2
\end{equation*}
for simplicity. If $\beta$ extends to a genuine map $\mE_1\to \mE_2$, it is called an isogeny. 
Now we recall some basic facts of quasi-isogenies.

Let $k$ be a perfect field, and let $\beta: E_1\dasharrow E_2$ be a quasi-isogeny of finite projective $W(k)$-modules\footnote{Finite projective modules over $W(k)$ are usually denoted by $E_0,E_1,\ldots$ instead of $\mE_0,\mE_1,\ldots$ in the sequel.}. The relative position of $\beta$, denoted by $\inv(\beta)$,  is defined as an element in 
$$\xcoch(D_n)^+=\{\mu=(m_1,\ldots,m_n)\in \bZ^n\mid m_1\geq m_2\geq\cdots\geq m_n\}$$
as follows.  There always exist a basis $(e_1,\ldots,e_n)$ of $E_1$ and a basis $(f_1,\ldots,f_n)$ of $E_2$ such that $\beta$ is given by 
$$\beta(e_i)=p^{m_i}f_i$$ 
and $m_1\geq m_2\geq\cdots\geq m_n$. In addition, this sequence $(m_1,\ldots,m_n)$ is independent of the choice of the basis. Then we define
\begin{equation}\label{rel pos}
\inv(\beta)=(m_1,\ldots,m_n).
\end{equation}
Note that $\beta$ is an isogeny if and only if $m_n\geq 0$\footnote{In this case, the relative position of $\beta$ is sometimes also called the Hodge polygon of $\beta$ and denoted by $HP(\beta)$ in literature.}. 

For $0\leq i\leq n$, we denote by $\omega_i=(1,\ldots,1,0,\ldots,0)$ with the first $i$ entries $1$ and the last $n-i$ entries $0$. Let $\omega_i^*=\omega_{n-i}-\omega_{n}$. Note that  $\inv(\beta)=\omega_i$ if and only if $\beta$ extends to a genuine map $E_1\to E_2$  such that $pE_2\subset E_1$ and $E_2/E_1$ is a $k$-vector space of dimension $i$. Similarly, $\inv(\beta)=\omega_i^*$ if and only if $\beta^{-1}$ induces the inclusions $pE_1\subset E_2\subset E_1$ such that $E_1/E_2$ is of dimension $i$.

Note that $\xcoch(D_n)^+$ can be identified with the set of dominant coweights of $\GL_n$ in the usual way. The partial order ``$\leq$" from \S\ \ref{not} then can be explicitly described as follows: $\mu_1=(m_1,\ldots,m_n)\leq \mu_2=(l_1,\ldots,l_n)$ if 
$$m_1+\cdots+m_j\leq l_1+\cdots+l_j, \ \ \ j=1,\ldots, n, \quad \mbox{ and }\ \ m_1+\cdots+m_n=l_1+\cdots+l_n.$$ 
Note that $\omega_0$ is a minimal element.

Now let $R$ be a perfect $k$-algebra, and let $\beta: \mE_1\dasharrow \mE_2$ be a quasi-isogeny of finite projective $W(R)$-modules. For $x\in \Spec R$, we denote by
$$\beta_x: \mE_1\otimes_{W(R)}W(k(x))[1/p]\to \mE_2\otimes_{W(R)}W(k(x))[1/p]$$
the base change of $\beta$ to $x$. Let
\[(\Spec R)_{\mu}=\{x\in \Spec R \mid \inv(\beta_x)=\mu\}\subset (\Spec R)_{\leq \mu}=\{x\in \Spec R\mid \inv(\beta_x)\leq \mu\}.\]
If $(\Spec R)_{\mu}=\Spec R$, we say that the quasi-isogeny $\beta$ is of relative position $\mu$. 

\begin{lem}\label{minu}
Let $R$ be perfect $k$-algebra and
let $\beta: \mE_1\dasharrow \mE_2$ be a quasi-isogeny as above. If it is of relative position $\omega_i$, then $\beta$ induces a chain of inclusions
$p\mE_2\subset \mE_1\subset \mE_2$ and the quotient $\mE_2/\mE_1$ is a finite projective $W(R)/p=R$ module of rank $i$. A similar statement holds for $\omega_i^*$.
\end{lem}
\begin{proof}We claim that if $\beta_x$ is a genuine map at every point $x\in \Spec R$, then $\beta$ is a genuine map. Indeed, there is an open cover $\Spec W(R)=\cup \Spec W(R)_{f_i}$ such that both $\mE_1$ and $\mE_2$ are free so that we can represent $\beta$ by an element in $A_i\in M_{n\times n}(W(R)_{f_i}[1/p])$. We need to show that $A_i\in M_{n\times n}(W(R)_{f_i})$. Let $\bar f_i=f_i \mod p$. Then there is a natural map $j:W(R)_{f_i}\to W(R_{\bar f_i})$ and it is enough to show that $j(A_i)\in M_{n\times n}(W(R_{\bar f_i}))$. But this can be checked at every point of $\Spec R_{\bar f_i}$. This proves that $\beta$ induces an inclusion $\mE_1\subset \mE_2$. By applying the same argument to the quasi-isogeny 
$$\frac{1}{p}\beta^{-1}: p\mE_2\dasharrow \mE_1,$$ 
one shows that $p\mE_2\subset \mE_1$.

To show that $\mE_2/\mE_1$ is locally free, first note that for every homomorphism $R\to R'$ of perfect $k$-algebras, there is a natural isomorphism
\begin{equation}\label{obs}
(\mE_2/\mE_1)\otimes_RR'\simeq (\mE_2\otimes_{W(R)}W(R'))/(\mE_1\otimes_{W(R)}W(R')).
\end{equation}
So we can assume that both $\mE_1,\mE_2$ are free, and the dimension of the fibers of $\mE_2/\mE_1$ is constant on $\Spec R$. Note that $\mE_1/\mE_2=\on{coker}(\mE_1/p\to \mE_2/p)$. So we reduce to show that on a reduced affine scheme $\Spec R$, if $N$ is finitely presented, and the fiber dimension of $N$ is constant, then $N$ is locally free. But as $N$ is finitely presented, it is isomorphic to an $R$-module of the form $(\on{coker}(A^m\to A^n))\otimes_AR$, where $A\subset R$ is a subring, of finite type over $\bZ$. Then we reduce to the noetherian situation, in which case the statement is well-known.
\end{proof}

Recall the following basic fact (\cite[\S 2.3]{Ka}).
\begin{lem}\label{Hodge stra}
$(\Spec R)_{\leq \mu}$ is closed in $\Spec R$, and $(\Spec R)_\mu$ is open in $(\Spec R)_{\leq \mu}$. In particular, $(\Spec R)_{\omega_0}$ is closed.
\end{lem}

Here is a direct corollary of this lemma. See \S~\ref{nonsense pf sp} for the definition of closed embedding between two perfect spaces.

\begin{cor}\label{diagonal}
The diagonal map $\Delta:\Gr\to \Gr\times \Gr$ is a closed embedding.
\end{cor}
\begin{proof}
Let $\Spec R\to \Gr\times \Gr$ be a map, given by two pairs $(\mE,\beta)$ and $(\mE',\beta')$. We consider the quasi-isogeny $(\beta')^{-1}\beta:\mE\dasharrow \mE'$. Then $\Spec R \times_{\Gr\times\Gr, \Delta}\Gr$ is represented by $(\Spec R)_{\omega_0}$.
\end{proof}

For every $\mu\in\xcoch(D_n)^+$, let
\[\Gr_{\leq \mu}(R)=\{ (\mE,\beta)\in \Gr(R)\mid (\Spec R)_{\leq \mu}=\Spec R\}.\]
If $\la\leq \mu$, we have the closed embedding $\Gr_{\leq \la}\subset \Gr_{\leq \mu}$ by Lemma \ref{Hodge stra}. Define
$$\Gr_\mu=\Gr_{\leq \mu}-\cup_{\la<\mu}\Gr_{\leq \la},$$ 
which is an open subspace of $\Gr_{\leq \mu}$.

We record the following fact for later use.
For $\mu=(m_1,\ldots,m_n) \in\xcoch(D_n)^+$, let
\begin{equation}\label{pmu}
\La_\mu=W(k)\{p^{m_1}e_1,\ldots,p^{m_n}e_n\}\subset W(k)[1/p]^n
\end{equation} be the lattice generated by $\{p^{m_1}e_1,\ldots,p^{m_n}e_n\}$. Then $\La_\mu$ defines a $k$-point $\Gr$, denoted by $p^\mu$. 
The following lemma is a reformulation of the Cartan decomposition.
\begin{lem}\label{decomp}
(i) The group $\GL_n(W(k))$ acts transitively on the set $\Gr_\mu(k)$. In fact, $\Gr_\mu(k)=\GL_n(W(k)) p^\mu$. 
(ii) $\Gr(k)=\sqcup_{\mu\in \xcoch(D_n)^+}\Gr_{\mu}(k)$.
\end{lem}

The above discussions can be generalized to general split reductive groups over $\mO$. See \S~\ref{Sch in Gr}.

\subsubsection{}\label{representability}
We write $\bGr_N$ instead of $\Gr_{\leq N\omega_1}$, and $\Gr_N$ instead of $\Gr_{N\omega_1}$.
Note that the group $L\GL_n$ acts on $\Gr$, and every $\Gr_{\leq \la}$ is contained in $g\bGr_N$ as a closed subspace for some $g\in \GL_n(F)$ and some $N$ big enough.  Therefore, it enough to prove the representability of $\bGr_N$.
We make the moduli interpretation of $\bGr_N$ more explicit. Namely, for a perfect $k$-algebra $R$,
\[\bGr_N(R)=\left\{(\mE,\beta)\ \left|\ \begin{split}&\mE \mbox{ is a rank } n \mbox{ projective } W(R)\mbox{-module}, \\ 
& \mbox{and }\beta:\mE\to  \mE_0 \mbox{ is an isogeny, which}\\
& \mbox{induces }\wedge^n\beta:\wedge^n\mE\simeq p^NW(R)\subset \wedge^n\mE_0\end{split}\right.\right\}.\]

Let $M_n$ denote the scheme of $n\times n$ matrices.  Define the following morphisms of $\bZ_p$-schemes
$$\pi: M:=M_n\times\bG_m\to \bA^1,\quad \pi(A,t)=t\det A,\quad i_N:\Spec\bZ_p\to \bA^1=\Spec\bZ_p[u],\ u\mapsto p^N,$$ and define a scheme of finite type over $\bZ_p$ as
$$V_N=\bZ_p\times_{i_N,\bA^1,\pi}M.$$
By definition, $L_p^+V_N(R)$ is the set of pairs $(A,t)$ consisting of an $n\times n$-matrix $A$ with entries in $W(R)$ and $t\in W(R)^\times$ such that $t\det A=p^N$. Note that $L_p^+\GL_n$ acts on $L_p^+V_{N}$ by left and right multiplications. Passing to the perfection, both actions become free. By the same proof of Lemma \ref{moduli interpretation of Gr}, we obtain the following statement.

\begin{lem}There is a canonical isomorphism
\[L^+V_N/L^+\GL_n=\bGr_N.\]
\end{lem}
This lemma expresses $\bGr_N$ as a quotient of an affine scheme by an affine group scheme. But it is not very useful since both $L^+V_N$ and $L^+\GL_n$ are infinite dimensional. We need to work at the finite level.

Recall that we have the affine group scheme $L^+\GL_n^{(h)}=\ker (L^+\GL_n\to L^h\GL_n)$. Define
\[\bGr_{N,h}=L^+V_N/L^+\GL_n^{(h)}.\]
In terms of the moduli interpretation,
\[\bGr_{N,h}(R)=\left\{(\mE,\beta, \bar\epsilon)\ \left|\ \begin{split}&(\mE,\beta)\in \bGr_N(R) \\
& \bar\epsilon:\mE_0|_{W_{h}(R)}\simeq \mE|_{W_{h}(R)} \end{split}\right.\right\}.\]
This is an $L^h\GL_n$-torsor over $\bGr_N$ on which $L^h\GL_n$ acts by changing the isomorphism $\bar\epsilon$. Our main observation is that $\bGr_{N,h}$ is already represented by an affine scheme when $h$ is large. To prove this, we need to introduce certain affine schemes defined by matrix equations.

We assume that $h>N$. Via the Greenberg realization, the determinant map $\det: M_n\to \bA^1$ induces a morphism
$$(\on{det}_0,\ldots,\on{det}_{h-1}): L^h_p M_n=\bA^{n^2h}\to L^h_p\bA^1=\bA^h.$$
Define
\begin{equation}\label{notation}
V'_{N,h}:=\{A \in L^h_pM_n\mid \det{\! _0}A=\cdots=\det{\! _{N-1}}A=0, \det{\! _N}A\in \bG_m\}, \quad V_{N,h}:= (V'_{N,h})^\pf.
\end{equation}
Note that there is an $L^h_p\GL_n\times L^h_p\GL_n$-action on $V'_{N,h}$ by left and right multiplications. Passing to the perfection, we obtain an action of $L^h\GL_n\times L^h\GL_n$ on $V_{N,h}$. Let $J$ be the stabilizer group scheme over $V_{N,h}$ with respect to the \emph{right} multiplication by $L^h\GL_n$, i.e. $J$ is defined by the Cartesian product
\begin{equation}\label{stab}
\begin{CD}
J@>>>  V_{N,h}\times L^h\GL_n\\
@VVV@VVV\\
V_{N,h}@>\Delta >>V_{N,h}\times V_{N,h}.
\end{CD}\end{equation}
Or explicitly
\[J=\{(A,\ga)\in V_{N,h}\times L^h\GL_n\mid A\ga=A\}.\]
Likewise, let $J'$ denote the stabilizer group scheme over $V'_{N,h}$ with respect to the right multiplication by $L^h_p\GL_n$. Then $J'$ is an affine scheme of finite type over $k$, which is a deperfection of $J$.

There is a natural map
\[\bGr_{N,h}\to V_{N,h},\]
given by $(\mE,\beta,\bar\epsilon)\mapsto (\beta|_{W_{h}(R)})\bar\epsilon$. 

The key lemma is the following. 
\begin{lem}\label{key}
Assume that $h> N$. There is an isomorphism
\[J\simeq \bGr_{N,h}.\]
In particular, $\bGr_{N,h}$ is represented by a perfect affine scheme, perfectly of finite type.
\end{lem}
\begin{proof}
Recall that $J=(J')^\pf$. Therefore the second statement follows from the first, which we now prove.

Let $R$ be a perfect $k$-algebra, and let $A\in V_{N,h}(R)$. Then $J_R$ classifies those $\gamma\in L^h\GL_n(R)$ that make the following diagram commutative
\[\begin{CD}
\mE_0|_{W_{h}}@>A>>\mE_0|_{W_{h}}\\
@V\gamma VV@| \\
\mE_0|_{W_{h}}@>A>>\mE_0|_{W_{h}}.
\end{CD}\]
On the other hand, $(\bGr_{N,h})_R$ classifies those $(\mE,\beta,\bar\epsilon)$ such that the following diagram is commutative
\[\begin{CD}
\mE @>\beta>>\mE_0\\
@V\bar\epsilon^{-1} VV@VVV \\
\mE_0|_{W_{h}}@>A>>\mE_0|_{W_{h}},
\end{CD}\]
where the notation $\bar\epsilon^{-1}$ is understood as the composition $\mE\to \mE|_{W_{h}(R)}\stackrel{\bar\epsilon^{-1}}{\to}\mE_0|_{W_h(R)}$.
Therefore, there is a natural action
\[\bGr_{N,h}\times_{V_{N,h}}J\to \bGr_{N,h},\quad ((\mE,\beta,\bar\epsilon),\gamma)\mapsto (\mE,\beta,\bar\epsilon\gamma).\]
Note that the natural map $L^+V_N\to V_{N,h}$ is surjective on $R$-points if $h>N$. Indeed, if $A\in V_{N,h}(R)$, then $\det A \in p^NW_{h}(R)^\times$. Regard $A$ as a matrix in $M_n(W_h(R))$, and let $\tilde A\in M_n(W(R))$ denote a lifting. Then $\det \tilde{A}\in p^NW(R)^\times$, and there is a unique $t\in W(R)^\times$ such that $t\det\tilde{A}=p^N$. Then $(\tilde{A},t)\in L^+V_N(R)$ is a lifting of $A$.  

As a consequence, the map $\bGr_{N,h}\to V_{N,h}$ admits a section. 
Indeed, if $(\tilde A,t)\in L^+V_N$ is a lifting of $A$, then $(\mE_0,\tilde A,\id)\in\bGr_{N,h}$.

Let us fix such a section 
$$s:V_{N,h}\to \bGr_{N,h}, \quad A\mapsto (\mE_A,\beta_A,\bar\epsilon_A).$$ It induces a map 
$$s:J\to \bGr_{N,h},\quad \ga\mapsto (\mE_A,\beta_A,\bar\epsilon_A\ga),$$ 
which is injective on $R$-points since the action of $L^h\GL_{n}$ on $\bGr_{N,h}$ is free. To show that it is also surjective on $R$-points, let $(\mE,\beta,\bar \epsilon)$ be a point of $\bGr_{N,h}$ such that $(\beta|_{W_{h}(R)})\epsilon=A$. Then there exists a unique isomorphism $\alpha:\mE_A\simeq \mE$ such that the following diagram is commutative 
\[\begin{CD}
0@>>> \mE_A@>\beta_A>>\mE_0@>>> \on{coker} A@>>> 0\\
@.@V\alpha VV@|@|@.\\
0@>>>\mE@>\beta>>\mE_0@>>>\on{coker} A@>>> 0.
\end{CD}\]
Let $\ga=\bar\epsilon_A(\alpha|_{W_{h}(R)})^{-1}\bar\epsilon^{-1}$. Then $(A,\ga)\in J$ is the preimage of $(\mE,\beta,\bar \epsilon)$ under the above map $s:J\to \bGr_{N,h}$.
Therefore the first claim of the lemma follows.
\end{proof}

\begin{rmk}\label{set sect}
The above isomorphism depends on a choice of lifting of the projection $L^+V_N\to V_{N,h}$. To fix the ambiguity, we will use the obvious lifting given by $$W_{h}(R)\to W(R),\quad (\sum_{0\leq i<h} p^i[r_i] \mod p^h)\mapsto \sum_{0\leq i<h}p^i[r_i].$$
\end{rmk}

As a corollary of the above lemma and Theorem \ref{quotient}, we have
\begin{prop}
The functor $\bGr_N$ is represented by a separated perfect algebraic space, perfectly of finite type over $k$. In particular, $\Gr$ is representable.
\end{prop}
\begin{proof}Let $G=L^h\GL_n$, which is the perfection of the smooth algebraic group $G_0=L^{h}_p\GL_n$. To apply Theorem \ref{quotient}, it remains to check that $G\times \bGr_{N,h}\to \bGr_{N,h}\times\bGr_{N,h}$ is a closed embedding. But this follows from Corollary \ref{diagonal}. 
\end{proof}

\subsection{``Demazure resolution"}\label{Dem Res}
The perfect algebraic space $\bGr_N$ is in general ``singular". 
In this subsection, we construct a morphism $\wGr_N\to\bGr_N$, which can be regarded as the ``Demazure resolution" in the current setting. Using it, we show that $\bGr_N$ is irreducible and perfectly proper. Therefore, $\Gr$ is ind-perfectly proper.

\subsubsection{}
As before, let $\mE_0=W(R)^n$ denote the rank $n$ free $W(R)$-module. 

Let $\mu_\bullet=(\mu_1,\ldots,\mu_N)$ be a sequence, where each $\mu_i\in \{\omega_1,\ldots, \omega_n,\omega_1^*,\ldots,\omega_n^*\}$. We consider the following space $\Gr_{\mu_\bullet}$ on $\on{Aff}^{\on{pf}}_k$:  for a perfect $k$-algebra $R$, $\Gr_{\mu_\bullet}(R)$ classifies  chains of quasi-isogenies 
\begin{equation}\label{chain qiso}
\mE_N\stackrel{\beta_N}{\dasharrow} \mE_{N-1}\stackrel{\beta_{N-1}}{\dasharrow}\cdots\stackrel{\beta_2}{\dasharrow} \mE_1\stackrel{\beta_1}{\dasharrow} \mE_0,
\end{equation} 
where all $\mE_i$'s are rank $n$ finite projective $W(R)$-modules, and $\mE_{i}\dasharrow \mE_{i-1}$ is of relative position $\mu_i$.

\begin{prop}\label{rep of dem res}
The space $\Gr_{\mu_\bullet}$ is represented by a perfect $k$-scheme, perfectly proper over $k$. 
\end{prop}
\begin{proof}
We will prove the proposition by induction on $N$. First, we show that $\Gr_{\omega_i}$ is represented by $\Gr^\pf(i,n)$, the perfection of the usual Grassmannian variety that classifies $i$-dimensional \emph{quotients} of $k^n$. In fact, we will prove a slightly more general statement.

We make use of the following notations. Let $X$ be a perfect $k$-scheme, and let $\mE$ be a rank $n$ locally free crystal on $X$. Let $R$ be a perfect $k$-algebra and $x\in X(R)$ an $R$-point. We denote the value of $\mE$ at the universal PD thickening $W(R)\to R$ by $x^*\mE$, which is a finite projective $W(R)$-module of rank $n$, and denote the value at the trivial thickening $R\stackrel{\id}{\to}R$ by $x^*\mE/p$. By varying $x$, these $x^*\mE/p$ glue together to form a vector bundle of rank $n$ on $X$, denoted by $\mE/p$.

\begin{lem}\label{aux}
Let $X$ be a perfect $k$-scheme and $\mE$ a locally free crystal of rank $n$ on $X$. 
Let $Y$ be the perfect space over $X$ that assigns to every $x:\Spec R\to X$ the set of isogenies $\mF_x\to x^*\mE$ of finite projective $W(R)$-modules such that $x^*\mE/\mF_x$ is a locally free $W(R)/p$-module of rank $i$. Then $Y$ is represented by the perfect scheme $\Gr^\pf(i,\mE/p)$ introduced in Corollary \ref{perf Grass}.
\end{lem}
\begin{proof}The map $Y\to \Gr^\pf(i,\mE_0/p)$ sends an $R$-point of $Y$ represented by $\mF_x\to x^*\mE$ to an $R$-point of  $\Gr^\pf(i,\mE_0/p)$ represented by $x^*\mE/p\to x^*\mE/\mF_x\to 0$. Conversely, given an $R$-point of $\Gr^\pf(i,\mE_0/p)$ represented by $x^*\mE/p\to \mQ\to 0$, where $\mQ$ is a finite projective $R$-module of rank $i$, 
we define $\mF=\ker (x^*\mE\to x^*\mE/p\to \mQ)$. We need to show that it is a finite projective $W(R)$-module of rank $n$. Then $\mF\to x^*\mE$ is an isogeny and therefore defines a point of $Y$.

 It is enough to show that $\mF/p^i\mF$ is a finite projective $W(R)/p^i$-module of rank $n$ for every $i$.
First by definition,
\[p (x^*\mE)\subset \mF\subset x^*\mE,\]
and $\mF/p(x^*\mE)$ is a direct summand of $x^*\mE/p$. Therefore, $\mF/p(x^*\mE)$ is a finite projective $R$-module. Now, by tensoring the short exact sequence
$$0\to \mF\to x^*\mE\to \mQ\to 0$$ with $-\otimes_{W(R)}R$, we obtain an exact sequence
\[0\to \on{Tor}^{W(R)}(\mQ,W(R)/p)\to \mF/p\to x^*\mE/p\to \mQ\to 0.\]
In addition, there is a canonical isomorphism
\begin{equation}\label{switch}
 \mQ=\on{Tor}^{W(R)}(\mQ,W(R)/p).
\end{equation}
Therefore $\mF/p$, which is an extension of $\mF/p(x^*\mE)$ by $\mQ$, is a finite projective $R$-module of rank $n$.
From the exact sequence $$0\to p(x^*\mE)/p\mF\to \mF/p\to \mF/p(x^*\mE)\to 0,$$ we conclude that $p(x^*\mE)/p\mF$ is a direct summand of $\mF/p$, and therefore is a finite projective $R$-module. Now by induction, we deduce that each $p^i\mF/p^{i+1}\mF$ is a finite projective $R$-module of rank $n$, and that $p^{i+1}(x^*\mE)/p^{i+1}\mF$ is a direct summand of $p^i\mF/p^{i+1}\mF$.
Finally, using the exact sequence
$$0\to p^i\mF/p^{i+1}\mF\to \mF/p^{i+1}\mF\to \mF/p^i\mF\to 0,$$ and by induction again we conclude that each $\mF/p^{i}\mF$ is a finite projective $W(R)/p^i$-module.
This finishes the proof of the lemma. 
\end{proof}

Combining with Lemma \ref{minu}, we see that $\Gr_{\omega_i}$ is represented by $\Gr^\pf(i,n)$. Now assume that $\Gr_{\mu_\bullet}$ is represented by a perfectly projective perfect $k$-scheme. Let $\mu_{N+1}$ be an additional element in $\xcoch(D_n)^+$. Let $U=\Spec R$ be an affine open of $\Gr_{\mu_\bullet}$. Then by definition, there is the tautological chain of isogenies $\mE_N\to \mE_{N-1}\to\cdots\to \mE_0$ of finite projective $W(R)$-modules over $U$, and $\mE_N/p$ is a finite projective $R$-module of rank $n$. Clearly, by varying $U$, we obtain a locally free crystal $\mE_N$ on $\Gr_{\mu_\bullet}$. By Lemma \ref{minu} and Lemma \ref{aux} again, 
\[\Gr_{\mu_\bullet,\mu_{N+1}}\simeq\left\{\begin{array}{ll} \Gr^{\pf}(i,\mE_N/p), & \mu_{N+1}=\omega_i\\  \Gr^{\pf}(i,(\mE_N/p)^*), & \mu_{N+1}=\omega_i^*.\end{array}\right.\]
By Corollary \ref{perf Grass}, $\Gr_{\mu_\bullet}$ is perfectly proper.
\end{proof}

\begin{rmk}\label{wGr2}
One can show that
$$\wGr_1=\bP^{n-1,\pf},\quad \wGr_2=\bP^\pf(\Omega_{\bP^{n-1}}\oplus\mO_{\bP^{n-1}}).$$
See \S~\ref{sample cal} for a sample calculation.
On the other hand one can define the equal characteristic Demazure variety $\wGr_N^{\flat}$ which assigns every (not necessarily perfect) $k$-algebra $R$  the set of chains 
$$\mE_N\subset \mE_{N-1}\subset\cdots\subset \mE_0=R[[t]]^n$$ 
of finite projective $R[[t]]$-modules of rank $n$ such that each $\mE_i/\mE_{i+1}$ is an invertible $R[[t]]/t$-module. Then  
\[\wGr_N=(\wGr_N^\flat)^\pf,\quad N=1,2.\]
We do \emph{not} think this is true for general $N$.

Likewise, one can define the equal  characteristic analogue $\bGr_N^\flat$ of $\bGr_N$ as the moduli space of pairs $(\mE,\beta)$ where $\mE$ is a finite projective $R[[t]]$-module of rank $n$ and $\beta:\mE\to \mE_0$ is a map of $R[[t]]$-modules such that $\wedge^n\beta$ induces $\wedge^n\mE\simeq t^NR[[t]]\subset R[[t]]$. From the example given in \S~\ref{sample cal}, when $n=2$ and $N=2$, $\bGr_2\simeq (\bGr_2^\flat)^\pf$. But we do not think this is true for general $N$. 
\end{rmk}

\begin{rmk}
There is a canonical deperfection of $\wGr_N$, which can be regarded as certain moduli space related to $p$-divisible groups. See Proposition \ref{deperf of Demazure}. 
\end{rmk}

\subsubsection{}
Let $\la_1=(m_1,\ldots,m_n)$ and $\la_2=(l_1,\ldots,l_n)$. We define their sum as
$$\la_1+\la_2=(m_1+l_1,\ldots,m_n+l_n).$$
If we identify $\xcoch(D_n)^+$ with the semi-group of dominant coweights of $\GL_n$, this coincides with the usual addition. 
Let $\mmu=(\mu_1,\ldots,\mu_N)$ be a sequence with $\mu_i\in\{\omega_1,\ldots,\omega_n,\omega_1^*,\ldots,\omega_n^*\}$ as before, and let
$|\mu_\bullet|=\sum\mu_i$. There is a natural map 
\begin{equation}\label{DR}
\pi:\Gr_{\mu_\bullet}\to \Gr_{\leq |\mu_\bullet|},
\end{equation}
which sends $(\mE_\bullet,\beta_\bullet)\in \Gr_{\mu_\bullet}$ to $(\mE_N,\beta_1\cdots\beta_N)$.

\begin{lem}\label{resol}
The morphism $\pi:\Gr_{\mu_\bullet}\to \Gr_{\leq |\mu_\bullet|}$ is representable. It is perfectly proper, and fibers are perfectly proper perfect schemes.
\end{lem}
\begin{proof}
Let $\Spec R\to \Gr_{\leq |\mu_\bullet|}$ be a morphism represented by $(\mE,\beta:\mE\dasharrow\mE_0)$. Then the fiber product $$(\Gr_{\mu_\bullet})_R= \Spec R\times_{\Gr_{\leq |\mu_\bullet|},\pi}\Gr_{\mu_\bullet}$$ classifies all possible chains of quasi-isogenies as in \eqref{chain qiso} such that $\mE_N=\mE$ and $\beta_1\cdots\beta_N=\beta$. We consider another moduli problem $X$ over $\Spec R$ which assigns every homomorphism $R\to R'$ of perfect $k$-algebras the set
\[X(R')=\{\mF_N\dashleftarrow\mF_{N-1}\dashleftarrow\cdots\dashleftarrow\mF_0=\mE\mid \inv(\mF_i\dasharrow\mF_{i+1})=\mu_i^*\}.\]
By Lemma \ref{aux}, $X$ is represented by a perfect scheme over $R$, perfectly proper over $R$. Over $X$ we consider the quasi-isogeny $\mF_N\dasharrow \mF_0=\mE\dasharrow \mE_0$. Then $(\Gr_{\mu_\bullet})_R$ is represented by $X_{\omega_0}$, which is closed in $X$ by Lemma \ref{Hodge stra}. This finishes the proof of the lemma.
\end{proof}

Now we assume $\mu_i=\omega_1$ for all $i$. We denote $\Gr_{\mu_\bullet}$ by $\wGr_N$, which classifies those chains in \eqref{chain qiso} such that all $\beta_i$s are isogenies and   all $\mE_{i-1}/\mE_i$ are invertible $W(R)/p$-modules. Then \eqref{DR} specializes to a map $\pi: \wGr_N\to \bGr_N$.

\begin{lem}\label{fiber of pi}
The restriction of the map $\pi:\wGr_N\to \bGr_N$ to $\pi^{-1}\Gr_N\to \Gr_N$ is an isomorphism. The fiber of $\pi$ over every point $x\in \bGr_N-\Gr_N$ is non-empty, geometrically connected, and has positive dimension.
\end{lem}
\begin{proof}
First, we show that $\pi: \pi^{-1}(\Gr_N)\to \Gr_N$ is an isomorphism by exhibiting an inverse morphism.
Indeed, given $(\mE,\beta)\in \Gr_N(R)$, there is a chain of finitely generated $W(R)$-modules
\[\mE=\mE_N\subset \mE_{N-1}\subset\cdots \subset \mE_0=W(R)^n,\quad \mE_i=\mE+p^i\mE_0. \]
It is enough to show that each $\mE_i$ is a projective $W(R)$-module and $\mE_i/\mE_{i+1}$ is an invertible $W(R)/p$-module. Indeed, at each point $x\in \Spec R$, the dimension of the stalk of $\mE_i/\mE_{i+1}$ is one. Then by the same argument as in the last part of the proof of Lemma \ref{minu}, $\mE_i/\mE_{i+1}$ is invertible. In addition, by the same argument as in Lemma \ref{aux}, and by induction on $i$, each $\mE_i$ is a projective $W(R)$-module. 

Next, let $K$ be a perfect field over $k$. Every isogeny $\mE\to \mE_0=W(K)^n$ of finite projective $W(K)$-modules can be factored as a sequence of maps $\mE=\mE_N\to \mE_{N-1}\to\cdots\to \mE_1\to \mE_0$ such that $\mE_i/\mE_{i+1}$ is a one-dimensional vector space over $K$. This proves that the fibers of $\pi$ are non-empty.

Next, let $x\in \bGr_N$ be a geometric point. We show that $\pi^{-1}(x)$ is connected. Let $C_1,\ldots, C_r$ denote its connected components. We factor $\wGr_N\to \bGr_N$ as 
\begin{equation}\label{factorization}
\wGr_N\stackrel{\pi_1}{\to} \bP^\pf(\mE/p)\stackrel{\pi_2}{\to} \bGr_N,
\end{equation}
where $\mE\to \mE_0$ denotes the tautological isogeny over $\bGr_{N-1}$ so $\bP^\pf(\mE/p)$ is a $\bP^{n-1,\pf}$-bundle over $\bGr_{N-1}$.
Given $(\mE_\bullet,\beta_\bullet)\in\wGr_N$, the first map forgets $\mE_{N-2},\ldots,\mE_1$, and the second map further forgets $\mE_{N-1}$.
By induction, the first map has geometrically connected fibers. This implies that  $\{\pi_1(C_i)\}$ are disjoint subsets of $\pi_2^{-1}(x)$. In addition, since by Lemma \ref{resol}, $\pi^{-1}(x)$ is proper so each $\pi_1(C_i)$ is closed in $\pi_2^{-1}(x)$. Therefore, to show that $\pi^{-1}(x)$ is connected, it is enough to show that $\pi_2^{-1}(x)$ is connected. Let $K$ be the residue field of $x$.
Let us regard $K$-points of $\Gr$ as lattices $W(K)[1/p]^n$ and switch the notation to represent $x$ by a lattice $\La$. Then the fiber of $\pi_2$ over this point is given by $\bP^\pf((p^{-1}\La\cap \La_0)/\La)$ (recall that $\La_0=W(K)^n$ denotes the standard lattice), which is the perfection of a projective space and therefore is connected.

Finally, we show that the fiber over every point in $\bGr_N-\Gr_N$ has positive dimension. 
First note that $\wGr_N\to \bGr_N$ is $L^+\GL_n$-equivariant, where $L^+\GL_n$ acts on both spaces via automorphisms of $\mE_0$. Let $p^\la\in \Gr(k)$ be the point defined by the lattice $\La_\la$ as in \eqref{pmu}. Then
by Lemma \ref{decomp}, it is enough to show that for $\la<N\omega_1$, the fiber over $p^\la\in \bGr_N(k)$ has positive dimension. 
If $\la<N\omega_1$, there exists some $i$ such that $$\dim_k (\La_\la\cap p^i\La_0/ \La_\la \cap p^{i+1}\La_0)>1.$$ Therefore the fiber $\pi^{-1}(p^\la)$ contains at least a $\bP^{1,\pf}$, corresponding to possible choices of a line in $(\La_\la\cap p^i\La_0/ \La_\la\cap p^{i+1}\La_0)$.
\end{proof}

We have the following consequence.

\begin{cor}\label{irr and proper}
The separated pfp perfect algebraic space $\bGr_N$ is irreducible and perfectly proper. In particular, $\Gr=\Gr_{\GL_n}$ is ind-perfectly proper.
\end{cor}

\subsection{Affine Grassmannians and affine flag varieties}\label{gen aff flag}
\subsubsection{}
Once the representability of $\Gr_{\GL_n}$ is established, it is not hard to show that the affine Grassmannian $\Gr_{G}=LG/L^+G$ for a general smooth affine group scheme $G$ over $\mO$  is representable.

\begin{prop}\label{gen G}
Let $\rho:G\to \GL_n$ be a linear representation such that $\GL_n/G$ is quasi-affine, then $\Gr_{G}\to \Gr_{\GL_n}$ is a locally closed embedding. In addition, if $\GL_n/G$ is affine, this is a closed embedding.
\end{prop}
\begin{proof}The proof as in \cite[Theorem 4.5.1]{BD} or \cite[Theorem 1.4]{PR} extends \emph{verbatim} to the present situation.
\end{proof}

For a smooth affine group scheme $G$ over a Dedekind domain, it is known that there exists a linear representation $\rho: G\to \GL_n$ such that $\GL_n/G$ is quasi-affine (cf. \cite[\S 1.b]{PR}). In addition, if $G$ is reductive, one can choose $\rho$ such that $\GL_n/G$ is affine (cf. \cite[Corollary 9.7.7]{Ap}). Therefore, it follows from the representability of $\Gr_{\GL_n}$ and the above proposition that Theorem \ref{aff: main} holds for group schemes over $\mO_0=W(k)$. To finish the proof for the general case, it is enough to note that if $\mO$ is a totally ramified extension of $\mO_0$ and $G$ is an affine group scheme over $\mO$, then the affine Grassmannian $\Gr_G$ of $G$ is isomorphic to the affine Grassmannian $\Gr_{\Res_{\mO/\mO_0}G}$ of the Weil restriction $\Res_{\mO/\mO_0}G$ (which is a group scheme over $\mO_0$). 

\subsubsection{}\label{Schubert geom}
Now we study affine Grassmannians for an important  a class of group schemes over $\mO$, namely parahoric group schemes in the sense of  Bruhat-Tits. Following the standard terminology in literature, we call the affine Grassmannian of a parahoric group scheme a (partial) affine flag variety. 
As the theory is completely parallel to the equal characteristic situation (after passing to the perfection), we will be sketchy here and refer to \cite{PR} for details. 
We will assume that $k$ is algebraically closed. 

We temporarily use notations different from \S\ \ref{not}. Namely, we start with a connected reductive group  over $F$, denoted by $G$. 
Let $B(G,F)$ denote the Bruhat-Tits building of $G$. We fix an apartment $A(G,F)\subset B(G,F)$ 
and an alcove $\mathbf{a}\subset A(G,F)$. 
They determine a maximal split torus $A\subset G$ and an Iwahori group scheme $\mG_{\mathbf{a}}$ of $G$ over $\mO$. Let $T=Z_G(A)$ be the centralizer of $A$ in $G$, which is a maximal torus of $G$. Its connected N\'eron model, denoted by $\mT$, is a closed subgroup scheme of $\mG_{\mathbf{a}}$.
Let $\widetilde W$ denote the Iwahori-Weyl group, which is the quotient of the normalizer $N(F)$ of $T(F)$ by $\mT(\mO)$,
and let $W_a\subset \widetilde W$ denote the affine Weyl group. Let $\{s_i, i\in \bS\}$ denote the set of simple reflections, corresponding to the codimension one walls $\mathbf{a}_i$ of  the closure $\bar{\mathbf{a}}$ of $\mathbf{a}$ in $A(G,F)$, and let $``\leq"$ denote the Bruhat order on $\widetilde W$. We refer to \cite{PR} (and in particular \cite{HR}) for detailed discussions of the above notions.

For $i\in\bS$, let $\mG_i$ denote the corresponding parahoric group scheme. There is a natural map $\mG_{\mathbf{a}}\to \mG_i$.
Let $I=L^+\mG_{\mathbf{a}}$ and $P_i=L^+\mG_i$.  Let us write $\mathcal F\ell=LG/I$, and call it the affine flag variety of $G$. By Theorem \ref{aff: main}, it is representable. For $w\in \widetilde W$, let $S_w$ denote the closure of the $I$-orbit through $\dot{w}$, where $\dot{w}$ is a lifting of $w$ to $G(F)$. This is the ``Schubert variety", which in the current setting is a separated pfp perfect algebraic space. As in the equal characteristic situation, 
\[S_w=\bigsqcup_{v\leq w} I\dot{v} I/I\]
is a decomposition of $S_w$ into locally closed subsets and each $I\dot{v}I/I$ is isomorphic to the perfection of an affine space of dimension $\ell(v)$.
We show that $S_w$ is perfectly proper so that  $\mF\ell=\underrightarrow\lim S_w$ is ind-perfectly proper. The idea is similar to the proof of Corollary \ref{irr and proper}.

Note that $I$ is a subgroup of $P_i$ (however, $L^+_p\mG_C\to L^+_p\mG_i$ is \emph{not} a closed embedding). It is easy to see that $P_i/I\simeq \bP^{1,\pf}$.
 Then to any sequence $\tilde w=(s_{j_1},\ldots,s_{j_m}), j_1,\ldots, j_m\in \bS$ (sometimes called a word), one can associate the ``Demazure" variety
 \begin{equation}\label{Dem Var}
D_{\tilde w}= P_{j_1}\times^{I} P_{j_2} \times^I\cdots\times P_{j_m}/I.
\end{equation}
Similar to $\wGr_N$, this is an iterated $\bP^{1,\pf}$-bundle. In particular, it is perfectly proper. Now assume that $\tilde w$ is a reduced word, i.e. the length $\ell(w)=m$, where $w=s_{j_1}\cdots s_{j_m}$. Then as in \cite[\S 8]{PR}, there is a surjective map
\begin{equation}\label{DR}
\pi_{\tilde w}:D_{\tilde w}\to S_w,
\end{equation}
with geometrically connected fibers. This shows that $S_w$ is perfectly proper.

In addition, we have the following proposition as in the equal characteristic situation.
\begin{prop}\label{conn comp}
There is a canonical isomorphism $\pi_1(G)_{\on{Gal}(\bar{F}/F)}\simeq \pi_0(LG)\simeq \pi_0(\Gr_{\mG})$.
\end{prop}
\begin{proof}One can argue as in \cite[\S 5]{PR}: By the standard argument (using the $z$-extension), it reduces to consider the case when $G=T$ is a torus or when $G=G_{\on{sc}}$ is semisimple and simply connected.  Note that the functor $T\mapsto \pi_0(LT)$ from the category of $F$-tori to the category of abelian groups satisfies the condition of \cite[\S 2]{Ko}. Therefore it follows from \emph{loc. cit.} that the proposition holds for $G=T$. Using the ``Demazure resolution" and the Cartan decomposition, one shows that $LG$ is connected if $G$ is simply-connected.
\end{proof}

\subsubsection{}\label{Sch in Gr}
Now we switch back to the notations as in \S~\ref{not}. So $G$ denotes an affine group scheme over $\mO$. In addition we assume that $G$ is split reductive.

We first discuss some generalizations of \S~\ref{rel pos latt}.
Let $\mE_1$ and $\mE_2$ be two $G$-torsors over $\mO$, and
let $\beta:\mE_1|_{D^*_{F}}\simeq\mE_2|_{D^*_{F}}$ be an isomorphism between them over $F$. One can generalize \eqref{rel pos} to define the relative position $\inv(\beta)$ of $\beta$ as an element in $\xcoch^+$. In addition, Lemma \ref{Hodge stra} also admits a natural generalization.
\begin{lem}\label{Hodge stra2}
Let $\mE_1$ and $\mE_2$ be two $G$-torsors over $D_{F,R}=\Spec W_\mO(R)$, and let $\beta: \mE_1|_{D^*_{F,R}}\simeq \mE_2|_{D^*_{F,R}}$ be an isomorphism. Then the set
\[(\Spec R)_{\leq \mu}=\{x\in \Spec R\mid \inv_x(\beta)\leq \mu\}\]
is a closed subset.
\end{lem}
In equal characteristic, these facts are well-known (e.g. see \cite[\S~2.1]{Z16}), and exactly the same arguments apply here.

Then we define $\Gr_{\leq \mu}\subset\Gr$ as
\[\Gr_{\leq \mu}=\{(\mE,\beta)\in \Gr\mid \inv(\beta)\leq \mu\},\]
which is a closed subspace of $\Gr$ by Lemma \ref{Hodge stra2}. It contains
\[\Gr_{\mu}=\{(\mE,\beta)\in \Gr\mid \inv(\beta)=\mu\}\]
as an open subspace. We call $\Gr_{\leq \mu}$ the (spherical) ``Schubert variety" corresponding to $\mu$ and $\Gr_{\mu}$ the corresponding ``Schubert cell". The terminology is justified by the following proposition.

\begin{prop}\label{sph sch var}
\begin{enumerate}
\item Let $\mu\in\xcoch^+$, and let $\varpi^\mu\in \Gr$ be the corresponding point (see \S~\ref{not}). Then the map
\begin{equation}\label{Sch cell}
i_\mu:L^+G/(L^+G\cap \varpi^\mu L^+G \varpi^{-\mu}) \to LG/L^+G,\quad g \mapsto g\varpi^\mu
\end{equation}
induces an isomorphism $L^+G/(L^+G\cap \varpi^\mu L^+G \varpi^{-\mu})\simeq \Gr_\mu$. 
\item $\Gr_{\mu}$ is the perfection of a quasi-projective smooth variety of dimension $(2\rho,\mu)$. 
\item $\Gr_{\leq \mu}$ is the Zariski closure of $\Gr_\mu$ in $\Gr$ and therefore is perfectly proper of dimension $(2\rho,\mu)$. 
\end{enumerate}
\end{prop}
\begin{proof}
Note that for $h\gg 0$, the Greenberg realization $L^h_pG$ of $G\otimes \mO/\varpi^h$ is a canonical model of $L^hG=L^+G/L^+G^{(h)}$, and there is a unique reduced closed subgroup $K\subset L^h_pG$ whose perfection is $(L^+G\cap \varpi^\mu L^+G \varpi^{-\mu})/L^+G^{(h)}$.
Then the quotient $L^h_pG/K$ is represented by a smooth quasi-projective variety $\Gr'_\mu$ whose perfection is $L^+G/(L^+G\cap \varpi^\mu L^+G \varpi^{-\mu})$. In addition, similar to the equal characteristic situation, it is not hard to show that $\dim \Gr'_\mu=(2\rho,\mu)$. By Proposition \ref{perfect orbit} and the Cartan decomposition, the inclusion ${\Gr'_\mu}^\pf\subset \Gr_\mu$ is a bijective locally closed embedding, and therefore is an isomorphism. This implies (1) and (2).

Finally (3) follows from Lemma \ref{Hodge stra2} and (2) by the same argument as in the equal characteristic situation (e.g. see \cite[Proposition 2.1.4]{Z16} for details).
\end{proof}

For a coweight $\mu$, let $P_{\mu}$ denote the parabolic subgroup of $G$ generated by the root subgroups $U_\al$ of $G$ corresponding to those roots $\al$ satisfying $\langle\al,\mu\rangle\leq 0$. Let $\bar G$ and $\bar P_{\mu}$ be the special fibers of $G$ and $P$. Let
us denote the natural  projection $L^+G\to \bar G^\pf$ defined by reduction mod $\varpi$ by $g\mapsto \bar{g}$.
Then there is a natural projection
\begin{equation}\label{mod to fil}
\pi_\mu:  L^+G/(L^+G\cap \varpi^\mu L^+G \varpi^{-\mu})\to (\bar G/\bar P_\mu)^\pf,\quad (gt^\mu \mod L^+G)\mapsto (\bar{g}\mod \bar P^\pf_{\mu}).
\end{equation}
The fibers are isomorphic to the perfection of affine spaces. 

We have the following generalization of the isomorphism $\Gr_{\omega_i}\simeq\Gr^\pf(i,n)$ from Proposition \ref{rep of dem res}. Recall that $\mu$ is called minuscule if $\langle\mu,\al\rangle\leq 1$ for every positive root. 
\begin{cor}\label{minusch}
If $\mu$ is minuscule, then
$\Gr_\mu=\Gr_{\leq \mu}$ and therefore $\pi_\mu$ induces an isomorphism $\Gr_\mu=(\bar G/\bar P_\mu)^\pf$. 
\end{cor}
In particular, for minuscule $\mu$, $\Gr_{\leq \mu}$ is isomorphic to the perfection of its equal characteristic counterpart. But as mentioned in Remark \ref{wGr2}, we do \emph{not} think this is true for general ``Schubert varieties".

There is a map $\xcoch\to \bZ/2,\ \mu\mapsto (-1)^{(2\rho,\mu)}$, which factors through $\xcoch(T)\to \pi_1(G)\to \bZ/2$ and therefore induces a map
\begin{equation}\label{parity map}
p:\Gr_G\to \pi_0(\Gr_G)\to \bZ/2
\end{equation}
by Proposition \ref{conn comp}. 
\begin{lem}\label{parity}
The Schubert cell $\Gr_{\mu}$ is in the even (resp. odd) components, i.e. $p(\Gr_{\mu})=1$ (resp. $p(\Gr_{\mu})=-1$) if and only if $\dim \Gr_{\mu}$ is even (resp. odd).
\end{lem}

To finish this section, we remark that although affine Grassmannians in mixed and equal characteristic share many similar properties, there are some essential difference. The first difference is that since there is no analogue of the Birkhoff decomposition for $p$-adic groups, it is not clear whether one can construct the ``big open cell" in the mixed characteristic affine Grassmannian. This is also related to the lack of a Beauville-Laszlo style description of the mixed characteristic affine Grassmannian via a ``global curve". The second difference is that  there is no ``rotation" torus acting on the mixed characteristic affine Grassmannian. As a result, 
there is no natural section of the projection $\pi_\mu$ defined in \eqref{mod to fil}.

\section{The geometric Satake}\label{GS}
We establish the geometric Satake equivalence in this setting. We use notations from \S~\ref{not}. In addition, we assume that $k$ is algebraically closed and that $G$ is a connected reductive group scheme over $\mO$  in this section. 
As explained in the introduction, one can define the category of $L^+G$-equivariant perverse sheaves on $\Gr_G$, denoted by $\on{P}_{L^+G}(\Gr_G)$. As will be explained below, there is a convolution product that makes it a semisimple monoidal category. In addition, the global cohomology functor is a natural monoidal functor. Then we establish the commutativity constraints using some numerical results of  the affine Hecke algebra. In the course, we will also develop the Mirkovi\'{c}-Vilonen's theory in this setting. 

For simplicity, we will write $\Gr$ for $\Gr_G$ if the group $G$ is clear. Proofs are sketchy or omitted if they are similar to their equal characteristic counterparts.

\subsection{The Satake category $\Sat_G$}
In this subsection, we define the Satake category $\Sat_G$ as a monoidal category.
\subsubsection{}

Recall that $\Gr$ can be written as an inductive limit of $L^+G$-invariant closed subsets $\Gr_{\leq \mu}$, which are perfectly proper, and that the action of $L^+G$ on $\Gr_{\leq \mu}$ factors through some $L^hG$ which is perfectly of finite type. Therefore, it makes sense to talk about the category of $L^+G$-equivariant perverse sheaves on $\Gr_{\leq \mu}$ (see \S \ref{equivariant category}), denoted by $\on{P}_{L^+G}(\Gr_{\leq \mu})$. Then we denote by 
$$\on{P}_{L^+G}(\Gr_G)=\underrightarrow\lim \on{P}_{L^+G}(\Gr_{\leq \mu})$$ the category of $L^+G$-equivariant perverse sheaves on $\Gr_G$. We denote by $\IC_\mu$ the intersection cohomology sheaf on $\Gr_{\leq \mu}$. Then
$\IC_\mu|_{\Gr_\mu}=\Ql[(2\rho,\mu)]$, and its restriction to each stratum $\Gr_\la$ is constant. As $$\Gr_\mu=L^+G/(L^+G\cap \varpi^{\mu} L^+G\varpi^{-\mu})$$ and  $L^+G\cap \varpi^{\mu} L^+G\varpi^{-\mu}$ is connected, the irreducible objects of $\on{P}_{L^+G}(\Gr_G)$ are exactly these $\on{IC}_\mu$'s.

\begin{lem}\label{semisimplicity}
The category $\on{P}_{L^+G}(\Gr_G)$ is semisimple.
\end{lem}
\begin{proof}The proof is literally the same as the equal  characteristic situation (see \cite{Lu} and \cite[Proposition 1]{Ga} for details): The existence of the ``Demazure resolution" (see \eqref{DR}) whose fibers have pavings by (perfect) affine spaces implies the parity property of the stalk cohomology of $\IC_\mu$s. Together with Lemma \ref{parity}, one concludes that there is no extension between two irreducible objects.
\end{proof}

\subsubsection{}
We refer to \S \ref{torsor} for the definition of the twisted product, which will also be called the convolution product in the current setting.
Note that there are the $L^+G$-torsor $LG\to \Gr$ and the $L^+G$-space $\Gr$. Then one can form the convolution affine Grassmannian 
$$\Gr\tilde\times \Gr:=LG\times^{L^+G}\Gr.$$ As in the equal  characteristic situation (e.g. \cite{MV}), one can interpret $\Gr\tilde\times\Gr$ as
\[\Gr\tilde\times\Gr(R)=\left\{(\mE_1,\mE_2,\beta_1,\beta_2)\ \left|\ \begin{split}&\mE_1,\mE_2 \mbox{ are } G\mbox{-torsors on } D_{F,R}, \\
&  \beta_1:\mE_1|_{D_{F,R}^*}\simeq \mE_0|_{D_{F,R}^*}, \beta_2:\mE_2|_{D_{F,R}^*}\simeq  \mE_1|_{D_{F,R}^*}\end{split}\right.\right\}.\]

Note that there is the convolution map
$$m:\Gr\tilde\times\Gr\to \Gr,\quad (\mE_1,\mE_2,\beta_1,\beta_2)\mapsto (\mE_2, \beta_1\beta_2)$$
and the natural projection
$$\pr_1:\Gr\tilde\times\Gr\to \Gr, \quad (\mE_1,\mE_2,\beta_1,\beta_2)\mapsto (\mE_1,\beta_1),$$
which induces $(\pr_1,m):\Gr\tilde\times\Gr\simeq \Gr\times \Gr$. In particular, the convolution Grassmannian is representable as an ind-perfect algebraic space, ind-perfectly proper. Given $\mu_1,\mu_2\in \xcoch^+$ of $G$, one can form the convolution product of $\Gr_{\leq \mu_1}$ and $\Gr_{\leq \mu_2}$,
\[\Gr_{\leq \mu_1}\tilde{\times}\Gr_{\leq \mu_2}=\{(\mE_1,\mE_2,\beta_1,\beta_2)\in\Gr\tilde\times\Gr\mid \inv(\beta_1)\leq \mu_1, \inv(\beta_2)\leq \mu_2.\},\]
which is closed in $\Gr\tilde\times\Gr$ and therefore is representable. Similarly, one can form the $n$-fold convolution Grassmannian $\Gr\tilde\times\cdots\tilde\times\Gr$, classifying $\{(\mE_i,\beta_i), \ i=1,\ldots,n\}$ where $\mE_i$ is a $G$-torsor on $D_{F,R}$ and $\beta_i: \mE_{i}|_{D_{F,R}^*}\simeq \mE_{i-1}|_{D_{F,R}^*}$ is an isomorphism. By sending $\{(\mE_i,\beta_i), \ i=1,\ldots,n\}$ to $\beta_1\cdots\beta_i:\mE_i|_{D_{F,R}^*}\simeq \mE_0|_{D_{F,R}^*}$, we obtain a map $m_i:\Gr\tilde\times\cdots\tilde\times\Gr\to \Gr$. They together induce an isomorphism
\begin{equation}\label{Grn}
(m_1,\ldots,m_n): \Gr\tilde\times\cdots\tilde\times\Gr\simeq \Gr^n.
\end{equation}
We call $m_n=m: \Gr\tilde\times\cdots\tilde\times\Gr\to \Gr$ the convolution map.
Given a sequence of dominant coweights $\mmu=(\mu_1,\ldots,\mu_n)$  of $G$, one can define the closed subspace $\Gr_{\leq\mmu}=\Gr_{\leq \mu_1}\tilde\times\cdots\tilde\times\Gr_{\leq \mu_n}$ which classifies those $\{(\mE_i,\beta_i), \ i=1,\ldots,n\}$ satisfying $\inv(\beta_i)\leq \mu_i$. Let $|\mmu|=\sum \mu_i$, then the convolution map $m$ induces
\begin{equation}\label{conv prod}
m: \Gr_{\leq \mmu}\to \Gr_{\leq |\mmu|}, \quad (\mE_\bullet,\beta_\bullet)\mapsto (\mE_n,\beta_1\cdots\beta_n).
\end{equation}
There are variants of the above construction. Namely, one can replace $\Gr_{\leq \mu_i}$ by $\Gr_{\mu_i}$ and define $\Gr_{\mu_\bullet}=\Gr_{\mu_1}\tilde\times\cdots\tilde\times\Gr_{\mu_n}$. In particular, 
\begin{equation}\label{str on conv}
\Gr_{\leq \mu_\bullet}=\bigcup_{\mu'_\bullet\leq \mu_\bullet}\Gr_{\mu'_\bullet}
\end{equation}
form a stratification of $\Gr_{\leq \mu_\bullet}$, where $\mu'_\bullet\leq \mu_\bullet$ means $\mu'_i\leq \mu_i$ for each $i$.

Now, as in the equal  characteristic situation, one can define a monoidal structure on $\on{P}_{L^+G}(\Gr)$, using Lusztig's convolution of sheaves (e.g. see \cite[\S 4]{MV} for more details). For $\mA_1,\mA_2\in \on{P}_{L^+G}(\Gr)$, we denote by $\mA_1\tilde\boxtimes\mA_2$ the ``external twisted product" of $\mA_1$ and $\mA_2$ on $\Gr\tilde\times\Gr$, i.e., the pullback of $\mA_1\tilde\boxtimes\mA_2$ along $LG\times \Gr\to \Gr\tilde\times\Gr$ is equal to the pullback the external product $\mA_1\boxtimes\mA_2$ along $LG\times \Gr\to \Gr\times \Gr$. Define
$$\mA_1\star\mA_2:=m_!(\mA_1\tilde\boxtimes\mA_2)$$ to be their convolution product, which is an $L^+G$-equivariant $\ell$-adic complex on $\Gr$. Similarly, one can define the $n$-fold convolution $\mA_1\star\cdots\star\mA_n=m_!(\mA_1\tilde\boxtimes\cdots\tilde\boxtimes\mA_n)$.

\begin{prop}\label{conv prod}
The convolution $\mA_1\star\mA_2$ is perverse.
\end{prop}
This can be proved using the numerical results of the affine Hecke algebra \cite{Lu} (see \cite{Gi1} for details).
We will outline a direct proof in the next subsection (see \S \ref{qm:conclusion}) following \cite[\S 9]{NP}, after we introduce the semi-infinite orbits.

There is an equivalent formulation of this proposition.
\begin{prop}\label{semismall}
The convolution product $m:\Gr_{\leq\mmu}\to \Gr$ is semismall. I.e., the dimension of $\Gr_{\leq \mmu}^{\la}:=m^{-1}(\Gr_{\leq \la})$ is at most $(\rho,|\mmu|+\la)$.
\end{prop}
\begin{proof}
The direction from Proposition \ref{semismall} to Proposition \ref{conv prod} is \cite[Lemma 4.3]{MV}. The inverse direction is mentioned in \cite[Remark 4.5]{MV}. As we will make use of this statement in Proposition \ref{unramified to split}, we include a sketch of the proof.

Let $d=\dim \Gr_{\mmu}\cap m^{-1}(\varpi^\la)$.
By Lemma \ref{semisimplicity} and Proposition \ref{conv prod}, we can write
\[\IC_{\mmu}:=\IC_{\mu_1}\star\cdots\star\IC_{\mu_n}=\bigoplus_{\la} V_{\mmu}^\la\otimes \IC_\la,\]
where $V_{\mmu}^\la=\Hom(\IC_\la,\IC_{\mmu})$. The spectral sequence induced by the stratification \eqref{str on conv} implies that the degree $2d- (2\rho,|\mmu|)$ stalk cohomology of the left hand side at $\varpi^\la$ is given by
$$\on{H}^{2d}_c(\Gr_{\mmu}\cap m^{-1}(\varpi^\la),\Ql).$$
The perversity of the right hand side then implies that $2d- (2\rho,|\mmu|)\leq -(2\rho,\la)$. This implies that $d\leq (\rho,|\mmu|-\la)$. By induction on $\la$, we conclude that
\[\dim \Gr_{\leq\mmu}^\la\leq d+\dim \Gr_\la=(\rho,|\mmu|+\la).\]
\end{proof}
\begin{rmk}\label{Satake fiber}
This argument also gives a canonical isomorphism
\[V_{\mmu}^\la=\on{H}_c^{(2\rho,|\mmu|-\la)}(\Gr_{\mmu}\cap m^{-1}(\varpi^\la),\Ql).\]
Together with \S~\ref{cycle class}, we see that there is a canonical basis of $V_{\mmu}^\la$ given by the set $\bB_{\mmu}^\la$ of irreducible components of $\Gr_{\mmu}\cap m^{-1}(\varpi^\la)$ of dimension $(\rho,|\mmu|-\la)$.
\end{rmk}

By identifying $(\mA_1\star\mA_2)\star\mA_3$ and $\mA_1\star(\mA_2\star\mA_3)$ with $\mA_1\star\mA_2\star\mA_3$, one equips $\on{P}_{L^+G}(\Gr)$ with a natural monoidal structure. The monoidal category $(\on{P}_{L^+G}(\Gr),\star)$ is sometimes also denoted by $\Sat_G$ for simplicity.

\subsection{Semi-infinite orbits}\label{semiinfinite geometry}
In this subsection, we discuss the geometry of semi-infinite orbits and establish the corresponding Mirkovi\'c-Vilonen theory in our setting. We will use the notations from \S~\ref{not}. 
\subsubsection{}
The affine Grassmannian $\Gr_U$ of $U$ is clearly represented by an inductive limit of perfect affine spaces. Since $U\backslash G$ is quasi-affine, Proposition \ref{gen G} implies that $\Gr_U\subset \Gr_G$ is a locally closed embedding. Write $S_0=\Gr_U\subset \Gr_G$. For $\la\in\xcoch$, let $S_\la= LU\varpi^\la$ be the orbit through $\varpi^\la$. Then $S_\la= \varpi^\la \Gr_U$ and therefore is locally closed in $\Gr_G$. 
By the Iwasawa decomposition,
\[\Gr_G=\bigcup_{\la\in \xcoch} S_\la.\]
As in the equal characteristic situation, one can also regard semi-infinite orbits as the attractor locus of certain torus-action on $\Gr_G$. Namely, let $2\rho^\vee$ denote the sum of positive coroots of $G$ (with respect to $B$), regarded as a cocharacter of $G$. Note that the projection $L_p^+\bG_m\to \bG_m$ admits a unique section $\bG_m\to L_p^+\bG_m$ (as the maximal torus of $L^+_p\bG_m$). Then we have a cocharacter
\begin{equation*}\label{2rhoGm}
\bG^\pf_m\to L^+\bG_m\stackrel{L^+(2\rho^\vee)}{\to} L^+T\subset L^+G.
\end{equation*}
The action of $L^+G$ on $\Gr$ induces a $\bG^\pf_m$-action on $\Gr_G$. The set of fixed points are $\{\varpi^\la\mid \la\in\xcoch\}$ and the action contracts $S_\la$ to $\varpi^\la$.
On the other hand, let $B^-\subset G$ be the Borel opposite to $B$ with $U^-$ its unipotent radical, and let $\{S^-_\la= LU^-\varpi^\la,\ \la\in\xcoch\}$ be the opposite semi-infinite orbits. These orbits can be regarded the repeller locus of the above $\bG_m^\pf$-action on $\Gr_G$.

As in the equal  characteristic situation, we have the following closure relation for semi-infinite orbits.
\begin{prop}\label{closure semi-infinite1}
The closure $\bar S_\la= \cup_{\la'\leq \la} S_{\la'}$. More precisely $\overline{S_\la\cap\Gr_{\leq \mu}}=\cup_{\la'\leq \la} S_{\la'}\cap\Gr_{\leq \mu}$. Similarly $\bar S^-_\la=\cup_{\la'\geq\la}S^-_{\la'}$.
\end{prop}
\begin{proof}It is enough to prove the statement for $S_\la$.
The argument of \cite[Proposition 3.1]{MV} does not apply directly because in mix characteristic one cannot attach to an affine root a map $\SL_2\to LG$ and (currently) there is no Kac-Moody theory available. However, the alternative argument given in  \cite[Proposition 5.3.6]{Z16} applies to the current situation.
\end{proof}

Note that the restriction of the $L^+G$-torsor $LG\to \Gr_G$ over $S_{\la}$ has a canonical reduction as an $L^+U$-torsor given by $LU\to S_{\la},\ n\mapsto \varpi^{\la}n \mod L^+G$. Then it makes sense to talk about the twisted product of these semi-infinite orbits.
Let $\nu_\bullet$ be a sequence of (not necessarily dominant) coweights of $G$.  One can define
\[S_{\nu_\bullet}:=S_{\nu_1}\tilde\times S_{\nu_2}\tilde\times\cdots\tilde\times S_{\nu_n}\subset \Gr\tilde\times \Gr\tilde\times\cdots\tilde\times\Gr.\]

The formula
\begin{equation}\label{fact for}
(\varpi^{\nu_1}x_1,\ldots,\varpi^{\nu_n}x_n)\mapsto (\varpi^{\nu_1}x_1,\varpi^{\nu_1+\nu_2}(\varpi^{-\nu_2}x_1\varpi^{\nu_2})x_2,\ldots),
\end{equation}
defines an isomorphism 
\begin{equation}\label{factorizationI}
m: S_{\nu_\bullet}\simeq S_{\nu_1}\times S_{\nu_1+\nu_2}\times\cdots\times S_{|\nu_\bullet|},
\end{equation}
as locally closed subsets of $\Gr\tilde\times \Gr\tilde\times\cdots\tilde\times\Gr\simeq \Gr^n$ (see \eqref{Grn} for this isomorphism).

Note that each $S_{\nu_i}\cap\Gr_{\leq \mu_i}$ is $L^+U$-invariant.
So it also makes sense to define the twisted product of these $L^+U$-spaces $S_{\nu_i}\cap\Gr_{\leq \mu_i}$. In addition, there is the canonical isomorphism
\begin{equation}\label{factorizationII}
(S_{\nu_1}\cap\Gr_{\leq \mu_1})\tilde\times\cdots\tilde\times (S_{\nu_n}\cap \Gr_{\leq \mu_n})\simeq S_{\nu_\bullet}\cap \Gr_{\leq\mmu}.
\end{equation}

\begin{rmk}\label{non-split conv}
(i) Note that $S_\nu\cap \Gr_{\leq \mu}$ is closed in $S_\nu$ and therefore is a scheme.
(ii) Unlike \cite[Lemma 9.1]{NP}, it is not clear whether the twisted product on the left hand side of \eqref{factorizationII} splits as a product.
\end{rmk}

\medskip

The Mirkovi\'c-Vilonen theory exists in our situation. The key statement is the following.

\begin{prop}\label{coh vanishing}
For any $\mA\in\on{P}_{L^+G}(\Gr_G)$, $\on{H}_c^i(S_\la,\mA)=0$ unless $i=(2\rho,\la)$.
\end{prop}
The proof of this proposition will be sketched in \S~\ref{qm:conclusion}. 
Note that the proof in \cite[Theorem 3.5]{MV} does not work in mixed characteristic. Following \cite{MV}, we define the weight functor
\begin{equation}\label{weight functor F}
\on{CT}_\la: \Sat_G\to \on{Vect}_\Ql, \quad \on{CT}_\la(\mA)=\on{H}_c^{2(\rho,\la)}(S_\la,\mA).
\end{equation}

\begin{cor}\label{dim MV cycle}
The perfect scheme $S_\la\cap \Gr_{\leq \mu}$ is equidimensional, of dimension $(\rho,\la+\mu)$. The number of its irreducible components equals to the dimension of the $\la$-weight space $V_\mu(\la)$ of the irreducible representation $V_\mu$ of $\hat G$ of highest weight $\mu$.
\end{cor}
\begin{proof}
First, we show that $S_\la \cap \Gr_{\leq \mu}$ is of dimension $(\rho,\la+\mu)$, and the number of irreducible components of maximal dimension equals to the dimension of $V_\mu(\la)$. The proof is a special case of \cite[Proposition 5.4.2]{GHKR}: first note that the group $G$ is in fact already defined over the ring of integers of a $p$-adic field so $\Gr$ is in fact defined over some finite field. Then it is enough to show that
\begin{equation}\label{asym}
\lim\limits_{q\to\infty}\frac{|(S_\la \cap \Gr_{\leq \mu})(\bF_q)|}{q^{(\rho,\la+\mu)}}= \dim V_\mu(\la).
\end{equation}
To prove this, one can replace $\Gr_{\leq \mu}$ in the above formula by the open cell $\Gr_{\mu}$. Now we regard $U(F)$ as a locally compact topological group and we normalize the measure on $U(F)$ so that the volume of $U(\mO)$ is one. Then one can express
\begin{equation}\label{count pt}
|(S_\la \cap \Gr_{\mu})(\bF_q)|=\int_{U(F)}1_{G(\mO)\varpi^\mu G(\mO)}(\varpi^\la u)du.
\end{equation}

Recall the Satake isomorphism
$$\on{Sat}:C_c^\infty(G(\mO)\backslash G(F)/G(\mO))\simeq C_c^\infty(T(F)/T(\mO))^W=C[\xcoch(T)]^W$$ is given by $$\on{Sat}(f)(\varpi^\la)=q^{-(\rho,\la)}\int_{U(F)}f(\varpi^\la u)du .$$
Let $H_\mu$ denote the function on $C_c^\infty(G(\mO)\backslash G(F)/G(\mO))$ such that
\begin{equation}\label{KL}
\on{Sat}(H_\mu)(\varpi^\la)=\dim V_\mu(\la).
\end{equation}
The theory of Lusztig-Kato polynomials implies that
\begin{equation}\label{LK}
q^{-(\rho,\mu)}1_{G(\mO)\varpi^\mu G(\mO)}= H_\mu+ \sum_{\nu<\mu}  P_{\mu\nu}(q^{-1})H_{\nu},
\end{equation}
where $P_{\mu\la}(v)$ is some polynomial of $v$ without the constant coefficient.
Combining \eqref{count pt}, \eqref{KL} and \eqref{LK}
\[
\frac{|(S_\la\cap\Gr_{\mu})(\bF_q)|}{q^{(\rho,\la+\mu)}} = \dim V_\mu(\la)+\sum_{\nu<\mu}c_{\mu\nu}(q^{-1}) \dim V_\nu(\la).
\]
As $q\to \infty$, the error term goes to zero and the dominant term becomes $\dim V_\mu(\la)$.

Next, one can follow  \cite[Lemma 2.17.4]{GHKR} to deduce the equidimensionality of $S_\la\cap\Gr_{\leq \mu}$ from the the upper bounds of the dimension of $S_\la\cap\Gr_{\leq \mu}$ and Proposition \ref{coh vanishing}.
\end{proof}

We have another two corollaries of Proposition \ref{coh vanishing}.
Let $\bB_{\mu}(\la)$ denote the set of irreducible components of $S_\la\cap \Gr_{\leq \mu}$. More generally,  let $\bB_{\mmu}(\la)$ denote the set of irreducible components of $m^{-1}(S_\la)\cap \Gr_{\leq\mmu}$. 

\begin{cor}\label{semi coh}There is a canonical isomorphism
$\on{CT}_\la(\IC_\mu)\simeq \Ql[\bB_\mu(\la)]$.
More precisely, the cycle classes of irreducible components of $S_\la\cap\Gr_{\leq \mu}$ form a basis of $\on{H}^i_c(S_\la,\IC_\mu)$.
\end{cor}
\begin{proof}One can use the same argument as in \cite[Proposition 3.10]{MV}. Namely, the stratification of $S_\la$ by $\{S_\la\cap\Gr_\mu, \mu\in\xcoch(T)^+\}$ induces a spectral sequence with the $E_1$-term $\on{H}_c^*(S_\la\cap\Gr_\mu,\mA)$ and the abutment $\on{H}_c^*(S_\la,\mA)$. 
Combining Proposition \ref{coh vanishing} and Corollary \ref{dim MV cycle}, one obtains that $\on{H}^{2(\rho,\la)}_c(S_\la,\IC_\mu)\simeq \on{H}^{(2\rho,\la)}_c(S_\la\cap\Gr_\mu,\Ql)$.
The claim follows.
\end{proof}
We define the total weight functor (which is the categorical analogue of the Satake transform) as
\begin{equation}\label{weight functor1}
\on{CT}:=\bigoplus_{\la} \on{CT}_\la: \on{P}_{L^+G}(\Gr_G)\to \on{Vect}_{\Ql}.
\end{equation}
\begin{cor}\label{weight functor}
There is a canonical isomorphism
\[\on{H}^*(\Gr_G,-)\simeq \on{CT}: \on{P}_{L^+G}(\Gr_G)\to \on{Vect}_{\Ql}.\]
The functor $\on{H}^*(\Gr_G,-)$ is faithful. 
\end{cor}
\begin{proof}
The argument is the same as in \cite[Theorem 3.6]{MV}. Namely, according to Proposition \ref{closure semi-infinite1}, there are two stratifications of $\Gr$, one by semi-infinite orbits $\{S_\la, \la\in\xcoch\}$, and the other by opposite semi-infinite orbits $\left\{S^-_\la, \la\in\xcoch\right\}$. The first stratification induces a spectral sequence with the $E_1$-term $\on{H}_c^*(S_\la,-)$ and the abutment $\on{H}^*(-)$. It degenerates at the $E_1$-term for degree reasons by virtue of Proposition \ref{coh vanishing}. So there is a natural filtration on $\on{H}^*$ with the associated graded being $\oplus_\la \on{H}_c^*(S_\la,-)$. Explicitly, this is a filtration indexed by $(\xcoch,\leq)$ defined as
\[\on{Fil}_{\geq \mu}\on{H}^*(\mA)=\ker(\on{H}^*(\mA)\to \on{H}^*(S_{<\la},\mA))\] 
where $S_{<\la}=\bar{S}_\la-S_\la$. 

For $\mA\in \Sat_G$ and $Z\subset \Gr$ a closed subset, let $\on{H}^*_{Z}(\mA)$ denote the cohomology of the $!$-restriction of $\mA$ to $Z$. Applying Braden's theorem for algebraic spaces (see \cite{DG}) to some model, there is a canonical isomorphism
\begin{equation}\label{braden}
\on{H}^*_c(S_\la, \mA)\simeq \on{H}^*_{S^-_{\la}}(\mA).
\end{equation}
Then the second stratification of $\Gr$ also induces a filtration of $\on{H}^*$ as
\[\on{Fil}'_{<\la}\on{H}^*(\mA)=\on{Im}(\on{H}^*_{S^-_{<\la}}(\mA)\to \on{H}^*(\mA)),\]
where $S^-_{<\la}=\bar S^-_{\la}-S^-_\la$.
These two filtrations are complimentary to each other by \eqref{braden} and together define the decomposition $\on{H}^*=\oplus_\la \on{H}_c^*(S_\la,-)$.

Since $\on{P}_{L^+G}(\Gr_G)$ is semisimple and $\on{H}^*(\Gr_G,\IC_\mu)$ is non-zero for every $\mu$, $\on{H}^*$ is faithful.
\end{proof}

\subsubsection{}We discuss the geometry of $\Gr_{\leq \mu}$ when $\mu$ is a (quasi-)minuscule cocharacter, similar to \cite[\S 6-\S 8]{NP}, but with a few justifications. Denote by $\bar{G}=G\otimes_{\mO}k$ the special fiber of $G$, which is a reductive group over $k$, with $\bar{U}\subset \bar{B}\subset \bar{G}$.

Recall that a dominant coweight $\mu$ of $G$ is called (quasi-)minuscule if all (non-zero weights) of the irreducible representation $V_\mu$ of $\hat G$ are in a single orbit under the Weyl group. 
Recall that if $G$ is a simple group not of type $A$, then the unique quasi-minuscule (but non-minuscule) coweight is the unique \emph{short} dominant coroot. 

\begin{lem}\label{min and qmin}
Proposition \ref{coh vanishing} hold for $\mA=\IC_\mu$ when $\mu$ is (quasi-)minuscule.
\end{lem}
If $\mu$ is a minuscule coweight of $G$, then  $\Gr_{\leq \mu}=\Gr_{\mu}= (\bar{G}/\bar P_{\mu})^\pf$ by Corollary \ref{minusch}.  In this case,
\[S_\la\cap \Gr_\mu=\left\{\begin{array}{ll}
\emptyset & \la\not\in W\mu \\
(\bar{U}w\bar P_{\mu}/\bar P_{\mu})^\pf & \la=w\mu,
\end{array}\right.\] 
is irreducible, isomorphic to the perfection of an affine space. Then Lemma \ref{min and qmin} is clear.

Next we assume that $\mu$ is quasi-minuscule but non-minuscule. In this case $\mu=\theta$ is a coroot and we denote the corresponding root by $\theta^\vee$. In this case $\Gr_{\leq \mu}=\Gr_\mu\sqcup \Gr_0$.
Several discussions in \cite{NP} need justification. We first construct a ``resolution" of $\Gr_{\leq \mu}$. The one given in \emph{loc. cit.} does not work in mixed characteristic. Our construction is different, and arises as a discussion with X. He.

Recall that we fix a maximal torus $T\subset G$ over $\mO$. For a root $\al$, let $U_\al$ denote the corresponding root subgroup of $G$ over $\mO$. We identify $U_\al(F)=F$ such that $U_\al(\mO)=\mO$.
For a real number $r\in [0,1]$, we consider the parahoric subgroup of $G(F)$ generated by $T(\mO)$ and the subgroups $\varpi^{\lceil \langle r\mu,\al\rangle\rceil}\mO \subset F=U_\al(F)$ for all roots $\al$ . It determines the parahoric group scheme $\mG_r$ over $\mO$. Let $Q_r=L^+\mG_r$ denote the corresponding $p$-adic jet group. 
Note that 
\begin{enumerate}
\item $Q_0=L^+G$ and $Q_1=\varpi^{\mu}L^+G\varpi^{-\mu}$. 
\item $Q_{\frac{1}{2}}$ is a maximal parahoric. 
\item $Q_{\frac{1}{4}}= Q_0\cap Q_{\frac{1}{2}}$ and $Q_{\frac{3}{4}}= Q_{\frac{1}{2}}\cap Q_{1}$.
\end{enumerate}

\begin{lem}\label{resolution qm}
(i) The quotient $Q_{\frac{1}{2}}/ Q_{\frac{3}{4}}$ is isomorphic to $\bP^{1,\pf}$.

(ii) The map
\[\pi:\wGr_{\leq \mu}:=Q_0\times^{Q_{\frac{1}{4}}}Q_{\frac{1}{2}}/Q_{\frac{3}{4}}\to \Gr_{\leq \mu}, \quad (g,g')\mapsto gg'\varpi^\mu\]
restricts to an isomorphism 
$$\mathring{\pi}:Q_0\times^{Q_{\frac{1}{4}}}(Q_{\frac{1}{4}}Q_{\frac{3}{4}})/Q_{\frac{3}{4}}\simeq \Gr_{\mu},$$ 
and to a contraction
\[\pi_0: (\bar G/\bar P_{\mu})^\pf\simeq Q_0\times^{Q_{\frac{1}{4}}} Q_{\frac{1}{4}}s_{1+\theta}Q_{\frac{3}{4}}/Q_{\frac{3}{4}}\to \Gr_0=\{1\},\]
where $s_{1+\theta}$ is the affine reflection corresponding to $1+\theta$.
\end{lem}

\begin{proof}For (i), it is enough to observe that the only affine root appearing in $Q_{\frac{1}{2}}$ but not in $Q_{\frac{1}{4}}$ is $-\theta-1$. For (ii), note that $Q_0\cap Q_{\frac{3}{4}}=Q_0\cap Q_1$. Therefore, the statement for $\mathring{\pi}$ holds. The statement for $\pi_0$ is clear.
\end{proof}
Let us write
\[\phi: \wGr_{\leq \mu}\to Q_0/Q_{\frac{1}{4}}=(\bar{G}/\bar P_{\mu})^\pf, \quad \mathring{\phi}: \Gr_\mu\stackrel{\mathring{\pi}^{-1}}{\to} Q_0\times^{Q_{\frac{1}{4}}}(Q_{\frac{1}{4}}Q_{\frac{3}{4}})/Q_{\frac{3}{4}}\to Q_0/Q_{\frac{1}{4}} \]
where $\bar P_{\mu}$ as before is the parabolic of $\bar G$ whose roots are those $\al$ with $\langle \al,\mu\rangle\leq 0$. Note that $\mathring{\phi}$ is nothing but the projection $\pi_\mu$ from \eqref{mod to fil}.
Let $\Delta_{\theta}$ denote the subset of simple coroots that are conjugate to $\theta$ under the action of the Weyl group. If $G$ is a simple group, then $\Delta_{\theta}$ is the set of short simple coroots.

Now we study $S_\la\cap \Gr_{\leq \mu}$. If $\la=w\mu$ for some $w\in W$, then $S_\la\cap\Gr_{\leq \mu}= L^+U\varpi^\la$, from which one deduces: if $\la=w\mu$ is a positive coroot, then
\[S_\la\cap \Gr_{\leq \mu}=S_\la\cap \Gr_\mu=\mathring{\phi}^{-1}(\bar{U}w\bar P_{\mu}/\bar P_{\mu})^\pf;\]
if $\la=w\mu$ is a negative coroot, then still $S_\la\cap \Gr_{\leq \mu}=S_\la\cap \Gr_\mu$ and
\[\mathring{\phi}: S_\la\cap \Gr_{\leq \mu}\simeq (\bar{U}w\bar P_{\mu}/\bar P_{\mu})^\pf.\]
Finally,
\[S_0\cap\Gr_{\leq \mu}=\pi(\phi^{-1}(\bigcup_{w\mu<0} \bar{U}w\bar{P}_{\mu}/\bar{P}_{\mu})^\pf)\setminus\bigcup_{w\mu<0}(S_{w\mu}\cap\Gr_{\leq \mu}).\]
There is a canonical bijection between $\Delta_\theta$ and the set of irreducible components of $S_0\cap\Gr_{\leq \mu}$ given as follows:
$\alpha\in\Delta_\theta$ corresponds to the unique irreducible component of $S_0\cap\Gr_{\leq \mu}$ given by 
\begin{equation}\label{irr comp}
(S_0\cap\Gr_{\leq \mu})^\al:=\Gr_0\bigcup\pi(\phi^{-1}(\bar{U}w\bar{P}_{\mu}/\bar{P}_{\mu})^\pf)\setminus (S_{w\mu}\cap\Gr_{\leq \mu}),
\end{equation} 
where $w\mu=-\al$. 

Now we prove Lemma \ref{min and qmin} for $\mu=\theta$. This is clear for $\la=w\mu$. It remains to consider the case $S_0\cap\Gr_{\leq \mu}$. Let $d=(2\rho,\mu)$. We will ignore the Tate twist in the sequel.
By the decomposition theorem (applying to certain model of $\pi:\wGr_{\leq \mu}\to \Gr_{\leq \mu}$), we have
\[\pi_*\Ql[d]=\IC_\mu\oplus \mC,\]
where $\mC$ is a certain complex of vector spaces supported at $\Gr_0$. One has
\begin{equation}\label{cohC}
\on{H}^i(\mC)=\left\{\begin{array}{ll}  \on{H}^{i+d}(\bar G/\bar P_{\mu}) & i\geq 0\\
                                                          \on{H}^{i+d-2}(\bar G/\bar P_{\mu}) & i<0.
\end{array}\right.
\end{equation}
Indeed, the first equality follows from the fact that the stalk cohomology of $\IC_\mu$ is concentrated in negative degrees, and the second equality follows from the first by duality.

On the other hand, we have
\[R\Gamma_c(\pi^{-1}(S_0\cap \Gr_{\leq \mu}),\Ql[d])=R\Gamma_c(S_0,\IC_\mu)\oplus \mC.\]
Note that $\pi^{-1}(S_0\cap\Gr_{\leq \mu})=\phi^{-1}(\bigcup_{w\mu<0} \bar{U}w\bar{P}_{\mu}/\bar{P}_{\mu})^\pf\setminus \pi^{-1}(\bigcup_{w\mu<0}(S_{w\mu}\cap\Gr_{\leq \mu}))$ and that the map
\[\phi:\phi^{-1}(\bigcup_{w\mu<0} \bar{U}w\bar{P}_{\mu}/\bar{P}_{\mu})^\pf\to (\bigcup_{w\mu<0} \bar{U}w\bar{P}_{\mu}/\bar{P}_{\mu})^\pf\]
is a $\bP^{1,\pf}$-fibration. In addition, $\pi^{-1}(\bigcup_{w\mu<0}(S_{w\mu}\cap\Gr_{\leq \mu}))$ can be regarded as a section of this map.
Therefore,
\begin{equation}\label{seccoh}
R\Gamma_c(\pi^{-1}(S_0\cap \Gr_{\leq \mu}),\Ql[d])=R\Gamma_c(\bigcup_{w\mu<0} \bar{U}w\bar{P}_{\mu}/\bar{P}_{\mu},\Ql[d-2]).
\end{equation}

To prove Lemma \ref{min and qmin} for $S_0\cap\Gr_{\leq \mu}$, it remains to compare \eqref{cohC} and \eqref{seccoh}. However, note that the right hand sides of both equalities only involve the group $\bar G$, defined over $k$. Therefore, one can directly apply the computation in \cite[\S 8]{NP} to conclude that $\on{H}^i(\mC)=\on{H}^i_c(\pi^{-1}(S_0\cap\Gr_{\leq \mu}))$ for $i\neq 0$ and if $i=0$,
\[\on{H}^0(S_0,\IC_\mu)\simeq\Ql^{|\Delta_\theta|}.\]

This finishes the proof of Lemma \ref{min and qmin}.

\begin{rmk}
In fact, we also showed that Corollary \ref{semi coh} holds in these cases.
\end{rmk}

\subsubsection{}\label{qm:conclusion}
Now combining the proof of Lemma \ref{min and qmin} with \eqref{factorizationII}, we have the following corollaries, whose proofs are as in \cite[ 9.2-9.4]{NP}. Let $M$ be the set of minimal elements in $\xcoch^+\setminus \{0\}$. This is exactly the set of non-zero quasi-minuscule coweights. 
\begin{cor}\label{dim conv MV}
Let $\mu_\bullet=(\mu_1,\ldots,\mu_m)\subset M$. Then for any $\la_\bullet=(\la_1,\ldots,\la_m)$, $S_{\la_\bullet}\cap\Gr_{\leq \mu_\bullet}$ is equidimensional, and
\[\dim (S_{\la_\bullet}\cap \Gr_{\leq \mu_\bullet})=(\rho,|\la_\bullet|+|\mu_\bullet|).\] 
\end{cor}

\begin{cor}\label{Semismall conv}
Let $\mu_\bullet=(\mu_1,\ldots,\mu_m)\subset M$. Then the map $\pi: \Gr_{\leq \mu_\bullet}\to \Gr_{\leq |\mu_\bullet|}$ is semi-small, and therefore $\IC_{\mu_1}\star\cdots\star\IC_{\mu_m}$ is perverse.
\end{cor}

This corollary allows us to define a full additive subcategory of $\Sat_G$, spanned by objects isomorphic to $\IC_{\mu_1}\star\cdots\star\IC_{\mu_m}$ for $\mmu\subset M$. Let us denote this subcategory by $\Sat^0_G$. Note that $\Sat^0_G$ is in fact a monoidal subcategory of $\Sat_G$ under the convolution.

\begin{lem}\label{idempotent completion}
As a monoidal abelian category, $\Sat_G$ is the idempotent completion of $\Sat^0_G$. Concretely, every $\IC_\mu$ appears as a direct summand of $\IC_{\mu_1}\star\cdots\star\IC_{\mu_m}$ for $\mmu\in M$.
\end{lem}
This is a geometric version of the so-called PRV conjecture. The argument as in \cite[Proposition 9.6]{NP} applies here. Note that this lemma and Corollary \ref{Semismall conv} together  imply Proposition \ref{conv prod}.

In addition, we have the following corollary. The argument is similar to the proof of \cite[Theorem 3.1]{NP} given at the beginning of  \S 11 of \emph{ibid.}. But due to Remark \ref{non-split conv}, one justification is needed.
\begin{cor}\label{van coh aux}
Let $\mu_\bullet\subset M$. For any $\la$, $R\Gamma_c(S_\la,\IC_{\mu_1}\star\cdots\star\IC_{\mu_m})$ is concentrated in degree $(2\rho,\la)$.
\end{cor}
\begin{proof}Note that $\pi^{-1}S_\la=\bigsqcup_{\nu_\bullet,|\nu_\bullet|=\la}S_{\nu_\bullet}$. It is enough to show that $R\Gamma_c(S_{\nu_\bullet},\IC_{\mu_1}\star\cdots\star\IC_{\mu_m})$ is concentrated in degree $(2\rho,\la)$.

For an integer $n$, let $S_\nu^{(n)}$ denote the pushout of the $L^+U$-torsor $LU\to S_\nu$ along $L^+U\to L^nU$, and let $(S_\nu\cap\Gr_{\leq \mu})^{(n)}$ denote the restriction of $S_\nu^{(n)}$ to $S_\nu\cap\Gr_{\leq \mu}\subset S_\nu$. This is an $L^nU$-torsor over $S_\nu\cap \Gr_{\leq \mu}$.
Note that the action of $L^+U$ on every $S_{\nu_i}\cap \Gr_{\leq \mu_i}$ factors through some $L^{r_i}U$. Now we can choose $\{r_i,\ i=1,\ldots,m\}$ such that $r_m=0$ and that the action of $L^+U$ on $(S_{\nu_i}\cap \Gr_{\leq \mu_i})^{(r_i)}$ factors through $L^{r_{i-1}}U$. Let $\pr^*\IC_{\mu_i}$ denote the pullback the sheaf along the projection $(S_{\nu_i}\cap\Gr_{\leq \mu_i})^{(r_i)}\to (S_{\nu_i}\cap\Gr_{\leq \mu_i})$. Then 
$\prod (S_{\nu_i}\cap \Gr_{\leq \mu_i})^{(r_i)}$ is an $\prod L^{r_i}U$-torsor. Since $L^{r_i}U$ is isomorphic to the perfection of an affine space of dimension $r_i\dim U$, we have
\begin{multline*}R\Gamma_c((S_{\nu_1}\cap\Gr_{\leq \mu_1})\tilde\times\cdots\tilde\times (S_{\nu_m}\cap \Gr_{\leq \mu_m}),\IC_{\mu_1}\tilde\boxtimes\cdots\tilde\boxtimes\IC_{\mu_m})\\ =R\Gamma_c((S_{\nu_1}\cap\Gr_{\leq \mu_1})^{(r_1)}, \pr^*\IC_{\mu_1})\otimes\cdots\otimes R\Gamma_c((S_{\nu_m}\cap \Gr_{\leq \mu_m})^{(r_m)},\pr^*\IC_{\mu_m})[2\dim U\sum r_i].\end{multline*}
The corollary now follows from \eqref{factorizationII} and Lemma \ref{min and qmin}.
\end{proof}

Note that this corollary and Lemma \ref{idempotent completion} together imply Proposition \ref{coh vanishing}.

\subsection{The monoidal structure on $\on{H}^*$}\label{mon str of H}
\subsubsection{}
We endow the hypercohomology functor 
$$\on{H}^*(-):=\on{H}^*(\Gr,-):\Sat_G\to \on{Vect}_{\Ql}$$
with a monoidal structure. In equal characteristic, this is achieved by identifying the convolution product with the fusion product defined using a global curve (cf. \cite{MV} and \cite[\S 5.3]{BD}). If in addition, $k=\bC$, one can endow $\on{H}^*$ with  another monoidal structure by identifying $\Gr^\flat_G$ with the based loop space of a maximal compact subgroup of $G$ and identifying the convolution map of the affine Grassmannian with the multiplication of the loop group (cf. \cite{Gi1}). Neither method applies directly in our setting so we need a third construction. It is not hard to check that in equal  characteristic, all three monoidal structures coincide.

Recall that for $\mA\in\Sat_G$, it makes sense to consider its $L^+G$-equivariant cohomology
$\on{H}^*_{L^+G}(\mA)$, which is an $R_{\bar G,\ell}$-module (see \S~\ref{equivariant category}). But as is well-known, there is another $R_{\bar G,\ell}$-module structure on $\on{H}^*_{L^+G}(\mA)$ so it is an $R_{\bar G,\ell}$-bimodule. In fact, let $L^+G^{(m)}\subset L^+G$ denote the $m$th congruence subgroup, and let 
$$\Gr^{(m)}=LG/L^+G^{(m)}$$ denote the universal $L^m G$-torsor on $\Gr$. Then $\Gr^{(m)}$ admits an action of $L^+G\times L^mG$ and the projection $\pi_m:\Gr^{(m)}\to \Gr$ is $L^+G$-equivariant. Then by \eqref{coh free action},
\[\on{H}^*_{L^+G}(\mA)\simeq \on{H}^*_{L^+G\times L^mG}(\pi_m^*\mA),\] 
giving an $R_{\bar G,\ell}$-bimodule structure on $\on{H}^*_{L^+G}(\mA)$. This structure is independent of $m$, as soon as $m>0$.
 Recall that the category of $R_{\bar G,\ell}$-bimodules has a natural monoidal structure. 
\begin{lem}\label{soergel monoidal}
There is a natural monoidal structure on $\on{H}^*_{L^+G}(-):\Sat_G\to (\on{R}_{\bar G,\ell}\otimes\on{R}_{\bar G,\ell})\on{-mod}$. I.e., for every $\mA_1,\mA_2,\ldots,\mA_n$, there is a canonical isomorphism of $R_{\bar G,\ell}$-bimodules
\[\on{H}^*_{L^+G}(\mA_1\star\cdots\star\mA_n)\simeq \on{H}^*_{L^+G}(\mA_1)\otimes_{R_{\bar G,\ell}}\cdots\otimes_{R_{\bar G,\ell}}\on{H}^*_{L^+G}(\mA_n),\]
satisfying the natural compatibility conditions.
\end{lem}
\begin{proof}This is standard (in light of Soergel's bimodules) and we sketch a proof. In fact, the idea already appears in the proof of Corollary \ref{van coh aux}.

For a closed subset $Z\subset \Gr$, let $Z^{(m)}$ denote its preimage in $\Gr^{(m)}$. We choose a sequence of positive integers $(m_1,\ldots,m_n)$, such that $L^+G$ acts on $\on{Supp}(\mA_i)^{(m_i)}$ via $L^+G\to L^{m_{i-1}}G$. Then there is an $L^+G\times\prod_i L^{m_i}G$-equivariant projection
\[
\prod_i\on{Supp}(\mA_i)^{(m_i)}\to \on{Supp}(\mA_1)\tilde\times\cdots\tilde\times\on{Supp}(\mA_n),\]
where $L^+G$ acts by left multiplication, and $L^{m_i}G$ acts on $\on{Supp}(\mA_{i})^{(m_i)}\times\on{Supp}(\mA_{i+1})^{(m_{i+1})}$  diagonally from the middle. It induces a canonical isomorphism
\begin{equation}\label{de1}
\on{H}^*_{L^+G}(\mA_1\star\cdots\star\mA_n)\simeq \on{H}^*_{L^+G\times\prod L^{m_i}G}(\boxtimes_i\pi_{m_i}^*\mA_i).
\end{equation}
On the other hand, the $L^+G\times\prod_i L^{m_i}G$-equivariant projection
\[\prod_i\on{Supp}(\mA_i)^{(m_i)}\to \prod_i\on{Supp}(\mA_i),\]
where $L^+G$ acts on $\on{Supp}(\mA_1)$ and $L^{m_i}G$ acts on $\on{Supp}(\mA_{i+1})$ by left multiplication,
induces a map
\begin{equation}\label{de2}
\on{H}^*_{L^+G}(\mA_1)\otimes_{R_{\bar G,\ell}}\cdots\otimes_{R_{\bar G,\ell}}\on{H}^*_{L^+G}(\mA_n)\to  \on{H}^*_{L^+G\times\prod L^{m_i}G}(\boxtimes_i\pi_{m_i}^*\mA_i).
\end{equation}
The composition of \eqref{de1} and \eqref{de2} gives a map
\[\on{H}^*_{L^+G}(\mA_1)\otimes_{R_{\bar G,\ell}}\cdots\otimes_{R_{\bar G,\ell}}\on{H}^*_{L^+G}(\mA_n)\to \on{H}^*_{L^+G}(\mA_1\star\cdots\star\mA_n),\]
which is an isomorphism by an easy spectral sequence argument. Its inverse then gives the desired isomorphism, which is clearly compatible with the associativity constraints.
\end{proof}

\subsubsection{}\label{monoidal equiv coh}
To continue, we make the following observation. Recall that we fix $T\subset B\subset G$, and denote $\bar T\subset\bar B\subset \bar G$ their fibers over $\mO/\varpi$.
Let $\widetilde W=N_G(T)(F)/T(\mO)$ denote the Iwahori-Weyl group of $G(F)$, where $N_G(T)$ is the normalizer of $T$ in $G$. Let $\overline W=\widetilde{W}/\xcoch$ denote the finite Weyl group, i.e. the Weyl group of $G$.  For (the $p$-adic jet group of) a parahoric $P$ that contains $T$, let $\bar L_P$ denote the reductive quotient of $P$ (we ignore the perfection). Then $\bar T\subset \bar L_P$. Let $W_P\subset \widetilde W$ denote the Weyl group of $\bar L_P$, and let $\overline W_P$ denote its image in $\overline W$. Then
\begin{equation}
R_{\bar L_P,\ell}= R_{\bar T,\ell}^{\overline W_P}.
\end{equation}
In particular, we see that for every $P$, $R_{\bar G,\ell}=R_{\bar T,\ell}^{\overline W}\subset R_{\bar L_P,\ell}\subset R_{\bar T,\ell}$.

\begin{lem}\label{left=right}
The two $R_{\bar G,\ell}$-structures on $\on{H}^*_{L^+G}(\mA)$ coincide.
\end{lem}
\begin{proof}
According to Lemma \ref{idempotent completion} and Lemma \ref{soergel monoidal}, it is enough to prove this for $\mA=\IC_\mu$ when $\mu$ is quasi-minuscule.
We first consider the case when $\mu$ is quasi-minuscule but non-minuscule. Recall the definition of $\wGr_{\leq \mu}$ from Lemma \ref{resolution qm},
\[\wGr_{\leq \mu}=Q_0\times^{Q_{\frac{1}{4}}}Q_{\frac{1}{2}}\times^{Q_{\frac{3}{4}}}Q_1/Q_1.\]
Then by the same argument as in Lemma \ref{soergel monoidal},
\[\on{H}_{L^+G}^*(\wGr_{\leq \mu})= R_{\bar L_{Q_{1/4}},\ell}\otimes_{R_{\bar L_{Q_{1/2}},\ell}} R_{\bar L_{Q_{3/4}},\ell}.\]
 The first $R_{\bar G,\ell}$-structure comes from the inclusion $R_{\bar G,\ell}=R_{\bar L_{Q_0},\ell}\subset R_{\bar L_{Q_{1/4}},\ell}$, and the second comes from the map $R_{\bar G,\ell}=R_{\bar L_{Q_1},\ell}\subset R_{\bar L_{Q_{3/4}},\ell}$. But as $R_{\bar G,\ell}$ is a subring of $R_{\bar L_{Q_{1/2}},\ell}$, these two $R_{\bar G,\ell}$ structures coincide.  It follows that the two $R_{\bar G,\ell}$-structures on $\on{IH}_{L^+G}(\Gr_{\leq \mu})=\on{H}^*_{L^+G}(\IC_{\mu}[-(2\rho,\mu)])$ also coincide, as it is direct summand of $\on{H}^*_{L^+G}(\wGr_{\leq \mu})$.

Next we consider the case when $\mu$ is minuscule. Note that the definition of $Q_r$ before Lemma \ref{resolution qm} in fact makes sense for every $\mu$. In particular, $Q_0=L^+G$, $Q_{1}=\varpi^\mu L^+G\varpi^{-\mu}$ and $Q_{\frac{1}{2}}=Q_1\cap Q_0$. Then
one can argue similarly to conclude that
\[\on{H}^*_{L^+G}(\Gr_\mu)= R_{\bar L_{Q_{1/2}},\ell},\]
with the two $R_{\bar G,\ell}$ structures given by $R_{\bar G,\ell}=R_{\bar L_{Q_0},\ell}\subset R_{\bar L_{Q_{1/2}},\ell}$ and $R_{\bar G,\ell}=R_{\bar L_{Q_1},\ell}\subset R_{\bar L_{Q_{1/2}},\ell}$, which clearly coincide.
\end{proof}

Note that there is a canonical isomorphism $\on{H}^*(\mA)=\Ql\otimes_{R_{\bar G,\ell}}\on{H}^*_{L^+G}(\mA),$ where $R_{\bar G,\ell}\to \Ql$ is via the augmentation map, again by an easy spectral sequence argument. 
Then combining the above two lemmas, we obtain the following statement.
\begin{prop}\label{monoidal functor}
The $L^+G$-equivariant hypercohomology functor 
$$\on{H}_{L^+G}^*(-):=\on{H}_{L^+G}^*(\Gr_G,-):\Sat_G\to \on{Proj}_{R_{\bar G,\ell}}$$ has a canonical monoidal structure, where $ \on{Proj}_{R_{\bar G,\ell}}$ denotes the tensor category of finite projective $R_{\bar G,\ell}$-modules. After base change along the augmentation map $R_{\bar G,\ell}\to \Ql$, the usual hypercohomology functor
$$\on{H}^*(-):=\on{H}^*(\Gr_G,-):\Sat_G\to \on{Vect}_{\Ql}$$ 
is a natural monoidal functor.
\end{prop}

\subsection{The commutativity constraints}\label{comm constr}
\subsubsection{}
In this subsection, we endow $\Sat_G$ with the commutativity constraints. The main statement is
\begin{prop}\label{existence of comm}
For every $\mA_1,\mA_2\in\Sat_G$, there exists a unique isomorphism $c_{\mA_1,\mA_2}:\mA_1\star\mA_2\simeq\mA_2\star\mA_1$ such that the following diagram is commutative
\[\begin{CD}
\on{H}^*(\mA_1\star\mA_2)@>\on{H}^*(c_{\mA_1,\mA_2})>>\on{H}^*(\mA_2\star\mA_1)\\
@V\simeq VV@VV\simeq V\\
\on{H}^*(\mA_1)\otimes\on{H}^*(\mA_2)@>\simeq>c_{\on{vect}}>\on{H}^*(\mA_2)\otimes\on{H}^*(\mA_1),
\end{CD}\]
where the vertical isomorphisms come from Proposition \ref{monoidal functor}, and the isomorphism $c_{\on{vect}}$ in the bottom row is the usual flip isomorphism between vector spaces.
\end{prop}
As $\on{H}^*:\Sat_G\to\on{Vect}_{\Ql}$ is faithful, the uniqueness of $c_{\mA_1,\mA_2}$ is clear. The content is its existence.
This proposition will be proved in the rest of the subsection. We first give its consequence.

\begin{cor}The monoidal category $\Sat_G$, equipped with the above constraints $c_{\mA_1,\mA_2}$, form a symmetric monoidal category. The hypercohomology functor $\on{H}^*$ is a tensor functor. 
\end{cor} 
\begin{proof}The proof of the first statement follows the idea of Ginzburg (cf. \cite{Gi1}). Namely, we need to check $c_{\mA_2,\mA_1}c_{\mA_1,\mA_2}=\id$ and the hexagon axiom.
Using the faithfulness of $\on{H}^*$, it is enough to prove these statements after taking the cohomology. Using Proposition \ref{existence of comm}, and the fact $c_{\on{vect}}^2=\id$, we conclude that $\on{H}^*(c_{\mA_2,\mA_1}c_{\mA_1,\mA_2})=\id$, and therefore $c_{\mA_2,\mA_1}c_{\mA_1,\mA_2}=\id$. The hexagon axiom can be proved similarly.
The second statement is clear.
\end{proof}

\subsubsection{}\label{functor Id'}
In order to construct $c_{\mA_1,\mA_2}$, we need some preparations. Define
\[\Gr_G^{\on{op}}:=L^+G\backslash LG,\]
on which  $L^+G$ acts by right multiplication. As before, sometimes we denote $\Gr^{\on{op}}_G$ by $\Gr^{\on{op}}$ for simplicity. Let $\on{P}_{L^+G}(\Gr_G^{\on{op}})$ denote the corresponding category of equivariant perverse sheaves. Note that $\on{P}_{L^+G}(\Gr_G^{\on{op}})$ also has a monoidal structure: There is the convolution Grassmannian $$\Gr^{\on{op}}\tilde\times\Gr^{\on{op}}:=L^+G\backslash LG\times^{L^+G} LG$$ equipped with $(m,\pr_2):\Gr^{\on{op}}\tilde\times\Gr^{\on{op}}\to \Gr^{\on{op}}\times \Gr^{\on{op}}$. Then for $\mA_1,\mA_2\in \on{P}_{L^+G}(\Gr_G^{\on{op}})$, one forms the twisted product $\mA_1\tilde\boxtimes\mA_2$ whose pullback along $\Gr^{\on{op}}\times LG\to \Gr^{\on{op}}\tilde\times\Gr^{\on{op}}$ is the pullback of $\mA_1\boxtimes\mA_2$ along $\Gr^{\on{op}}\times LG\to \Gr^{\on{op}}\times\Gr^{\on{op}}$, and forms the convolution product $\mA_1\star\mA_2=m_!(\mA_1\tilde\boxtimes\mA_2)$. For simplicity, we sometimes denote $(\on{P}_{L^+G}(\Gr^{\on{op}}_G),\star)$ by $\Sat_G^{\on{op}}$.

We have the following statements.
\begin{lem}
There is an equivalence of the monoidal categories
\[\Id':\Sat_G^{\on{op}}\simeq \Sat_G,\]
sending the intersection cohomology sheaf $\IC_{\mu}^{\on{op}}$ of $\Gr_{\leq \mu}^{\on{op}}$ to $\IC_\mu$, where $\Gr_{\leq\mu}^{\on{op}}$ is the closure of $\Gr_{\mu}^{\on{op}}=L^+G\backslash L^+G\varpi^\mu L^+G$. 
\end{lem}
Let $(LG)_{\leq \mu}$ denote the preimage of $\Gr_{\leq \mu}$ under the projection $LG\to \Gr$. Let $m$ be an integer large enough such that the $m$th congruence subgroup $L^+G^{(m)}$ is contained in  $L^+G\cap \varpi^{\mu}L^+G\varpi^{-\mu}$. Then we obtain the following diagram of surjective maps
\begin{equation}\label{left to right corr}
\Gr_{\leq \mu}\stackrel{\pi_m}{\leftarrow} \Gr_{\leq \mu}^{(m)}=(LG)_{\leq \mu}/L^+G^{(m)}\stackrel{\phi_m}{\to} L^+G\backslash (LG)_{\leq \mu}=\Gr_{\leq\mu}^{\on{op}}.
\end{equation}
\begin{lem}\label{can c}
There exists a unique isomorphism 
\[
\id'_\mu:\phi_m^*\IC_\mu^{\on{op}}\simeq \pi_m^*\IC_\mu
\]
of sheaves on $\Gr_{\leq \mu}^{(m)}$, whose restriction to $\Gr_\mu^{(m)}$ is given by 
$$\phi_m^*\IC_\mu^{\on{op}}|_{\Gr_{\mu}^{(m)}}=\Ql[(2\rho,\mu)]=\pi_m^*\IC_\mu|_{\Gr_\mu^{(m)}}.$$
In particular, $\phi_m^*\IC_\mu^{\on{op}}[m\dim G]$ is perverse.
\end{lem}
\begin{rmk}
(i) Informally, one can think both categories as certain category of $(L^+G\times L^+G)$-equivariant sheaves on $LG$. As we did not introduce sheaves on infinite-dimensional spaces, we give a concrete approach here.

(ii) As $\pi_m$ is an $L^mG$-torsor,  $\pi_m^*[m\dim G]$ preserves perversity. However, as we do not know whether $\phi_m$ is perfectly smooth, a priori it is not obvious that $\phi_m^*\IC_\mu^{\on{op}}[m\dim G]$ is perverse. On the other hand, as soon as the perversity of $\phi_m^*\IC_\mu^{\on{op}}[m\dim G]$ is known, the existence and the uniqueness of $\id'_\mu$ are clear.
\end{rmk}
\begin{proof}
We will prove these two lemmas simultaneously. First, if $\mu$ is minuscule, then $\Gr_{\leq \mu}$ and $\Gr_{\leq \mu}^{\on{op}}$ are perfectly smooth so there is a unique isomorphism 
$\id'_\mu: \phi_m^*\IC_\mu^{\on{op}}=\Ql[(2\rho,\mu)]= \pi_m^*\IC_\mu$
as required by Lemma \ref{can c}.

Next, if $\mu$ is quasi-minuscule but non-minuscule, let $\wGr_{\leq\mu}\to\Gr_{\leq \mu}$ denote the ``resolution" as constructed in Lemma \ref{resolution qm}. We can define
$$\wGr_{\leq \mu}^{\on{op}}=Q_0\backslash Q_0\times^{Q_{1/4}}Q_{1/2}\times^{Q_{3/4}}Q_1=Q_{1/4}\backslash Q_{1/2}\times^{Q_{3/4}}Q_1,$$ 
and the map 
$$\pi_\mu^{\on{op}}:\wGr_{\leq \mu}^{\on{op}}\to \Gr^{\on{op}}_{\leq \mu}, \ (g,g')\mapsto gg'\varpi^\mu,$$ which is a ``resolution" of $\Gr^{\on{op}}_{\leq \mu}$. We define $\wGr_{\leq \mu}^{(m)}$ by requiring that  both squares in the following diagram
\[\begin{CD}
\wGr_{\leq \mu}^{\on{op}}@<\tilde \phi_m<< \wGr_{\leq \mu}^{(m)}@>\tilde \pi_m>> \wGr_{\leq \mu}\\
@V\pi_\mu^{\on{op}}VV@VV\pi_\mu^{(m)}V@VV\pi_\mu V\\
\Gr_{\leq \mu}^{\on{op}}@<\phi_m<<\Gr_{\leq \mu}^{(m)}@>\pi_m>> \Gr_{\leq \mu}
\end{CD}\]
are Cartesian. Then we obtain the canonical isomorphisms
\[\phi_m^*\IC_\mu^{\on{op}}\oplus \phi_m^*\mC^{\on{op}}\simeq \phi_m^*(\pi_\mu^{\on{op}})_*\Ql[d]\simeq (\pi_\mu^{(m)})_*\Ql[d]\simeq \pi_m^*(\pi_\mu)_*\Ql[d]\simeq \pi_m^*\IC_\mu\oplus\pi_m^*\mC,\]
where $d=(2\rho,\mu)$, and $\mC$ and $\mC^{\on{op}}$ are as in the proof of Lemma \ref{min and qmin}. We therefore obtain $\id'_\mu$ as in Lemma \ref{can c}.

Now, let $\mmu\subset M$ as in \S~\ref{qm:conclusion}. Let $(m_1,\ldots,m_n)$ be a sequence of positive integers, such that $L^+G$ acts on $\Gr_{\leq \mu_i}^{(m_i)}$ via $L^+G\to L^{m_{i-1}}G$, and that $\phi_{m_i}:\Gr_{\leq \mu_i}^{(m_i)}\to \Gr_{\leq \mu_i}^{\on{op}}$ is defined. Then from the diagram
\[\begin{CD}\prod \Gr_{\leq \mu_i}^{\on{op}}@<\prod\phi_{m_i}<< \prod \Gr_{\leq \mu_i}^{(m_i)}@>\prod\pi_{m_i}>> \prod \Gr_{\leq \mu_i},\\
@.@ VVV @.\\
\Gr_{\leq \mmu}^{\on{op}}@<<< \Gr_{\leq \mu_1}\tilde\times\cdots\tilde\times\Gr_{\leq \mu_{n-1}}\tilde\times\Gr_{\leq\mu_n}^{(m_n)}@>>>\Gr_{\leq \mmu}.\\
@VVV@VVV@VVV\\
\Gr_{\leq |\mmu|}^{\on{op}}@<\phi_m<<\Gr_{\leq |\mmu|}^{(m)}@>\pi_m>>\Gr_{\leq |\mmu|}
\end{CD}\]
and the canonical isomorphism $\prod_i\id'_{\mu_i}: (\prod \phi_{m_i})^*(\boxtimes\IC^{\on{op}}_{\mu_i})\simeq (\prod \pi_{m_i})^*(\boxtimes\IC_{\mu_i})$, we obtain a canonical isomorphism
\[\id'_{\mmu}:\phi_m^*(\IC_{\mu_1}^{\on{op}}\star\cdots\star\IC_{\mu_n}^{\on{op}})\simeq \pi_m^*(\IC_{\mu_1}\star\cdots\star\IC_{\mu_n}).\]
By Lemma \ref{idempotent completion}, we conclude that for every $\mu$ the isomorphism $\id'_\mu$ as required in Lemma \ref{can c} exists. 

In addition, the isomorphism $\id'_{\mmu}$ also provides us the desired monoidal structure on $\Id'$. Again, by Lemma \ref{idempotent completion}, it is enough to exhibit the monoidal structure of $\Id'$ when restricted to the subcategories $\Id':\Sat_G^{0,\on{\on{op}}}\simeq\Sat_G^0$, where $\Sat_G^0$ is defined before Lemma \ref{idempotent completion} and $\Sat_G^{0,\on{\on{op}}}$ is defined similarly. For $\la_\bullet, \mmu\subset M$, we write $\IC_{\la_\bullet}=\IC_{\la_1}\star\cdots\star\IC_{\la_n}$, etc. Then there are canonical isomorphisms
\[\Hom(\IC_{\la_\bullet},\IC_{\mmu})\simeq \Hom(\pi_m^*\IC_{\la_\bullet},\pi_m^*\IC_{\mmu})\simeq \Hom(\phi_m^*\IC_{\la_\bullet}^{\on{op}},\phi_m^*\IC_{\mmu}^{\on{op}})\simeq\Hom(\IC_{\la_\bullet}^{\on{op}},\IC_{\mmu}^{\on{op}}),\]
which is clearly independent of $m$ (as soon as $m$ large enough). This isomorphism provides the monoidal structure on $\Id'$ as it is compatible with the union of sequences of coweights in $M$.
\end{proof}

We have the following corollary of Lemma \ref{can c}. 
For $\la\in \xcoch$, let $\mH^j_\la\mA$ (resp. $\mH^j_{\la,!}\mA$) denote the degree $j$ stalk (resp. costalk) cohomology of $\mA$ at $\varpi^\la$.
\begin{cor}\label{stalk identification}
There is a canonical isomorphism $\mH^j_\la \id': \mH^j_\la \mA\simeq \mH^j_\la\Id'\mA$ for $\mA\in\on{P}_{L^+G}(\Gr^{\on{op}})$, and similarly for $\mH^j_{\la,!}$
\end{cor}
\begin{proof}We prove the first statement as the second is obtained by the Verdier duality. 
It is enough to assume that $\mA=\IC_\mu^{\on{op}}$.
Then the isomorphism $\mH^j_\la \id'_\mu$ is given by the composition
\[\mH^j_\la\IC^{\on{op}}_\mu= \mH^j_\la\phi_m^*\IC^{\on{op}}_\mu\stackrel{\mH^j_\la\id'_\mu}{\simeq}\mH^j_\la\pi_m^*\IC_\mu=\mH^j_\la\IC_\mu,\]
which is clearly independent of the choice of $m$.
\end{proof}

\subsubsection{}\label{constr of comm}
Now, we construct $c_{\mA_1,\mA_2}$ as in Theorem \ref{existence of comm}. In the equal  characteristic situation, this is obtained from the fusion product interpretation of the convolution product (see \cite{MV} and \cite[\S 5.3]{BD}). Currently, the fusion product does not exist in mixed characteristic. Our method is a kind of categorification of classical Gelfand's trick (see also \cite[\S 5.3.8]{BD} which modifies the construction of \cite{Gi1}).

Fix a pinning $(G,B,T,X)$ of $G$ and let $\theta'$ be the involution that sends a dominant coweight $\la$ to its dual $\la^*=-w_0(\la)$, where $w_0$ is the longest element in the finite Weyl group $\overline{W}$ of $G$. We define the anti-involution $\theta$ of $G$ as $\theta(g)=\theta'(g)^{-1}$. It induces an anti-involution of $LG$ preserving $L^+G$,  which are still denoted by $\theta$ (rather than $L\theta$ if no confusion will arise). Note that $\theta$ induces an isomorphism
\[\theta: \Gr_G^{\on{op}}=L^+G\backslash LG\simeq LG/L^+G=\Gr_G,\]
and therefore an equivalence of categories
\[\theta^*: \on{P}_{L^+G}(\Gr_G)\simeq \on{P}_{L^+G}(\Gr_G^{\on{op}}).\]

Now $\theta$ also induces 
$$\theta\tilde\times\theta:\Gr^{\on{op}}\tilde\times\Gr^{\on{op}}\to\Gr\tilde\times\Gr,\quad (g_1,g_2)\mapsto (\theta(g_2),\theta(g_1)),$$ and there is a canonical isomorphism $(\theta\tilde\times\theta)^*(\mA_1\tilde\boxtimes\mA_2)\simeq \theta^*\mA_2\tilde\boxtimes\theta^*\mA_1$. Using $m(\theta\tilde\times\theta)=\theta m$,
and the proper base change, we obtain a canonical isomorphism 
$$\theta^*(\mA_1\star\mA_2)\simeq \theta^*\mA_2\star\theta^*\mA_1.$$
Considering the $3$-fold convolutions, we concludes that $\theta^*$ is an anti-equivalence of monoidal categories.

Therefore, we obtain an anti-autoequivalence $\Id'\circ \theta^*$ of $\Sat_G$ as a monoidal category.
Now we define an isomorphism of (plain) functors $$e:\Id'\circ\theta^*\to\Id.$$ We will fix a square root $\sqrt{-1}$ in $\Ql$ in the sequel and define $(-1)^{(\rho,\mu)}:=\sqrt{-1}^{(2\rho,\mu)}$ for any coweight $\mu$. By Lemma \ref{semisimplicity}, it is enough to give an isomorphism $e_\mu:\Id'\circ\theta^*\IC_\mu\to \IC_\mu$ for every $\mu$. Note that $\theta^*\IC_\mu$ is (non-canonically) isomorphic to $\IC_\mu^{\on{op}}$. We define the isomorphism 
\begin{equation}\label{Nmu}
N_\mu: \theta^*\IC_\mu\to \IC_\mu^{\on{op}}
\end{equation} by requiring its restriction to $\Gr_\mu^{\on{op}}$ is given by
$$\theta^*\IC_\mu|_{\Gr_\mu^{\on{op}}}=\IC_{\mu}|_{\Gr_\mu}=\Ql[(2\rho,\mu)]=\Ql[(2\rho,\mu)]=\IC_\mu^{\on{op}}|_{\Gr_{\mu}^{\on{op}}}.$$
We define $M_\mu=(-1)^{-(\rho,\mu)}N_\mu$ and let $e_{\mu}=\Id'(M_\mu)$. Let us emphasize that the factor $(-1)^{-(\rho,\mu)}$ is crucial.

Now, we define the isomorphism $c'_{\mA_1,\mA_2}$ as
\begin{equation}\label{def of comm}
c'_{\mA_1,\mA_2}:\mA_1\star\mA_2\stackrel{e_{\mA_1\star\mA_2}}{\longleftarrow} \Id'\theta^*(\mA_1\star\mA_2)\simeq \Id'(\theta^*\mA_2\star\theta^*\mA_1)\simeq \Id'\theta^*\mA_2\star\Id'\theta^*\mA_1\stackrel{e_{\mA_2}\star e_{\mA_1}}{\longrightarrow} \mA_2\star\mA_1.
\end{equation}
Note that $c'_{\mA_1,\mA_2}$ is independent of the choice of $\sqrt{-1}$.

Finally, the isomorphism $c_{\mA_1,\mA_2}$ is obtained from $c_{\mA_1,\mA_2}'$ by a Koszul sign change. Namely, the category $\on{P}_{L^+G}(\Gr)$ admits a $\bZ/2$-grading induced by \eqref{parity map}. We say $\mA$ has pure parity if $p(\on{Supp}(\mA))$ is $1$ or $-1$, in which case we define $p(\mA)=p(\on{Supp}(\mA))$. Then
\begin{equation}\label{def of comm 2}
c_{\mA_1,\mA_2}:=(-1)^{p(\mA_1)p(\mA_2)}c'_{\mA_1,\mA_2},
\end{equation}
if $\mA_1$ and $\mA_2$ have the pure parity $p(\mA_1)$ and $p(\mA_2)$.  See also \cite[\S 5.3.21]{BD} or \cite{MV} after Remark 6.2 for a more elegant formulation.

\subsubsection{}
We prove that $c_{\mA_1,\mA_2}$ constructed as above satisfies the requirement as in Proposition \ref{existence of comm}. From the definition, this will be the consequence of the following three statements Lemma \ref{aux1}-Lemma \ref{aux 3}. Recall that we set $\on{IH}_{L^+G}(\Gr_{\leq \mu})=\on{H}^*_{L^+G}(\Gr,\IC_{\mu}[-(2\rho,\mu)])$.

To state the first lemma, note that Lemma \ref{soergel monoidal} and Lemma \ref{left=right} hold for $\Sat_G^{\on{op}}$,
and therefore $\on{H}^*:\Sat^{\on{op}}\to\on{Vect}_{\Ql}$ has a natural monoidal structure.

\begin{lem}\label{aux1}
There is a natural isomorphism of monoidal functors $\ga:\on{H}^*\simeq \on{H}^*\circ\Id': \Sat_G^{\on{op}}\to \on{Vect}_{\Ql}$.
\end{lem}
\begin{proof}
It is enough to construct the canonical isomorphism $\ga_\mu:\on{IH}^*(\Gr_{\leq \mu}^{\on{op}})\simeq \on{IH}^*(\Gr_{\leq \mu})$ for every $\mu$. From the diagram \eqref{left to right corr}, we obtain a canonical isomorphism 
\[\on{IH}_{L^+G}^*(\Gr_{\leq \mu}^{\on{op}})\simeq \on{IH}^*_{L^+G\times L^+G}(\Gr_{\leq \mu}^{(m)})\simeq \on{IH}^*_{L^+G}(\Gr_{\leq \mu}).\]
as $(R_{\bar G,\ell}\otimes R_{\bar G,\ell})$-bimodules. Note that this is independent of the choice of $m$ (as soon as it is large). As 
$$\on{IH}^*(\Gr_{\leq \mu})= \Ql\otimes_{R_{\bar G,\ell}}\on{IH}^*_{L^+G}(\Gr_{\leq \mu}),\quad \on{IH}^*(\Gr_{\leq \mu}^{\on{op}})= \on{IH}^*_{L^+G}(\Gr_{\leq\mu}^{\on{op}})\otimes_{R_{\bar G,\ell}}\Ql,$$
we obtain the desired isomorphism $\ga_\mu$ by Lemma \ref{left=right}.
It follows  from the construction of the monoidal structure of $\on{H}^*$ given by Lemma \ref{soergel monoidal} and Lemma \ref{left=right} and the construction of the monoidal structure on $\Id'$ as in \S~\ref{functor Id'} that $\ga$ is an isomorphism of monoidal functors.
\end{proof}

The second lemma is as follows.
\begin{lem}\label{aux2}
There is a canonical isomorphism of functors $\delta:\on{H}^* \simeq \on{H}^*\circ \theta^*$, such that for every $\mA_1,\mA_2\in\on{P}_{L^+G}(\Gr)$, the following diagram is commutative
\[\begin{CD}
\on{H}^*(\mA_1\star\mA_2)@>\delta>>\on{H}^*(\theta^*(\mA_1\star\mA_2))@>\simeq>>\on{H}^*(\theta^*\mA_2\star\theta^*\mA_1)\\
@VVV@.@VVV\\
\on{H}^*(\mA_1)\otimes\on{H}^*(\mA_2)@>c_{\on{vect}}>>\on{H}^*(\mA_2)\otimes\on{H}^*(\mA_1)@>\delta\otimes\delta>>\on{H}^*(\theta^*\mA_2)\otimes\on{H}^*(\theta^*\mA_1).
\end{CD}\]
\end{lem}
\begin{proof}If $f: X\to Y$ is a morphism and $\mF$ a complex of sheaves on $Y$, there is a canonical map $f^*: \on{H}^*(Y,\mF)\to \on{H}^*(Y,f_*f^*\mF)\simeq \on{H}^*(X,f^*\mF)$. Applying this construction to $\theta: \Gr^{\on{op}}\simeq \Gr$ gives the isomorphism $\delta$. It remains to check the commutativity of the diagram.

We will use notations as in the proof of Lemma \ref{soergel monoidal}. So for a closed subset $Z\subset\Gr^{\on{op}}$, let $Z^{(m)}$ denote its preimage in $L^+G^{(m)}\backslash LG\to \Gr^{\on{op}}$. Note that the following diagram is commutative
\[\begin{CD}
\on{Supp}(\theta^*\mA_n)^{(m_n)}\times\cdots\times\on{Supp}(\theta^*\mA_1)^{(m_1)}@>\theta>>\on{Supp}(\mA)^{(m_1)}\times\cdots\times\on{Supp}(\mA_n)^{(m_n)}\\
@VVV@VVV\\
\on{Supp}(\theta^*\mA_i)@>\theta>>\on{Supp}(\mA_i).
\end{CD}\]
In addition, from the construction of the isomorphism in Lemma \ref{soergel monoidal}, the following diagram is also commutative
\[\begin{CD}
\on{H}^*_{L^+G}(\mA_1\star\cdots\star\mA_n)@>\simeq>>\on{H}^*_{L^+G}(\theta^*\mA_n\star\cdots\star\theta^*\mA_1)\\
@V\simeq VV@VV\simeq V\\
\on{H}^*_{L^+G}(\mA_1)\otimes_{R_{\bar G,\ell}}\cdots\otimes_{R_{\bar G,\ell}}\on{H}^*_{L^+G}(\mA_n)@>\simeq>>\on{H}^*_{L^+G}(\theta^*\mA_n)^{\on{op}}\otimes_{R_{\bar G,\ell}}\cdots\otimes_{R_{\bar G,\ell}}\on{H}^*_{L^+G}(\theta^*\mA_1)^{\on{op}},
\end{CD}\]
where for an $R_{\bar G,\ell}$-bimodule $M$, $M^{\on{op}}$ denotes the new $R_{\bar G,\ell}$-bimodule structure on $M$ by switching the two actions. Specializing along $R_{\bar G,\ell}\to \Ql$ shows that
$\delta$ is an isomorphism of monoidal functors.
\end{proof}

Now, for every $\mA\in \on{P}_{L^+G}(\Gr)$, we can define an automorphism of its cohomology
\[\Theta:\on{H}^*(\mA)\stackrel{\delta}{\simeq} \on{H}^*(\theta^*\mA)\stackrel{\ga}{\simeq} \on{H}^*(\Id'\circ\theta^*\mA)\stackrel{\on{H}^*(e)}{\simeq}\on{H}^*(\mA).\]
The above two lemmas imply that the following diagram
\[\begin{CD}
\on{H}^*(\mA_1\star\mA_2)@>\simeq>>\on{H}^*(\mA_1)\otimes\on{H}^*(\mA_2)@>c_{\on{vect}}>>\on{H}^*(\mA_2)\otimes\on{H}^*(\mA_1)\\
@V\Theta VV@.@VV\Theta\otimes\Theta V\\
\on{H}^*(\mA_1\star\mA_2)@>\on{H}^*(c'_{\mA_1\star\mA_2})>>\on{H}^*(\mA_2\star\mA_1)@>\simeq>>\on{H}^*(\mA_2)\otimes\on{H}^*(\mA_1)
\end{CD}\]
is commutative.
Now it is easy to see that Proposition \ref{existence of comm} is a consequence of the following lemma.
\begin{lem}\label{aux 3}
For every $j$, $\Theta=\sqrt{-1}^j$ on $\on{H}^{j}(\mA)$.
\end{lem}

Note that by the definition of $e$, and our normalization $\on{IH}^*(\Gr_{\leq \mu})=\on{H}^*(\IC_\mu[-(2\rho,\mu)])$, it is enough to show the following lemma.
\begin{lem}\label{involution action}
The map $$\Theta_\mu:\on{IH}^{2j}(\Gr_{\leq \mu})\stackrel{\on{H}(N_\mu)\delta}{\simeq} \on{IH}^{2j}(\Gr_{\leq \mu}^{\on{op}})\stackrel{\ga}{\simeq} \on{IH}^{2j}(\Gr_{\leq \mu})$$ is given by the multiplication by $(-1)^{j}$, where $N_\mu:\theta^*\IC_\mu\to \IC_\mu^{\on{op}}$ is the canonical isomorphism in \eqref{Nmu}.
\end{lem}
We do not know a direct proof of this lemma. In \cite{LY}, its equal  characteristic analogue was deduced from the equal  characteristic geometric Satake. They use this formula to deduce a numerical result for the affine Hecke algebra, as conjectured by Lusztig \cite{Lu2}. We will reverse their steps to deduce this lemma from this numerical result.  In the sequel, we follow the convention in literature to write $\on{H}(N_\mu)\delta$ as $\theta^*$. It should not be confused with the pullback of sheaves.
First note the following.
\begin{lem}\label{theta square}
The map $\Theta_\mu$ is an involution.
\end{lem}
\begin{proof}Choose some $m,m'$ such that the following diagram is commutative
\[\begin{CD}
\Gr_{\leq \mu}^{(m')}@>\theta^{-1}>> (\Gr_{\leq \mu}^{\on{op}})^{(m')}\\
@VVV@VVV\\
(\Gr_{\leq \mu}^{\on{op}})^{(m)}@>\theta>> \Gr_{\leq\mu}^{(m)}.
\end{CD}\]
Then  taking the $(L^+G\times L^+G)$-equivariant intersection cohomology and specializing along $R_{\bar G,\ell}\to \Ql$, wee obtain the following commutative diagram
\[\begin{CD}
\on{IH}^*(\Gr_{\leq \mu}^{\on{op}})@>(\theta^{-1})^*>>\on{IH}^*(\Gr_{\leq \mu})\\
@V\ga VV@AA\ga A\\
\on{IH}^*(\Gr_{\leq \mu})@>\theta^*>>\on{IH}^*(\Gr_{\leq\mu}^{\on{op}}).
\end{CD}\]
The lemma follows.
\end{proof}

To continue, let us understand a toy case. Note that $\theta$ induces an isomorphism between $L^+G$-orbits $\Gr_\mu^{\on{op}}\simeq \Gr_\mu$, and therefore we have a canonical isomorphism
\[\mathring{\Theta}_\mu:\on{H}^*(\Gr_\mu)\stackrel{\theta^*}{\simeq}\on{H}^*(\Gr_\mu^{\on{op}})\stackrel{\ga}{\simeq}\on{H}^*(\Gr_\mu),\]
where the isomorphisms $\ga$ is constructed by the same way as in Lemma \ref{aux1}. 
\begin{lem}\label{toy case}
For every $j$, $\mathring{\Theta}_\mu=(-1)^j$ on $\on{H}^{2j}(\Gr_\mu)$.
\end{lem}
\begin{proof}
The argument is essentially the same as \cite[Lemma 3.3]{LY}, although the set-up is different (the authors of \emph{loc. cit.} work over $\bC$ and with the based loop group of a compact Lie group rather that the affine Grassmannian).
Recall the projection $\pi_\mu:\Gr_\mu\to (\bar G/\bar P_{\mu})^\pf$ from \eqref{mod to fil}. 
As the fibers are the perfection of affine spaces of the same dimension, the pullback induces a canonical ring isomorphism
$\on{H}^*(\bar G/\bar P_{\mu})\simeq \on{H}^*(\Gr_\mu)$. Therefore, $\on{H}^*(\Gr_\mu)$ is generated by $\on{H}^2$. In addition, it is clear that $\mathring{\Theta}_\mu$ is a ring homomorphism so it is enough to prove that $\mathring{\Theta}_\mu=-1$ on $\on{H}^2$.

On the other hand $\Gr_\mu^{\on{op}}$ projects to $(\bar G/\bar P_{-\mu})^\pf$ given by $\varpi^{\mu}g\mapsto g^{-1} \mod \varpi$. 
A direct computation shows that the following diagram is commutative
\[\begin{CD}
\Gr^{\on{op}}_{\mu}@>\theta>> \Gr_\mu\\
@VVV@VVV\\
(\bar G/\bar P_{-\mu})^\pf @>g\mapsto \theta'(g) \dot{w}_0  >> (\bar G/\bar P_{\mu})^\pf,
\end{CD}\]
where $\dot{w}_0$ is a lifting of $w_0$ to $\bar G$.
Taking the equivariant cohomology, we obtain the commutativity of the following diagram 
\[\begin{CD}
\on{H}^*_{\bar G}(\bar G/\bar P_{\mu})=R_{(\bar P_{\mu})_{\on{red}},\ell}@>\theta^*>> R_{(\bar P_{-\mu})_{\on{red}},\ell}=\on{H}^*_{\bar G}(\bar G/\bar P_{-\mu})\\
@VVV@VVV\\
R_{\bar T,\ell}@>\chi\mapsto -\chi>> R_{\bar T,\ell},
\end{CD}\]
where $\chi\in \xch(\bar T)$, regarded as elements in $R_{\bar T,\ell}$ of degree two.

On the other hand, the isomorphism $\on{H}^*_{\bar G}(\bar G/\bar P_{-\mu})\simeq \on{H}^*_{\bar G}(\bar G/\bar P_{\mu})$, given by
$\ga:\on{H}^*_{L^+G}(\Gr_\mu^{\on{op}})\simeq \on{H}^*_{L^+G\times L^+G}(\Gr_\mu^{(m)})\simeq \on{H}^*_{L^+G}(\Gr_\mu)$, is the restriction of the identity map on $R_{\bar T,\ell}$ by definition. 
Therefore, the equivariant version of $\mathring{\Theta}_\mu$ acts as $(-1)$ on degree two parts. Specializing gives the lemma. 
\end{proof}
\begin{rmk}This lemma in particular proves Lemma \ref{involution action} in the case when $\mu$ is minuscule. The difficulty to prove Lemma \ref{involution action} for general $\mu$ is that the intersection cohomology ring is not generated by Chern classes but we do not know more cohomology classes in it\footnote{Although there are MV basis in $\on{IH}^*(\Gr_{\leq \mu})$, it seems hard to understand the map $\ga$ in terms of them.}.
\end{rmk}

To continue, it is convenient to set $\mathbf{C}_\mu=\IC_\mu[(2\rho,\mu)]$, as in \cite{LY}. For each $\la\leq\mu$, let $i_\la: \Gr_\la\to \Gr_{\leq \mu}$ denote the corresponding locally closed embedding. For $j$, let $\mH^{j}_\la\mathbf{C}_\mu$ denote the degree $j$th sheaf cohomology of $i_\la^*\mathbf{C}_\mu$, which is constant along $\Gr_\la$. Then there is a canonical isomorphism
\[\mH^{j}_\la\Psi_{\mu}: \mH^{j}_\la\mathbf{C}_\mu= \mH^{j}_\la\theta^*\mathbf{C}_\mu\simeq \mH^{j}_\la\mathbf{C}_\mu^{\on{op}}\simeq \mH^{j}_\la\mathbf{C}_\mu,\]
where the second isomorphism is $N_\mu: \theta^*\IC_\mu\simeq \IC_\mu^{\on{op}}$ from \eqref{Nmu}, and the last isomorphism is from Corollary \ref{stalk identification}. Clearly, $\mH^{j}_\la\Psi_\mu$ is an involution.
Recall that the existence of the Demazure ``resolution" in our setting (see \eqref{DR}) implies that all the stalk cohomology of $\mathbf{C}_\mu$ concentrate in even degrees. 
\begin{lem}\label{local involution}
For every $j$, $\mH^{2j}_\la\Psi_\mu=(-1)^j$.
\end{lem}

Now, we prove Lemma \ref{involution action}, assuming Lemma \ref{local involution}. In \cite[3.4, 6.4]{LY}, it was shown that the equal  characteristic analogue of Lemma \ref{involution action} implies the equal  characteristic analogue of Lemma \ref{local involution}. But their argument can be reversed. We sketch it here and refer to \emph{loc. cit.} for details (but note that their set-up is different). We extend the partial order $``\leq"$ on $\xcoch^+$ to a total order, still denoted by $\leq$. We consider the stratification of $\Gr_{\leq \mu}$ given by $\{\Gr_\la,\la\leq\mu\}$. Let $\Gr_{<\la}=\sqcup_{\la'<\la} \Gr_{\la'}$ and let $i_{<\la}$ and $i_{\leq \la}$ denote the corresponding closed embeddings from $\Gr_{<\la}$ and $\Gr_{\leq \la}$ to $\Gr_{\leq \mu}$. Then there is a long exact sequence of cohomology
\[\cdots\to \on{H}^i(\Gr_{<\la}, i^!_{<\la}\mathbf{C}_\mu)\to \on{H}^i(\Gr_{\leq \la},i_{\leq \la}^!\mathbf{C}_\mu)\to \on{H}^i(\Gr_\la,i_{\la}^!\mathbf{C}_\mu)\to\cdots,\]
which splits into short exact sequences as all the cohomology in odd degree vanish. Therefore, we obtain a filtration on $\on{IH}^*(\Gr_{\leq\mu})$, given by $\on{Im}(\on{H}^i(\Gr_{\leq \la},i_{\leq \la}^!\mathbf{C}_\mu)\to \on{IH}^*(\Gr_{\leq\mu}))$. The associated graded is $\oplus_{\la\leq \mu}\on{H}^i(\Gr_\la,i_{\la}^!\mathbf{C}_\mu)$.
There is a similar picture on $\Gr_{\leq\mu}^{\on{op}}$. 

The isomorphisms $\theta^*:  \on{IH}^*(\Gr_{\leq\mu})\simeq\on{IH}^*(\Gr_{\leq\mu}^{\on{op}})$ and $\ga:\on{IH}^*(\Gr_{\leq \mu})\simeq\on{IH}^*(\Gr_{\leq \mu})$ preserve the filtrations on $\on{IH}^*(\Gr_{\leq\mu})$ and on $\on{IH}^*(\Gr_{\leq\mu}^{\on{op}})$, and therefore give rise to isomorphisms 
$$\on{gr}\Theta:  \on{gr}\on{IH}^*(\Gr_{\leq\mu})\stackrel{\on{gr}\theta^*}{\simeq}\on{gr}\on{IH}^*(\Gr_{\leq\mu}^{\on{op}})\stackrel{\on{gr}\ga}{\to}\on{gr}\on{IH}^*(\Gr_{\leq \mu}).$$
Note that  $i_\la^!\mathbf{C}_\mu= (i_\la^*\mathbf{C}_\mu[2(2\rho,\la-\mu)])^*$. In addition, it is easy to identify $\on{gr}\Theta$ with the direct sum over $\la$ of the maps
\[ \mathring{\Theta}_{\la}\otimes \mH^*_{\la,!}\Psi_\mu:\on{H}^*(\Gr_\la)\otimes i_\la^!\mathbf{C}_\mu \simeq \on{H}^*(\Gr_\la)\otimes i_\la^!\mathbf{C}_\mu,\]
where $\mH^*_{\la,!}\Psi_\mu$ is the inverse of the dual of $\mH^*_{\la}\Psi_\mu$.
So the action of $\on{gr}\Theta$ on the degree $2j$ piece of $\on{H}^*(\Gr_\la)\otimes\mH_\la^*\mathbf{C}_\mu$ is given by $(-1)^j$. But as $\Theta$ itself is an involution, it acts on $\on{IH}^{2j}(\Gr_{\leq \mu})$ by $(-1)^j$.

\subsubsection{}\label{comb formula}
It remains to prove Lemma \ref{local involution}. 
Let $I$ be the ($p$-adic jet group of) the standard Iwahori (i,e. whose reduction mod $\varpi$ is $\bar B\subset \bar G$). Let $\widetilde W$ denote the Iwahori-Weyl group of $G(F)$ as before, and let $W_a\subset \widetilde W$ denote the corresponding affine Weyl group, with the set of simple reflections $\{s_i,i\in\bS\}$ determined by $I$. We identify $\bS$ with the set of vertices of the affine Dynkin diagram of $G(F)$. Let $0\in \bS$ denote the vertex corresponding to the hyperspecial parahoric $G(\mO)$. Let $J=\bS-\{0\}$, and let $W_J\subset \widetilde W_a$ denote subgroup generated by $\{s_i,\ i\in J\}$. Let $w_J$ denote the longest element in $W_J$. Then $W_J$ is isomorphic to the finite Weyl group $\overline W=\widetilde W/\xcoch$ of $G(F)$, and $w_J$ maps to the longest element $w_0$ in $\overline W$ mentioned before.
Let $\Omega\subset \widetilde W$ denote the subgroup of length zero elements, i.e. those that fix $I$. It acts on $W_a$ by conjugation. Then $\widetilde W=W_a\rtimes\Omega$. Let $*:\widetilde W\to\widetilde W$ be the involution given by $w^*:= w_Jw w_J$ for $w\in W_J$ and $\la^*=-w_0(\la)$ for $\la\in\xcoch$. This is an involution of $\widetilde W$ which stabilizes $\{s_i, i\in\bS\}$ and fixes $s_0$.

Let $\omega\in \Omega$.  Then by \cite[Lemma 6.2]{LY} $\omega^*=\omega^{-1}$, and the map
$$\diamond: W_a\to W_a, \quad w\mapsto w^{\diamond}:=\omega w^* \omega^{-1}$$
is an involution of $W_a$, which stabilizes $\{s_i,i\in\bS\}$. Let $I_{\diamond}=\{w\in W_a\mid w^{\diamond}=w^{-1}\}$, and $W_J^{\diamond}=\{w^{\diamond}\mid w\in W_J\}=\omega W_J\omega^{-1}$. Then as argued in \cite[Proposition 8.2]{Lu2} and \cite[Theorem 6.3 (1)]{LY}, the longest element in every $(W_J\times W_J^{\diamond})$-double coset belongs to $I_{\diamond}$. 

Applying the results of \cite{LV,Lu2} to $(W_a,\{s_i,i\in\bS\},\diamond)$, one attaches a polynomial $P_{y,w}^{\sigma,\diamond}(q)\in \bZ[q]$ to every pair $(y,w)\in I_{\diamond}$, with $y\leq w$. On the other hand, there is the usual Kazhdan-Lusztig polynomial $P_{y,w}(q)$ attached to $(y,w)$ \cite{KaLu}.
The following theorem was conjectured in \cite[Conjecture 8.4]{Lu2}, and was proved in \cite[Theorem 6.3]{LY}.
\begin{thm}\label{-q analog}
Let $d_1$ and $d_2$ be longest elements of $(W_J,W_J^{\diamond})$-double cosets in $W_a$. Then
\[P_{d_1,d_2}^{\sigma,\diamond}(q)=P_{d_1,d_2}(-q).\]
\end{thm}
Let us note that this theorem was deduced in \cite{LY} from the equal  characteristic analogue of Lemma \ref{toy case}.

Finally, we explain why Lemma \ref{local involution} follows from this theorem.  Let $\mu\in\xcoch$, and let $\omega$ denote the unique element in $\Omega$ such that $\varpi^\mu\in W_a\omega $. Let $d_\mu$ be the longest element in $W_J\varpi^\mu W_J\omega^{-1}=W_J (\varpi^{\mu}\omega^{-1}) W_J^{\diamond}$.  Then for $\la\leq \mu$, $\varpi^\la\omega^{-1}\in W_a$. Let $d_\la$ denote the corresponding longest element in $W_J \varpi^{\la}\omega^{-1}W_J^\diamond$. The usual Kazhdan-Lusztig theory \cite{KaLu,KaLu2} works in our situation. So $P_{y,w}(q)$ is the Poincare polynomial for the stalk cohomology at $y$ of the intersection cohomology sheaf $\on{IC}_w$ of the Schubert variety $S_w$ on the affine flag variety $\Fl=LG/I$. Then in particular (see \cite{Lu}),
\[P_{d_\la,d_\mu}(q)=\sum( \dim \mH^{2j}_\la\mathbf{C}_\mu) q^j.\]
On the other hand, in \cite[\S 3]{LV}, similar interpretations were given to the polynomials $P_{y,w}^{\sigma,\diamond}$. Such interpretations in particular imply that
\[P^{\sigma,\diamond}_{d_\la,d_\mu}(q)=\sum \tr(\mH_\la^{2j}\Psi_\mu\mid \mH^{2j}_\la\mathbf{C}_\mu)q^j.\]
Since $\mH_\la^{2j}\Psi_\mu$ is an involution, Theorem \ref{-q analog} implies Lemma \ref{local involution}.

\subsection{Identification with the dual group}
We have endowed $\on{P}_{L^+G}(\Gr)$ with a symmetric monoidal category structure and the hypercohomology functor $\on{H}^*:\on{P}_{L^+G}(\Gr)\to \on{Vect}_{\Ql}$ a tensor functor structure. It is clear that $\IC_0$ is a unit object in $\on{P}_{L^+G}(\Gr)$. Now we proceed as in \cite[\S 7]{MV} to conclude that $(\on{P}_{L^+G},\star,\on{H}^*)$ is a Tannakian category with the fiber functor $\on{H}^*$. Let $\tilde G=\Aut^\otimes \on{H}^*$ denote the Tannakian group. It is a connected reductive group, by the same argument as in \cite[\S 7]{MV}. Our next goal is to identify $\tilde G$ with the dual group $\hat G$ of $G$. 

First, if $G=T$ is a torus, $\Gr_T$ is a discrete set of points canonically isomorphic to $\xcoch(T)$. Then it is easy to see that $\Sat_T$ is equivalent to the category of $\xcoch(T)$-graded finite dimensional $\Ql$-vector spaces and $\on{H}^*$ is just the functor that forgets the grading. Therefore, $\tilde T=\hat T$ is the dual torus of $T$.

Now consider the general case. We can regard the weight functor
$\on{CT}$ as a functor from $\Sat_G\to \Sat_T$, and the isomorphism in Corollary \ref{weight functor} as an isomorphism $\on{H}^*\circ \on{CT}\simeq \on{H}^*: \Sat_G\to \on{Vect}_\Ql$.
\begin{prop}
There is a unique monoidal structure on $\on{CT}$ such that the isomorphism $\on{H}^*\circ \on{CT}\simeq \on{H}^*: \Sat_G\to \on{Vect}_\Ql$ in Corollary \ref{weight functor} is monoidal. 
\end{prop}
In equal characteristic, this was proved in \cite[Proposition 6.4]{MV} using the fusion product interpretation of the convolution product. However, there is another purely local approach using equivariant cohomology, given in \cite[Proposition 5.3.14]{Z16}. The latter approach works in mixed characteristic as well. Namely, $S_\la$ is stable under the action of the torus $\bar T^\pf\subset L^+T\subset L^+G$, and therefore one can use the $\bar T$-equivariant cohomology $\on{H}_{\bar T}^*$ as in \emph{loc. cit.} for the arguments.

Applying Proposition \ref{existence of comm}, we see that the weight functor $\on{CT}$ in fact respects to the symmetric monoidal structure, and thus is a tensor functor between two Tannakian categories. It thus induces a homomorphism
\[\hat{T}\simeq \tilde T\to \tilde G.\]
This defines a subtorus of $\tilde G$. By the same argument as in \cite[\S~7]{MV}, this is in fact a maximal torus. In addition, the filtration on $\on{H}^*(\Gr_G,-)$ defines a Borel subgroup $\hat B\subset\tilde G$ that contains $\hat T$. Then it follows by the same argument as the end of \cite[\S~7]{MV} that $\tilde G$ is isomorphic to $\hat G$. We refer to \cite[\S~5.3]{Z16} for more details.

\begin{rmk}Our methods can also be applied to establish the mixed characteristic geometric Satake for ramified groups (cf. \cite{Z11}).
\end{rmk}

\section{Dimension of affine Deligne-Lusztig varieties}\label{dim of RZ}
In this section, we give an application of  mixed characteristic affine Grassmannians to the study of the Rapoport-Zink (RZ) spaces. More applications will appear in \cite{XZ}.

\subsection{Dimension of affine Deligne-Lusztig varieties}\label{dim:ADLV}
\subsubsection{}
We use the notations as in \S\ \ref{not}. So $F$ is a totally ramified extension of $F_0=W(k)[1/p]$ with $\mO$ its ring of integers. Let $L$ be the completion of its maximal unramified extension, with $\mO_L$ its ring of integers. Let $\sigma\in \Gal(L/F)$ denote the Frobenius element. Let $G$ be a reductive group scheme over $\mO$. For $b\in G(L)$ and $\mu\in\xcoch^+$, we define the (closed) affine Deligne-Lusztig ``variety" as  
\begin{equation}\label{ADL}
X_{\leq \mu}(b)=\{g \mod L^+G\in \Gr_G\mid g^{-1}b\sigma(g)\in \overline{L^+G\varpi^\mu L^+G}\}.
\end{equation}
More precisely, one can interpret $X_{\leq \mu}(b)$ as the following moduli functor: let $\mE_0$ be the trivial $G$-torsor on $D_F=\Spec \mO$, with an isomorphism $b: \sigma^*\mE_0|_{D^*_F}\to \mE_0|_{D^*_F}$. Then for a perfect $k$-algebra $R$,
\begin{equation}\label{ADL2}
X_{\leq \mu}(b)(R)=\{(\mE,\beta)\in \Gr_G(R)\mid \inv_x(\beta^{-1}b\sigma(\beta))\leq \mu,\ \forall x\in\Spec R\}.
\end{equation}
By Lemma~\ref{Hodge stra2}, $X_{\leq \mu}(b)$ is a closed subset of $\Gr_G$.
One can replace ``$\leq$" in the above definition by ``$=$", which defines an open subset of $X_{\leq \mu}(b)$, denoted by $X_\mu(b)$.
If we denote $\Phi=\beta^{-1}b\sigma(\beta)$, then $(\mE,\Phi)$ is an $F$-crystal with $G$-structure on $\Spec R$, whose Hodge polygon is bounded by $\mu$ (resp. equal to $\mu$).

It turns out that the dimension of $X_{\leq \mu}(b)$ is finite, and Rapoport gave a conjectural formula of its dimension (\cite{R}) with a reformulation given by Kottwitz (\cite{GHKR})
\begin{equation}\label{dim ADL}
\dim X_{\leq \mu}(b)= \langle \rho,\mu-\nu_b\rangle-\frac{1}{2}\on{def}_G(b).
\end{equation}
Here $\nu_b$ is the Newton point of $b$ and $\on{def}_G(b)$ is the defect of $b$. We refer to \cite{GHKR} for the precise definitions.
This dimension formula has been proved in equal  characteristic by combining the works \cite{GHKR,V,Ham2}, but remains open in general in mixed characteristic. In fact, before our work, it is not clear how to define the dimension of $X_{\leq \mu}(b)$ in mixed characteristic in general, and this formula only makes sense for some special triples $(G,b,\mu)$ when \eqref{ADL} can be interpreted as the $\bar\bF_p$-points of some moduli spaces of $p$-divisible groups (a.k.a. RZ spaces). In the case when the RZ spaces are of PEL type, this dimension formula was proved recently by Hamacher (\cite{Ham1}) and some special cases were proved earlier by Viehmann  (\cite{Vi1,Vi2}).

\begin{thm}\label{Vi}
Rapoport's conjecture \eqref{dim ADL} holds in general.
\end{thm}

Not surprisingly, the machinery developed so far in the paper allows us to imitate the
the arguments in equal  characteristic with only a few justifications. 
First, one can argue as in
 \cite{GHKR,Ham2} to reduce the general Rapoport conjecture to the case when $b$ is superbasic. It was shown in \cite{GHKR,CKV} that if $G$ is of adjoint type, superbasic $\sigma$-conjugacy classes exist only when $G_F=\on{PGL}_n$ or $G_F=\Res_{E/F}\on{PGL}_n$, where $E/F$ is an unramified extension. The $\on{PGL}_n$ case was treated by Viehmann \cite{V} (in equal characteristic but the same arguments apply here).  We will reduce the $\Res_{E/F}\on{PGL}_n$ case to the $\on{PGL}_n$ case and then apply \cite{V}. This in particular gives a shorter proof of the main result of \cite{Ham2} (but it uses \cite{V}).  We sketch the arguments in the sequel. 
 
\begin{rmk}\label{rmk CKV}
This is a side remark arising as a comment by G. Pappas. 
Although the algebro-geometric structure on $X_{\leq\mu}(b)$ was not known before, the authors of \cite{CKV} defined a notion of the set of connected components $\pi_0(X_{\leq\mu}(b))$ of $X_{\leq\mu}(b)$. One can check that if two points in $X_{\leq\mu}(b)(\bar k)$ are in the same connected component in the sense of \emph{loc. cit.}, they are in the same connected component under the Zariski topology. The converse will also hold if in their definition arbitrary test rings (rather than just smooth rings) are allowed\footnote{In \emph{loc. cit.}, it was conjectured that these two definitions coincide.}. On the other hand, it seems that one can directly adapt their arguments to our setting to prove that the structure of connected components of $X_{\leq \mu}(b)$ in our sense is also given by the statement of \cite[Theorem 1.1]{CKV}. Then it would follow a posteriori that the two notions are the same.
In any case, when $X_{\leq\mu}(b)$ is the set of $\bar{\bF}_p$-points of a Rapoport-Zink space, their $\pi_0$ coincides with the $\pi_0$ of the RZ space, and by Proposition \ref{ADLV=RZ} below, also coincides with $\pi_0$ of $X_{\leq\mu}(b)$ as the perfection of an algebraic space.
\end{rmk}

\subsubsection{}  
Now one can argue as in \cite[Proposition 5.6.1, Theorem 5.8.1]{GHKR} to reduce the Rapoport conjecture for general $(G,\mu,b)$ to the case when $b$ is basic. First, the Newton point $\nu_b$ is defined over $F$, whose centralizer in $G$ is a rational Levi $M$. One can find a representative in the $\sigma$-conjugacy class of $b$ that is contained in $M(L)$. We rename this representative by $b$. So $b$ is basic in $M(L)$. Then their arguments reduce Rapoport's conjecture for $(G,b,\mu)$ to $(M,b,\mu_M)$ (for various $\mu_M$). These arguments rely on their Proposition 5.3.1 and 5.4.3. The proof of Proposition 5.3.1 in \emph{loc. cit.} applies to the current setting. Note that the arguments involve an $M$-equivariant isomorphism $N\simeq \frakn$. In the equal  characteristic situation, this isomorphism makes sense either as $F$-schemes of as $k$-ind-schemes. In our setting, it only makes sense as an isomorphism of $F$-schemes. But it still makes sense to talk about the $p$-adic loop space of $\frakn$ so the arguments in \S 4 of \emph{ibid.} apply. The proof of Proposition 5.4.3 in \emph{ibid.} extends verbatim in mixed characteristic, by taking account of the Lefschetz trace formula for separated pfp perfect algebraic spaces (see \S~\ref{trace formula}). A special case of this type of argument has appeared in the proof of Proposition \ref{semi coh} (where $M=T$). 

As explained in \emph{loc. cit.}, even $b$ is basic for $G$, it still might happen that $b$ is contained in a proper Levi subgroup of $G$. A basic $\sigma$-conjugacy class that does not meet in proper Levi subgroups of $G$ defined over $F$ is called a \emph{superbasic} $\sigma$-conjugacy class. Therefore, it is enough to prove Rapoport's conjecture for superbasic $b$. In addition, one can assume that $G=G_\ad$ is simple of adjoint type. Then it follows from \cite{GHKR,CKV} that superbasic $b$ exists only when $G_F=\Res_{E/F}\on{PGL}_n$ for some unramified extension $E/F$.

\subsubsection{}
It remains to prove the following.
\begin{prop}\label{unramified to split}
Formula \eqref{dim ADL} holds for $G_F=\Res_{E/F}\on{GL}_{n}$ and $b$ superbasic.
\end{prop}

\begin{rmk}
This proposition was proved by Hamacher when $F=\bQ_p$ and $\mu$ is minuscule. Our method is different and is simpler, but it uses \cite{V}.
\end{rmk}

\begin{proof}
We first reduce the $\Res_{E/F}\GL_n$ case to $\GL_n$ case. 

We start with a generalization of affine Deligne-Lusztig varieties. Let $H$ be a connected reductive group over $\mO_E$. First observe that $X_{\leq \mu}(b)$ can be defined as the following Cartesian pullback 
\begin{equation}\label{ADLV car}
\begin{CD}
X_{\leq \mu}(b)@>>> \Gr_H\tilde\times\Gr_{\leq \mu}\\
@VVV@VV\pr\times mV\\
\Gr @>1\times b\sigma >> \Gr_H\times \Gr_H.
\end{CD}
\end{equation}
Now by replacing $\mu$ by a sequence of dominant coweights $\mmu$, we can define a convolution version of the affine Deligne-Lusztig variety
\begin{equation}\label{ADLV conv car}
\begin{CD}
X_{\leq \mmu}(b)@>>> \Gr_H\tilde\times\Gr_{\leq \mmu}\\
@VVV@VV\pr_1\times mV\\
\Gr @>1\times b\sigma >> \Gr_H\times \Gr_H.
\end{CD}
\end{equation}
Concretely, $X_{\leq \mmu}(b)$ classifies the following commutative diagram of maps
\[\xymatrix{
\sigma^*\mE_1\ar^{\Phi_{d}}[r]\ar_{\sigma^*(\beta)}[d]&\mE_{d}\ar^{\Phi_{d-1}}[r]&\cdots\ar^{\Phi_1}[r]&\mE_1\ar^\beta[d]\\
\sigma^*\mE_0\ar^{b}[rrr]&&&\mE_0,
}\]
such that $\inv_x(\Phi_i)\leq \mu_i$ for every $x\in\Spec R$.
Note that in equal  characteristic, this is the local version of the moduli space of iterated Shtukas.

In the sequel, we write $X_{\leq \mu_\bullet}^H(b\sigma_E)$ for $X_{\leq \mu_\bullet}(b)$ if we want to emphasize that the underlying group $H$ is define over $\mO_E$, and that the Frobenius $\sigma_E\in \Gal(L/E)$.

\medskip

Now we start our reduction step. Assume that $E/F$ is unramified of degree $d$.
Let $\Sigma$ denote the set of embeddings $\tau:E\to L$ over $F$. Then $\Gal(L/F)$ acts transitively on $\Sigma$. We fix $\tau_0\in\Sigma$ and let $\tau_i=\sigma^i(\tau_0),\ i=0,1,2,\ldots,d-1$. Let $\sigma_E:=\sigma^d\in \Gal(L/\tau_0(E))$.

Now assume that $G=\Res_{\mO_E/\mO_F}H$, for some unramified group $H$ over $E$. The canonical isomorphism $E\otimes_FL\simeq \prod_{\Sigma} L, \ a\otimes b\mapsto (\tau_i(a)b,\tau_i\in\Sigma)$ induces a canonical isomorphism 
$$G\otimes L\simeq \prod_{\tau\in\Sigma}H\otimes_{E,\tau}L.$$
Let $\mu$ be a dominant coweight of $G_L$. Then under the above isomorphism, it gives a sequence  $\mmu=(\mu_{\tau_0},\ldots,\mu_{\tau_{d-1}})$, where $\mu_{\tau_i}$ is a dominant coweight of $H\otimes_{E,\tau_i}L$. Similarly, $b\in G(L)$ gives $(b_{\tau})\in\prod_{\tau\in\Sigma}(H\otimes_{E,\tau}L)(L)$.

Note that $\sigma^i\in \Gal(L/F)$ induces an isomorphism $H\otimes_{E,\tau_i}L\simeq H\otimes_{E,\tau_0}L$. By abuse of notations, the induced map on the cocharacters and on the $L$-points are still denoted by $\sigma^i$ (this coincides with the standard notation if $H=(H_0)_E$ for some group $H_0$ defined over $F$).

For an $\bar\bF_p$-algebra $R$, we identify $(\mE,\beta)\in \Gr_G$ with $(\mE_\tau,\beta_\tau)\in\prod_{\tau\in\Sigma}\Gr_{H}$ in an obvious way. Then the condition \eqref{ADL2} is equivalent to the commutativity of the following diagram
\[\begin{CD}(\sigma^d)^*\mE_{\tau_0}@>>> (\sigma^{d-1})^*\mE_{\tau_{d-1}}@>>>\cdots @>>>\mE_{\tau_0}\\
@V(\sigma^d)^*\beta_{\tau_0}VV@V(\sigma^{d-1})^*\beta_{\tau_{d-1}}VV@.@VV\beta_{\tau_0}V\\
(\sigma^d)^*\mE_0@>\sigma^{d-1}(b_{\tau_{d-1}})>>(\sigma^{d-1})^*\mE_0@>\sigma^{d-2}(b_{\tau_{d-2}})>>\cdots @>b_{\tau_0}>>\mE_0.
\end{CD}\]
Let
$$\on{Nm}b=b_{\tau_0}\sigma(b_{\tau_1})\cdots \sigma^{d-1}(b_{\tau_{d-1}})\in (H\otimes_{E,\tau_0}L)(L).$$ 
Then the above discussions imply the following lemma.
\begin{lem}
If $G=\Res_{\mO_E/\mO_F}H$ for some unramified group $H$ over $\mO_E$, then 
$$X^G_{\leq \mu}(b\sigma)\simeq X^H_{\leq \mu_\bullet}((\on{Nm}b)\sigma_E),$$ where $\mu_\bullet=(\mu_{\tau_0},\sigma(\mu_{\tau_1}),\ldots,\sigma^{d-1}(\mu_{\tau_{d-1}}))$.
\end{lem}
\begin{rmk}The map $b\mapsto \on{Nm}b$ defines a map from the $\sigma$-conjugacy class of $G(L)$ to the $\sigma_E$-conjugacy class $H(L)$ (where $E$ embeds into $L$ via $\tau_0$).
\end{rmk}

We also need the following purely group theoretical lemma, whose proof is by chasing the definitions.
\begin{lem}\label{change of group}
Let $\mu$ and $b$ be as above. 

(1) Let $\rho_G$ be the half sum of positive roots of $G\otimes L$ and let $\rho_H$ be the half sum of positive roots of $H\otimes_{E,\tau_0} L$. Then 
$$(\rho_G,\mu)=(\rho_H, \sum_i \sigma^{i}(\mu_{\tau_{i}})).$$ 

(2) Let $\nu_b$ be the Newton point of $b$ and let $\nu_{\on{Nm}b}$ be the Newton point of $\on{Nm}b$. Then $$(\rho_G,\nu_b)=(\rho_H,\nu_{\on{Nm}b}).$$

(3) Let $J^G_b$ be the $\sigma$-twisted centralizer of $b\in G(L)$, i.e. 
$$J^G_b(R)=\{g\in G(R\otimes_FL)\mid g^{-1}b\sigma(g)=b\}$$ for any $F$-algebra $R$. This is an $F$-group. Similarly, let $J^H_{\on{Nm}b}$ be the $\sigma_E$-twisted centralizer of $\on{Nm}b$, which is an $E$-group. Then $J_b^G=\Res_{E/F}J_{\on{Nm}b}^H$. In particular, 
$$\on{def}_G(b)=\on{def}_H(\on{Nm}b).$$
\end{lem}

Now, assuming that the dimension formula for affine Deligne-Lusztig varieties of $H$ has been established, we 
calculate the dimension of $X^G_{\leq \mu}(b\sigma)=X^{H}_{\leq \mmu}((\on{Nm}b)\sigma_E)$. Recall the convolution map of affine Grassmannians \eqref{conv prod} (for $H$)
\[m:\Gr_{\leq\mmu}\to\Gr_{\leq|\mmu|}.\]
By \eqref{ADLV car} and \eqref{ADLV conv car}, the following diagram is Cartesian
\[\begin{CD}
X_{\leq\mmu}^H((\on{Nm}b)\sigma_E)@>>>\Gr_H\tilde\times\Gr_{\leq\mmu}\\
@VVV@VVV\\
X_{\leq |\mmu|}^H((\on{Nm}b)\sigma_E)@>>>\Gr_H\tilde\times\Gr_{\leq |\mmu|}.
\end{CD}\]
By (the proof of) Proposition \ref{semismall}, for $\la$, the dimension of the fiber $m^{-1}(\varpi^\la)$ 
is $\leq (\rho_H,|\mmu|-\la)$. Therefore, by Lemma \ref{change of group}, the preimage of $X_{\leq \la}^{H}((\on{Nm}b)\sigma_E)\subset X_{\leq |\mmu|}^H((\on{Nm}b)\sigma_E)$ in $X_{\leq\mmu}^H((\on{Nm}b)\sigma_E)$ has dimension
\[\begin{array}{ll}\leq &(\rho_H,\la-\nu_{\on{Nm}b})-\frac{1}{2}\on{def}_{H}(\on{Nm}b)+(\rho_H,|\mmu|-\la)\\=&(\rho_H,|\mmu|-\nu_{\on{Nm}b})-\frac{1}{2}\on{def}_{H}(\on{Nm}b)\\=&(\rho_G,\mu-\nu_b)-\frac{1}{2}\on{def}_G(b).\end{array}\]
In addition, if $\la=|\mmu|$, the equality achieves.
It follows that
$$\dim X_{\leq\mu}^G(b\sigma)=\dim X_{\leq\mmu}^{H}((\on{Nm}b)\sigma_E)=(\rho_G,\mu-\nu_b)-\frac{1}{2}\on{def}_G(b).$$

Therefore, it remains to prove the case when $G=\GL_n$ and $b$ superbasic. Now one can argue exactly the same as \cite{V} to complete the proof.
\end{proof}

\subsection{Affine Deligne-Lusztig varieties and Rapoport-Zink spaces}\label{ADLV==RZ}

Let us recall the definition of Rapoport-Zink (RZ) spaces. In the PEL case, they were defined by Rapoport-Zink in their original work \cite{RZ}. In a more general situation but under the assumption that the group is unramified, they are recently defined by Kim \cite{Ki} and Howard-Pappas \cite{HP}. We assume (for simplicity) that $k=\bar\bF_p$ is algebraically closed. To follow the standard notation, we write $W=W(k)$ (which was usually denoted by $\mO$ in previous sections). Let $L=W\otimes\bQ_p$. We use $F$ to denote $\sigma$-linear maps between vector spaces over $L$ (unlike the rest part of the paper where $F$ denotes a local field). Let $\on{Nilp}_W$ denote the category of $W$-algebras in which $p$ is nilpotent. 

First we recall the following fundamental result of Rapoport-Zink. Let $\bX_0$ be a $p$-divisible group over $k$. We consider the functor $\breve\mM_{\bX_0}$ that associates every $R\in \on{Nilp}_W$ the groupoid of pairs $(\bX,\iota)$, where $\bX$ is a $p$-divisible group over $\Spec R$, and $\iota:\bX_0\otimes_k R/p\to \bX\otimes_R R/p$ is a quasi-isogeny.  Rapoport-Zink proved that $\breve\mM_{\bX_0}$ is represented by a separated formal scheme, formally smooth and formally locally of finite type over $W$.

Now we start with a reductive group $G$ over $\bZ_p$, a geometric conjugacy class of cocharacters $\mu: \bG_m\to G$, and a $\sigma$-conjugacy class $b$ of $G(L)$ with a representative in $G(W)p^{\mu} G(W)$, still denoted by $b$.
We assume that there exists a free $\bZ_p$-lattice $\La$\footnote{Our $\La$ corresponds to $\La^*$, and $\mu$ corresponds to $-\mu$ in \cite{Ki}.} and a faithful representation
\[\rho: G\to \GL(\La),\]
such that the cocharacter $\rho\mu:\bG_m\to \GL(\Lambda\otimes W)$ has weights $0,1$. We fix a representative  of $\mu$, still denoted by the same notation. Let
\[\La\otimes W=\La^0\oplus \La^{1}\]
denote the decomposition of $\La\otimes W$ according to the weights of $\rho\mu$, which in turn induces a filtration $\on{Fil}^0(\La\otimes W)=\La\otimes W\supset \on{Fil}^1(\La\otimes W)=\La^{1}$.
We assume that $\on{rk} \La^{1}=n$, and $\on{rk} \La=h$. This is equivalent to assuming that $\rho\mu$ is the $n$-th fundamental coweight of $\GL(\La\otimes W)$. 

Let $\La^\otimes$ denote the tensor algebra of $\La\oplus \La^*$. Note that $\La^\otimes=(\La^*)^\otimes$. Elements in $\La^\otimes$ are called tensors. We choose a finite collection of tensors $\{s_i \in \La^\otimes, i\in I\}$ such that
$G\subset \GL(\La)$ is the schematic stabilizer of this collection. I.e. 
$$G=\Aut(\La,\{s_i, i\in I\}).$$ For example, if $G=\GL_h$, we can choose $\{s_i\}$ to be the empty set. Note that $$P_\mu:=\Aut(\La,\{s_i, i\in I\},\on{Fil}^*(\La\otimes W))$$ is a parabolic subgroup of $G_W$ determined by $\mu$.

Note that by our assumption and the classical Dieudonn\'{e} theory, there exists a $p$-divisible group $\bX_0$ of dimension $n$ and height $h$ over $\bar\bF_p$, together with an isomorphism
\[\varepsilon: \bD(\bX_0)\simeq \La\otimes_{\bZ_p} W,\]
where $\bD(\bX_0)$ is the \emph{contravariant} Dieudonn\'{e} module of $\bX_0$, equipped with $(F,V)$, such that: 
\begin{enumerate}
\item $\varepsilon F= \rho(b)(\id_\La\otimes\sigma) \varepsilon$; 
\item $\varepsilon(\Lie \bX_0)^*= \on{Fil}^1\La\otimes \bar\bF_p$. 
\end{enumerate}
The pair $(\bX_0,\varepsilon)$ is unique up to a unique isomorphism and we fix it in the sequel.

Finally \cite{Ki}, we define (crystalline)-Tate tensors for $p$-divisible groups.
Let $R\in \on{Nilp}_W$ and a $p$-divisible group $\bX$ on $\Spec R$. Let $\bD(\bX)$ denote its \emph{contravariant} Dieudonn\'{e} crystal. This is an $F$-crystal on $\Spec R$, by which we mean
a locally free crystal $\mE$ on the big crystalline site $\on{CRIS}(R/W)$, with a map (the Frobenius map) 
$$F: \sigma^*\mE\to \mE,$$ such that there exist an integer $i\geq 0$ and $V: \mE\to \sigma^*\mE$ satisfying $VF=p^i$. In addition,
there is a decreasing filtration $\on{Fil}^\bullet\bD(\bX)_R$ on $\bD(\bX)_R$ (the value of $\bD(\bX)$ at the trivial PD-thickening $R\stackrel{\id}{\to}R$) whose associated graded is locally free over $R$. Namely, 
$$\on{Fil}^0\bD(\bX)_R=\bD(\bX)_R,\quad \on{Fil}^1\bD(\bX)_R=(\Lie \bX)^*,\quad \mbox{and } \on{Fil}^2\bD(\bX)_R=0.$$ Note that $\bD(\bX)^\otimes$ is also an $F$-crystal with a filtration $\on{Fil}^\bullet \bD(\bX)_R^\otimes$. For example, let 
$$1:=\bD(\bQ_p/\bZ_p)$$ 
be the filtered $F$-crystal given by the Dieudonn\'{e} module of the constant $p$-divisible group $\bQ_p/\bZ_p$. Then $1_{R'}=R'$ for every PD-thickening $R'\to R$ and $F:1_{R'}\to 1_{R'}$ sending $F(1)=1$. In addition, $\on{Fil}^11_{R}=0$. Then we call a (crystalline-)Tate tensor of $\bX$ a morphism $t:1\to \bD(\bX)^\otimes$ of crystals, such that $t_R:1_R\to \bD(\bX)^\otimes_R$ is compatible with the filtrations, and such that the induced map $t: 1\to \bD(\bX)^\otimes[\frac{1}{p}]$ of isocrystals is Frobenius-invariant.

For example, we can interpret $\{s_i, i\in I\}\subset \La^\otimes$ as Tate tensors of the above fixed $p$-divisible group $\bX_0$ as follows.
First, via $\varepsilon$, we can regard $\{s_i\}$ as tensors in $\bD(\bX_0)^\otimes$. Since $G$ fixes $\{s_i\}$, $b\sigma$ fixes $\{s_i\}$. So $\{s_i\}$ are $F$-invariant in $\bD(\bX_0)^\otimes[\frac{1}{p}]$. In addition, the cocharacter $\rho\mu: \bG_m\to\GL(\La\otimes W)$ also fixes $\{s_i\}$. Therefore, $\{s_i\}$ are in $\on{Fil}^0(\bD(\bX_0)_{\bar \bF_p}^\otimes)$. Then we can define $s_i:1\to \bD(\bX_0)^\otimes$ by sending $1$ to $s_i$.

For a $p$-divisible group over a general base $R$, the notion of Tate tensors may not be well-behaved. Following \cite{Ki},
let $\on{Nilp}_W^{\on{sm}}$ denote the full subcategory of $\on{Nilp}_W$ consisting of formally smooth formally finitely generated $W/p^m$-algebras for some $m>0$. 
\begin{dfn}\label{dfn of RZ}
The RZ space associated to $(G,b,\mu)$ is the functor $\breve\mM(G,b,\mu)$ on $\on{Nilp}^{\on{sm}}_W$ classifying: for every $R\in\on{Nilp}^{\on{sm}}_W$, 
\begin{enumerate}
\item a $p$-divisible group $\bX$ on $\Spec R$;
\item a collection of cyrstalline-Tate tensors $\{t_i\},i\in I$ of $\bX$;
\item a quasi-isogeny $\iota:\bX_0\otimes_k R/J\to \bX\otimes_R R/J$ that sends $t_i$ to $s_i\otimes 1$ for $i\in I$, where $J$ is some (and therefore any) ideal of definition of $R$ that contains $p$.
\end{enumerate}
such that
\begin{enumerate}
\item[(*)] the $R$-scheme $$\on{Isom}((\bD(\bX)_R, \{t_i\},\on{Fil}^\bullet(\bD(\bX)_R)),(\La\otimes_{\bZ_p}R,\{ s_i\otimes 1\},\on{Fil}^\bullet\La\otimes_{\bZ_p}R))$$ 
that classifies the isomorphisms between locally free sheaves $\bD(\bX)_R$ and $\La\otimes_{\bZ_p}R$ on $\Spec R$ preserving the tensors and the filtrations
is a $(P_\mu\otimes_W R)$-torsor\footnote{Recall that $P_\mu$ is the automorphism group scheme of $(\La,\{ s_i\},\on{Fil}^\bullet\La)$.}.
\end{enumerate}
\end{dfn}
\begin{rmk}(i) Our definition is slightly different from the original definition given in \cite{Ki}. But it is not hard to see that  Condition (*) combines the Item (2) and (3) in \cite[Definition 4.6]{Ki}.

(ii) As explained in \cite{Ki}, $\breve\mM(G,b,\mu)$  is independent of the choice of $\rho:G\to \GL(\La)$ up to isomorphism.
\end{rmk}

The main theorem of Kim \cite{Ki} (see also \cite[Theorem 3.2.1, \S 2.4]{HP}) is as follows.
\begin{thm}\label{representability of RZ}
Assume that $p>2$. Then 
\begin{enumerate}
\item $\breve\mM(G,b,\mu)$ is represented by a closed  formal subscheme $\breve\mM(G,b,\mu)\subset \breve\mM_{\bX_0}$, formally smoothly over $W$.  
\item If $G=\GL_h$, $\mu=\omega_n$ and $\rho=\id$, Then $\breve\mM(G,b,\mu)= \breve\mM_{\bX_0}$.
\item There is a canonical bijection $\breve\mM(G,b,\mu)(k)\simeq X_\mu(b)(k)$ compatible with the embeddings $\breve\mM(G,b,\mu)\to \breve\mM(\GL_h,\rho(b),\omega_n)$ and $X_\mu(b)\to X_{\rho\mu}(\rho(b))$.
\end{enumerate}
\end{thm}

We write $\overline\mM_\mu(b)$ for the special fiber of $\breve\mM(G,b,\mu)$.
\begin{prop}\label{ADLV=RZ}
 Fixing $(\bX_0,\varepsilon)$ as above. There is a canonical isomorphism
$X_\mu(b)\simeq \overline{\mM}_\mu(b)^\pf$. In particular, $\dim \overline\mM_\mu(b)_{\on{red}}=\dim X_{\mu}(b)$.
\end{prop}
Note that this proposition in particular describes the values of $\breve\mM(G,b,\mu)$ on a perfect ring $R$ (which is not obvious from Definition \ref{dfn of RZ}).
\begin{proof}
We first prove the proposition for $(G,b,\mu)=(\GL_h,b,\omega_n)$. 
The key input is a theorem of Gabber (see also \cite[\S 6]{La}) on $p$-divisible groups over perfect rings. We write $\overline{\mM}^\pf$ instead of $\overline\mM_\mu(b)^\pf$ for simplicity.

We first construct $X_\mu(b)\to \overline{\mM}^\pf$ as follows. Let $R$ be a perfect $\bar\bF_p$-algebra. We write $\sigma:R\to R$ for the Frobenius automorphism.
Let $(\mE,\beta)\in X_\mu(b)(R)$. We obtain a crystal $\bD:=\mE\times^{G,\rho}\La$ over $R$ ($=$ a locally free sheaf on $W(R)$), with $F= \beta^{-1}(b\sigma)\beta$. I.e., the following diagram is commutative
\begin{equation}\label{Fcry}
\begin{CD}
\bD[\frac{1}{p}]@>F>>\bD[\frac{1}{p}]\\
@V\beta VV@V\beta VV\\
W(R)\otimes_{\bZ_p}\La[\frac{1}{p}]@>b\sigma>>W(R)\otimes_{\bZ_p}\La[\frac{1}{p}].
\end{CD}
\end{equation}

Note that $\omega_n:\bG_m\to \GL(\La)$ is the $n$th fundamental coweight of $\GL(\La)$, and therefore the condition in \eqref{ADL2} together with Lemma \ref{minu} implies that $F:\sigma^*\bD\to \bD$ is regular and the quotient is a locally free module of rank $n$ over $W(R)/p=R$. Applying Lemma \ref{minu} again to the quasi-isogeny (of crystals) $V=pF^{-1}:\bD\to \sigma^*\bD$, we see that $V$ is regular. Therefore, there is a $\sigma^{-1}$-linear map $V:\bD\to \bD$ such that $FV=VF=p$. By Gabber's theorem (see also \cite[\S 6]{La}), there is a $p$-divisible group $\bX$ on $\Spec R$ together with a quasi-isogeny $\iota:\bX\to (\bX_0)_R$ such that $\bD=\bD(\bX)$ and the induced map $\bD(\iota):\bD[\frac{1}{p}]\to W(R)[\frac{1}{p}]\otimes \La$ is $\beta$. 

Conversely, we construct $\overline{\mM}^\pf\to X_{\mu}(b)$ as follows. Let $R$ be a perfect $\bar\bF_p$-algebra.
Let $(\bX,\iota)$ be an object in $\overline{\mM}(R)$. Then we have the $\GL(\La)$-torsor
\[\mE=\on{Isom}(\bD(\bX), \La \otimes_{\bZ_p}W(R)).\]
The quasi-isogeny $\iota$ defines an isomorphism
\[\bD(\bX)[\frac{1}{p}]\simeq \La\otimes_{\bZ_p}W(R)[\frac{1}{p}],\]
and therefore defines a quasi-isogeny $\beta:\mE\dasharrow \mE_0$. The map $(\bX,\iota)\mapsto (\mE,\beta)$ defines a map $\overline\mM^\pf\to \Gr_{\GL_h}$. It is clear that  $\inv_x(F)\leq \omega_n$ for every $x\in\Spec R$, and therefore $(\mE,\beta)\in X_\mu(b)$. We thus defines a map $\overline\mM^\pf\to X_\mu(b)$. 

Now for general $(G,b,\mu)$, $\overline{\mM}_\mu(b)^\pf\subset \overline{\mM}_{\rho\mu}(\rho(b))^\pf$ and $X_\mu(b)\subset X_{\rho\mu}(\rho(b))$. To prove that $X_\mu(b)\simeq \overline{\mM}_\mu(b)^\pf$, it is enough to show that the above isomorphism sends $X_\mu(b)(k)\subset X_{\rho\mu}(\rho(b))(k)$ to $\overline{\mM}_\mu(b)^\pf(k)\subset \overline{\mM}_{\rho\mu}(\rho(b))^\pf(k)$. But this follows from Theorem \ref{representability of RZ} (3).
\end{proof}

\begin{cor}
Rapoport's conjecture of the dimension formula holds for the reduced schemes of the Rapoport-Zink spaces.
\end{cor}

\appendix
\section{Generalities on perfect schemes}\label{generality of perf}
This section can be regarded as a user's guide to the algebraic geometry of perfect schemes and perfect algebraic spaces, which is the setting we work with in the paper. We include some discussions more general than needed in the paper. The main result is Theorem \ref{quotient}, which explains the construction of the quotients in this setting. 

\subsection{Perfect schemes and perfect algebraic spaces}
\subsubsection{}\label{spaces}
We fix a field $k$. Let $\on{Aff}_k$ denote the category of affine $k$-schemes, i.e. the category opposite to the category $k\on{-alg}$ of $k$-algebras. Following \cite{LMB, BL}, we call a sheaf on $\on{Aff}_k$ with respect to the fpqc topology a $k$-\emph{space}. Explicitly, a space $\mF$ is a covariant functor $k\on{-alg}\to \on{Set}$ that respects finite products, and such that if $R\to R'$ is faithfully flat, then the sequence
\begin{equation}\label{fpqc sheaf}
\mF(R)\to \mF(R')\rightrightarrows\mF(R'\otimes_RR') 
\end{equation}
is an equalizer. Morphisms between two spaces are natural transformations of functors. The category of $k$-spaces is denoted by $\on{Sp}_k$. It contains the category $\on{Sch}_k$ of $k$-schemes as a full subcategory. Recall that a map $f:\mF\to\mG$ in $\on{Sp}_k$ is called schematic if for every scheme $T$, the fiber product $\mF\times_{\mG}T$ is a scheme.

In this paper, we need to consider a subcategory of $\on{Sp}_k$ larger than $\on{Sch}_k$.
Recall that an algebraic space is an \emph{\'{e}tale} sheaf $X$ on $\on{Aff}_k$ such that: 
(i) $X\to X\times X$ is schematic ; (ii) There exists an \'{e}tale surjective map $U\to X$, where $U$ is a scheme.

We denote by $\on{AlgSp}_k$ the category of algebraic spaces. We have full embeddings $\on{Sch}_k\subset\on{AlgSp}_k\subset\on{Sp}_k$, where the second inclusion follows from a theorem of Gabber which says algebraic spaces are fpqc sheaves (see \cite[Tag03W8]{St}).

\begin{rmk}
In literature as \cite{Kn, LMB}, it sometimes requires that $X$ is quasi-separated, i.e. the diagonal $X\to X\times X$ is quasi-compact. We prefer not to make this additional assumption.
\end{rmk}

A map $f:\mF\to \mG$ of $k$-spaces is called representable if for every affine scheme $T$, $\mF\times_{\mG}T$ is represented by an algebraic space. It is call \emph{fpqc} if in addition $\mF\times_{\mG}T$ is faithfully flat over $T$, and there is a quasi-compact open subset $U$ of $\mF\times_{\mG}T$ that maps surjectively to $T$. Recall that fpqc maps are effective epimorphisms in $\on{Sp}_k$. I.e. if $U\to X$ is an fpqc map of spaces, then for every space $\mF$, the following diagram is an equalizer
\begin{equation}\label{fpqc descent}
\Hom_{\on{Sp}_k}(X,\mF)\to \Hom_{\on{Sp}_k}(U,\mF)\rightrightarrows \Hom_{\on{Sp}_k}(U\times_XU,\mF).
\end{equation}
In particular, any fpqc sheaf on $\on{Aff}_k$ extends uniquely to an fpqc sheaf on $\on{Sch}_k$ (although we do not use the latter in this paper).

\subsubsection{}\label{pf sp}
From now on we assume that $k$ is a perfect field of characteristic $p>0$. For a $k$-algebra $R$, let $\sigma:R\to R,\ r\mapsto r^p$ denote the Frobenius map. 
Recall that $R$ is called \emph{perfect} if $\sigma$ is an isomorphism. The forgetful functor from the category of perfect $k$-algebras to the category of all $k$-algebras admits a left adjoint 
$$R\mapsto R^\pf=\underrightarrow{\lim}_{\sigma} R.$$
Sometime, $R^\pf$ is called the perfection (or the perfect closure) of $R$.

These facts admit the following globalization. A $k$-scheme (resp. algebraic space) $X$ is called \emph{perfect} if its Frobenius endomorphism $\sigma_X: X\to X$ is an isomorphism. We write $\sigma$ for $\sigma_X$ if no confusion will likely arise. The category of perfect schemes (resp. perfect algebraic spaces) over $k$ is denoted by $\on{Sch}_k^{\on{pf}}$ (resp. $\on{AlgSp}_k^{\on{pf}}$). Then the embedding $\on{AlgSp}_k^{\on{pf}}\to \on{AlgSp}_k$ admits a right adjoint. To see this, first note that Frobenius endomorphisms commute with \'{e}tale localizations.
\begin{lem}\label{Frob etale loc}
For any \'{e}tale morphism of algebraic spaces $X\to Y$, the relative Frobenius morphism $X\to X\times_{Y,\sigma_Y}Y$ induced by $\sigma_X$ is an isomorphism.
\end{lem}
\begin{proof}We first assume that $X$ is a scheme. Then $X\to X\times_{Y,\sigma_Y}Y$ a schematic radical \'{e}tale surjective map, and therefore is an isomorphism by \cite[Theorem 17.9.1]{EGA4-4}. 
For general $X$, choose an \'{e}tale cover $U\to X$ by a scheme $U$. Then we have $U\to U\times_{X,\sigma_X}X\to U\times_{Y,\sigma_Y}Y$, with the first map and the composition map being isomorphisms. Therefore, the second map is an isomorphism as well. Note that $U\times_{X,\sigma_X}X\to U\times_{Y,\sigma_Y}Y$ is nothing but the base change of $X\to X\times_{Y,\sigma_Y}Y$ along the \'{e}tale cover $U\times_{Y,\sigma_Y}Y\to Y$. Therefore, $X\to X\times_{Y,\sigma_Y}Y$ is also an isomorphism.
\end{proof}
\begin{cor}The embedding $\on{AlgSp}_k^{\on{pf}}\to \on{AlgSp}_k$ admits a right adjoint functor, given by $X\to X^\pf=\underleftarrow\lim_{\sigma}X$. 
\end{cor}
We call $X^\pf$ the \emph{perfection} of $X$. 
\begin{proof}Applying Lemma \ref{Frob etale loc} to the \'{e}tale cover $U\to X$, we see that $\sigma:X\to X$ is an affine morphism. Then the diagonal of $X^\pf=\underleftarrow\lim_{\sigma}X$ is representable, and $U^\pf\simeq U\times_XX^\pf\to X^\pf$ is an \'{e}tale cover. Therefore $X^\pf$ is a perfect algebraic space. 

It remains to show that  the tautological map $\varepsilon:X^\pf\to X$ induces an isomorphism
\begin{equation}\label{perfection}
\Hom(Y,X^\pf)=\Hom (Y,X),
\end{equation} 
for every perfect $k$-algebraic space $Y$. But by Lemma \ref{Frob etale loc} and \eqref{fpqc descent}, we reduce to the known case where $X,Y$ are affine. 
\end{proof}
\begin{rmk}\label{top of perfection}
Recall that $\sigma_X$ is a universal homeomorphism if $X$ is a scheme, and therefore is a universal homeomorphism if $X$ is an algebraic space by Lemma  \ref{Frob etale loc}. Therefore, $\varepsilon:X^\pf\to X$ is a universal homeomorphism.  It also follows that $X$ is a scheme if and only if $X^\pf$ is a scheme. See Lemma \ref{perfection of morphism} below.
\end{rmk}

The following statement is crucial for later applications. For an algebraic space $X$, we denote by $X_{et}$ its small \'{e}tale site with objects being algebraic spaces \'{e}tale over $X$.
\begin{prop}\label{et top}
Let $X$ be an algebraic space over $k$ and let $X^\pf$ denote its perfection. There the functor $(U\to X)\mapsto (U^\pf\simeq U\times_XX^\pf\to X^\pf)$ induces an equivalence of \'{e}tale sites $X_{et}\simeq X^\pf_{et}$, and therefore the \'{e}tale topos $\varepsilon^*: \widetilde X_{et} \simeq \widetilde X_{et}^\pf: \varepsilon_*$.
\end{prop}
\begin{proof}
First, assume that $X$ is a scheme. Then $X^\pf\to X$ is a universal homeomorphism. Therefore $(U\to X)\mapsto (U^\pf\simeq U\times_XX^\pf\to X^\pf)$ induces an equivalence of subcategories of scheme objects in $X_{et}$ and $X^\pf_{et}$ (cf. \cite[Tag04DZ]{St}). Then an argument similar to \cite[Proposition A.1.3]{CLO} shows that it induces a full equivalence $X_{et}\simeq X^\pf_{et}$. Again, by a similar argument as \cite[Proposition.A.1.3]{CLO}, the case
when $X$ is an algebraic space also follows. 
\end{proof}
\begin{rmk}
It follows from Remark \ref{top of perfection} that the equivalence preserves the subcategories of scheme objects in $X_{et}$ and $X^\pf_{et}$.  
\end{rmk}

We list a few properties of morphisms that are preserved after passing to the perfection.
\begin{lem}\label{perfection of morphism}
Let $f:X\to Y$ be a morphism of algebraic spaces over $k$, and let $f^\pf:X^\pf\to Y^\pf$ denote its perfection. The following properties hold for $f$ if and only if the same hold for $f^\pf$:
(1) quasi-compact; (2) quasi-separated; (3) (universally) homeomorphic; (4) (universall) closed; (5) separated; (6) affine; (7) integral.

In addition, if $f$ is either: 
(8) \'{e}tale; or (9) (faithfully) flat; (10) fpqc, 
so is $f^\pf$.
\end{lem}
\begin{proof}(1)-(4) are clear since $\varepsilon: X^\pf\to X$ is a universal homeomorphism. 

For (5), first note that if $X^\pf\to X^\pf\times_{Y^\pf}X^\pf$ is a closed embedding, it is universally closed and therefore $\Delta_{X/Y}: X\to X\times_{Y}X$ is also universally closed. But $\Delta_{X/Y}$ is always a separated, locally of finite type monomorphism. Therefore it is a closed embedding. The inverse direction is clear. 

For (6),  we can assume that $Y$ is an affine scheme by Lemma \ref{Frob etale loc}. Then if $X$ is affine, so is $X^\pf$. Conversely, if $X^\pf$ is affine, then it is quasi-compact and separated, and so is $X$. As $X^\pf=\underleftarrow\lim_\sigma X$, $X$ is affine by \cite[Tag07SE, Lemma 5.8]{St}.

For (7), first note that if $f^\pf$ is integral, it is affine and therefore $f$ is affine by (6). We reduce to show that if $A^\pf\to B^\pf$ is integral, so is $A\to B$. Given $b\in B$, there is a monic polynomial $g(x)\in A^\pf[x]$ such that $g(b)=0$. Then there is some $n$ large enough such that all coefficients of $h(x):=g(x)^{p^n}$ are in $A$. Since $h(b)=0$, $b$ is integral over $A$. The inverse direction is clear.

(8) follows from Lemma \ref{Frob etale loc}. 

For (9), we may assume that $f:X\to Y$ is a morphism between affine schemes, and therefore is given by a ring homomorphism $f:R\to R'$. In addition, we may assume that $R$ is perfect. As $(R')^\pf=\underrightarrow\lim_\sigma R'$, it is enough to show that the composition $R\to R'\stackrel{\sigma^n}{\to} R'$ is flat. But this map is the same as $R\stackrel{\sigma^n}{\simeq} R\to R'$ and therefore is flat. Finally, (10) follows from (1) and (9).
\end{proof}

We will also consider the perfection of certain pro-algebraic spaces. Let $\{X_i\}$ be  a projective system of algebraic spaces, with the transition maps $X_{i+1}\to X_i$ being affine. Then the pro-algebraic space $X=\underleftarrow\lim X_i$ is also an algebraic space and it is easy to show that
\begin{equation}\label{perfection of pro scheme}
X^\pf\simeq \underleftarrow\lim X_i^{\pf},
\end{equation}
i.e. ``the perfection commute with inverse limits".

\subsubsection{}\label{torsor}
Let $H$ be an affine group scheme over $k$, regarded as a group object in $\on{Sp}_k$. An $H$-\emph{torsor} over a space $X$ is a space $E$ with a free $H$-action and an $H$-equivariant fpqc map $\pi:E\to X$ (where $X$ is endowed with the trivial $H$-action), such that the natural map $E\times H\to E\times_X E$ is an isomorphism. If $X$ is an algebraic space over $k$, then $E$ is represented by an algebraic space, affine over $X$. In addition, we have the following lemma.
\begin{lem}
If $X$ and $H$ are perfect, then $E$ is perfect. 
\end{lem}
\begin{proof}
Note that $E^\pf\to X^\pf$ is also an fpqc map by Lemma \ref{perfection of morphism}, and therefore is an $H^\pf$-torsor over $X^\pf$. In addition, $E^\pf\to E\times_{X,\varepsilon}X^\pf$ is a morphism of $H^\pf$-torsors, where $H^\pf$ acts on $E$ through the morphism $\varepsilon:H^\pf\to H$. Therefore, $E^\pf\simeq E\times_XX^\pf\simeq E$ is an isomorphism.
\end{proof}

Now, let $H'$ be a smooth affine group scheme over $k$, and let $H={H'}^\pf$ denote its perfection. Then $H$ is an affine group scheme. 
\begin{lem}\label{equiv of torsors}
Let $X$ be a perfect algebraic space. 
Then the functor $E'\to {E'}^\pf$ is an equivalence of categories between the groupoid of $H'$-torsors on $X$ and the groupoid of $H$-torsors on $X$. The quasi-inverse functor is given by push-out of an $H$-torsor along $\varepsilon: H\to H'$, denoted by $E\mapsto E\times^{H,\varepsilon}H'$. 
\end{lem}
\begin{proof}Given an $H$-torsor $E$, the natural map $E\to E\times^{H,\varepsilon}H'$ gives a morphism $E\to (E\times^{H,\varepsilon}H')^\pf$ of $H$-torsors, and therefore is an isomorphism. Conversely, let $E'$ be an $H'$-torsor on $X$, and let $E=E^\pf$ be the corresponding $H$-torsor. We want to show that $E\times^{H,\varepsilon}H'\simeq E'$. As $H'$ is smooth, we can trivialize $E'$ by an \'{e}tale cover $U\to X$ and therefore $E'$ can be represented by a cocycle $c':U\times_XU\to H'$. As $U\times_XU$ is perfect by Lemma \ref{Frob etale loc}, the cocycle $c'$ gives a cocycle $c:U\times_XU\to H$ by \eqref{perfection}, which is nothing but the cocycle representing $E$. Then $E\times^{H,\varepsilon}H'$ is represented by the cocycle $U\times_XU\stackrel{c}{\to}H\stackrel{\varepsilon}{\to}H'$, which is exactly $c'$.
\end{proof}

\medskip

For a space $X$ with an action by an affine group scheme $H$, we denote by $[X/H]$ the quotient stack (in fpqc topology) whose $R$-points are the groupoid of pairs $(E,\phi)$, where $E$ is an $H$-torsor on $\Spec R$, and $\phi:E\to Y$ is an $H$-equivariant morphism. Note that if the action is free, then $[X/H]$ is a $k$-space and the natural morphism $X\to [X/H]$ is an $H$-torsor. In this case, we write $[X/H]$ by $X/H$ for simplicity.

We also recall the construction of the twisted product. Let $H$ be an affine group scheme and $E\to X$ an $H$-torsor, and let $T$ be a space with an $H$-action. Then one can form the twisted product 
\begin{equation}\label{twisted product}
X\tilde\times T:= E\times^H T= E\times T/H,
\end{equation}
which is a space over $k$. Now assume that $H$ is (the perfection of) an affine group scheme of finite type over $k$. We claim that if $X, T$ are (perfect) algebraic spaces, so is $X\tilde\times T$. Indeed, we can find an fppf cover $U\to X$ that trivializes $E$\footnote{If $H$ is of finite type, this is clear. Otherwise, use Lemma \ref{et triv}.}. Then $U\times T$ is an fppf cover of $X\tilde\times T$. Therefore, $X\tilde\times T$ is an algebraic space by \cite[Tag04S5]{St}.

\subsubsection{}\label{nonsense pf sp}
In fact in the paper we will only consider presheaves on the category of perfect $k$-algebras.

\begin{dfn}
Let $\on{Aff}^{\on{pf}}_k$ be the opposite category of perfect $k$-algebras. A perfect space is a sheaf on $\on{Aff}^{\on{pf}}_k$ with respect to the fpqc topology. The category of perfect $k$-spaces is denoted by $\on{Sp}^{\on{pf}}_k$.
\end{dfn}
 There is a natural functor $\on{Sp}_k\to\on{Sp}_k^{\on{pf}}$
by restricting a sheaf $\mF$ on $\on{Aff}_k$ to a sheaf on $\on{Aff}^{\on{pf}}_k$. We denote the induced perfect space by $\mF^{\on{pf}}$. Note that if $X$ is an algebraic space, then $(X^\pf)^{\on{pf}}=X^{\on{pf}}$. More generally, we have the following lemma.
\begin{lem}\label{perfection of stack}
Let $V\rightrightarrows U$ be a flat groupoid of algebraic spaces over $k$ and let $V^\pf\rightrightarrows U^\pf$ denote its perfection. Let $[U/V]$ and $[U^\pf/V^\pf]$ denote the quotient stack (in fpqc topology). Then $[U^\pf/V^\pf]^{\on{pf}}\simeq [U/V]^{\on{pf}}$. That is, $[U^\pf/V^\pf](R)\simeq [U/V](R)$ for every perfect $k$-algebra $R$.
\end{lem}
\begin{proof}Clearly there is a morphism $[U^\pf/V^\pf]^{\on{pf}}\to [U/V]^{\on{pf}}$. Conversely, let $x$ be an $R$-point of $[U/V]$ where $R$ is perfect. Then there is a faithfully flat map $R\to R'$ and a lifting $\tilde x: \Spec R'\to U$ of $x$. Passing to the perfection gives $\Spec {R'}^\pf\to U^\pf$. Since $R\to {R'}^\pf$ is faithfully flat by Lemma \ref{perfection of morphism}, it gives an $R$-point of $[U^\pf/V^\pf]$.
\end{proof}

The functor $\on{Sp}_k\to\on{Sp}_k^{\on{pf}}$ is far from being faithful. However, the following statement is true.
\begin{lem}\label{descent}
The the composition
\[\on{AlgSp}_k^{\on{pf}}\subset \on{AlgSp}_k\subset \on{Sp}_k\to \on{Sp}_k^{\on{pf}}.\]
is a full embedding.
\end{lem}
\begin{proof}
Let $X$ and $Y$ be two perfect algebraic spaces. We need to show that 
$$\Hom_{\on{AlgSp}_k^{\on{pf}}}(X,Y)=\Hom_{\on{Sp}_k^{\on{pf}}}(X^{\on{pf}},Y^{\on{pf}}).$$ 
Let $\{U_i\to X\}$ be a family of \'{e}tale cover of $X$ by affine schemes, and let $\{V_{ijh}\to U_i\times_XU_j\}$ be a family of \'{e}tale cover of $U_i\times_XU_j$ by affine schemes.
By Lemma \ref{Frob etale loc}, all $U_i$ and $V_{ijh}$ are perfect schemes. Therefore, by definition $\Hom_{\on{AlgSp}_k^{\on{pf}}}(U_i,Y)=\Hom_{\on{Sp}_k^{\on{pf}}}(U^{\on{pf}}_i,Y^{\on{pf}})$, etc.

Note that \eqref{fpqc descent} implies that the following sequence is an equalizer
\[\Hom_{\on{AlgSp}_k^{\on{pf}}}(X,Y)\to \prod_i\Hom_{\on{AlgSp}_k^{\on{pf}}}(U_i,Y)\to \prod_{ijh}\Hom_{\on{AlgSp}_k^{\on{pf}}}(V_{ijh},Y).\]
Likewise, the sequence
\[\Hom_{\on{Sp}_k^{\on{pf}}}(X^{\on{pf}},Y^{\on{pf}})\to \prod_i\Hom_{\on{Sp}_k^{\on{pf}}}(U^{\on{pf}}_i,Y^{\on{pf}})\to \prod_{ijh}\Hom_{\on{Sp}_k^{\on{pf}}}(V^{\on{pf}}_{ijh},Y^{\on{pf}})\]
is also an equalizer (in fact, it is enough to use the injectivity of the first map).
The lemma follows by comparing these two sequences.
\end{proof}

Therefore, given a presheaf $\mF$ on $\on{Aff}^{\on{pf}}_k$, it makes sense to ask whether it is represented by a perfect algebraic space, and given a map $f:\mF\to \mG$ of presheaves, it makes sense to ask whether it is representable by perfect algebraic spaces. If a property (P) of morphisms between algebraic spaces is stable under base change and is \'etale local on the source and target, then it makes sense to say whether a representable morphism $f:\mF\to\mG$ of perfect spaces has Property (P). For example, we can define open/closed immersions, \'{e}tale morphisms, fpqc maps in $\on{Sp}_k^{\on{pf}}$, etc.

We can also define the notion of torsors in $\on{Sp}_k^{\on{pf}}$, just as $\on{Sp}_k$. Let $H$ be a perfect affine group scheme. It gives an object $H^{\on{pf}}$ in $\on{Sp}_k^{\on{pf}}$. If $X$ is a perfect space with a action of $H^{\on{pf}}$, then we can define a stack $[X/H^{\on{pf}}]$ on $\on{Aff}^{\on{pf}}_k$ as before. If the action is free, then $[X/H^{\on{pf}}]$ is also a perfect space and the natural map $X\to [X/H^{\on{pf}}]$ is an $H^{\on{pf}}$-torsor. As before in this case we write $[X/H^{\on{pf}}]$ by $X/H^{\on{pf}}$ for simplicity,
Note that if $X$ is a perfect algebraic space, which gives $X^{\on{pf}}$ in $\on{Sp}_k^{\on{pf}}$, then by Lemma \ref{descent} giving an action of $H^{\on{pf}}$ on $X^{\on{pf}}$ is the same as giving an action of $H$ on $X$ and if the action is free, $X^{\on{pf}}/H^{\on{pf}}=(X/H)^{\on{pf}}$.

We define an \emph{ind-perfect algebraic space} as a perfect $k$-space that can be represented as an inductive limit $\{X_i\}$ of perfect algebraic spaces, such that every transition map $X_i\to X_{i+1}$ is a closed embedding.

In the sequel and the main body of the paper, the image of a perfect algebraic space $X$ in $\on{Sp}_k^{\on{pf}}$ is still denoted by $X$, as opposed to $X^{\on{pf}}$ as above. However, for a general space $\mF$, its image in $\on{Sp}_k^{\on{pf}}$ will be denoted by $\mF^{\on{pf}}$.

\subsection{Perfect algebraic spaces perfectly of finite presentation}
\subsubsection{}
Perfect schemes/algebraic spaces of positive dimension are never of finite type over $k$. But as we shall see below, the ``infinity" here is really mild.

\begin{dfn}
A perfect $k$-algebraic space $X$ is said locally perfectly of finite type\footnote{The terminology is suggested by B. Conrad.} if there exist an \'{e}tale affine cover $\{U_i\}$ of $X$ such that each $U_i$ is the perfection of an affine scheme of finite type  over $k$. A perfect $k$-algebraic space $X$ is said perfectly of finite type  if it is locally perfectly of finite type and quasi-compact. A perfect $k$-algebraic space is said perfectly of finite presentation (\emph{pfp} for short) if it is perfectly of finite type and quasi-separated.
\end{dfn}

\begin{rmk} In \cite{Se}, a separated and perfectly of finite type perfect $k$-scheme is called a perfect variety. 
\end{rmk}
Clearly, if there exists an algebraic space $X'$ of finite presentation over $k$ such that $X={X'}^\pf$, then $X$ is perfectly of finite presentation.  We call such $X'$ a ``model" or a ``deperfection" of $X$. We will show a model of a pfp perfect algebraic space always exists. In fact, we will prove a slightly stronger result. For the purpose, we need some preparations. 

For an algebraic space $S$, let $|S|$ denote its underlying topological space. Recall that for a quasi-compact and quasi-separated algebraic space $S$, there is an open dense subspace $U\subset S$ that is a scheme (e.g. \cite[Tag03JG]{St}). The generic points of $|S|$ are in $|U|$ and are the generic points of $U$. So given a generic point $\eta$ of $S$, its residue field $k(\eta)$ makes sense.
Recall that a reduced scheme $X$ of finite type over $k$ is called weakly normal if every finite birational universal homeomorphism $f:Y\to X$ is an isomorphism. By \cite{Ma}, weak normality is local under the \'{e}tale topology (even under the fppf topology). Therefore this notion makes sense for algebraic spaces of finite presentation over $k$.
We have the following result, generalizing  \cite[\S 1.4, Proposition 9]{Se}. 
\begin{prop}\label{finite model}
Let $X$ be a pfp perfect algebraic space over $k$, with $\{\eta_1,\ldots,\eta_n\}$ the set of its generic points. For every $i$, let $K_i\subset k(\eta_i)$ be a subfield, which is finitely generated over $k$ and whose perfection is $k(\eta_i)$. Then there exists a unique weakly normal algebraic space $X'$, of finite presentation over $k$, such that $X={X'}^\pf$ and the residue fields of the generic points of $X'$ are these $K_i$.
\end{prop}
\begin{proof}
We first assume that $X$ is a scheme.  We define a sheaf of rings on $|X|$ by
\begin{equation}\label{prolong}
\mO_{X'}=\{f\in \mO_X\mid f(\eta_i)\in K_i\}.
\end{equation}
It is easy to check that the ringed space $(|X|,\mO_{X'})$ is a scheme, of finite type over $k$, and $X={X'}^\pf$. In addition, $X'$ is weakly normal. In fact, let $U'=\Spec A$ be an affine open subscheme of $X'$. Then the ring $A$ is $p$-closed in the sense that for every $a$ in its quotient ring, if $a^p\in A$, then $a\in A$. By the remark after Proposition 1 in \cite{It}, this condition is equivalent to weak normality of $A$, as proved by Yanagihara.

Note that $X'$ is just the push out of the diagram $\bigsqcup \Spec K_i\leftarrow \bigsqcup \eta_i\to X$ in the category of locally ringed spaces. In particular, for any scheme $Y$, the natural map
\begin{equation}\label{push-out}
\Hom(X',Y)\to \Hom(X,Y) \times_{\Hom(\bigsqcup \eta_i,Y)} \Hom(\bigsqcup \Spec K_i, Y)
\end{equation}
is an isomorphism.

Now for an algebraic space $X$, we can choose a presentation $V\rightrightarrows U\to X$ of $X$ where $U,V$ are schemes. The collection $\{K_i\subset k(\eta_i)\}$ determine a collection of subfields in the residue fields of the generic points of $|V|$ and $|U|$. Then the above construction gives $U'$ and $V'$. Let $\pr_i: V\to U$ be one of the two projections, which descends to an \'{e}tale map $V'_i\to U'$ for some scheme $V'_i$. As $U'$ is weakly normal, so is $V'_i$. Note that the quotient ring of $V'_i$ and $V'$ are the same. By \eqref{push-out}, there is a canonical map $V'\to V'_i$, which is finite birational, and bijective, therefore is an isomorphism. In other words, the \'{e}tale equivalence relation $V\rightrightarrows U$ descends to an \'etale equivalence relation $V'\rightrightarrows U'$. Then $X'=U'/V'$ is the sought-after algebraic space.
\end{proof}

\begin{cor}\label{correction6}
Let $f:X\to Y$ be a separated universal homeomorphism between two pfp perfect algebraic spaces over $k$. Then $f$ is an isomorphism.
\end{cor}
\begin{proof}Let $\eta$ be a generic point of $X$. Since $f$ is a universal homeomorphism, $k(\eta)$ is purely inseparable over $k(f(\eta))$ and therefore is an isomorphism since both fields are perfect. Now, we can choose for each generic point $\eta_i$ a finitely generated subfield $K_i\subset k(\eta_i)$ as above. Let $X'$ and $Y'$ be the corresponding weakly normal models. It follows from the construction that $f$ descends to a finite birational universal homeomorphism $f':X'\to Y'$, and therefore is an isomorphism. The corollary then follows by passing to the perfection.
\end{proof}

The following statement generalizes  \cite[\S 1.4, Proposition 8]{Se}.
\begin{prop}\label{fin:morphism}
Let $f:X\to Y$ be a morphism between pfp perfect algebraic spaces over $k$. Then there exists a morphism $f':X'\to Y'$ between algebraic spaces of finite presentation over $k$ such that $f={f'}^\pf$.
\end{prop}
\begin{proof}Let $X', Y'$ be models of $X, Y$. Then there is a canonical map $\varepsilon: Y\to Y'$. Recall that $\sigma:X'\to X'$ is affine by Lemma \ref{Frob etale loc}. Then by a criterion of locally of finite presentation morphisms (\cite[\S 8.14]{EGA4}, generalized in \cite[Proposition A.3.1]{CLO}, see also \cite[Tag049I]{St}), the map $\varepsilon f$ factors as $X\to {X'}^{(m)}\to Y'$, where ${X'}^{(m)}=X'$ with the $k$-structure given by $X'\stackrel{\sigma^m}{\to}X'\to \Spec k$. Rename ${X'}^{(m)}$ as $X'$, and we are done.
\end{proof}

\begin{dfn}Let $f:X\to Y$ be a morphism between two pfp perfect algebraic spaces over $k$. 
We say $f$ is perfectly proper if it is separated and is universally closed. We say $X$ is perfectly proper if $X\to \Spec k$ is perfectly proper.
\end{dfn}

\begin{lem}\label{model for proper}
For a morphism $f: X\to Y$ between two pfp perfect algebraic spaces over $k$, $f$ is perfectly proper if and only if for every $f':X'\to Y'$ is as in Proposition \ref{fin:morphism}, $f'$ is proper. 
\end{lem}
\begin{proof}By Lemma \ref{perfection of morphism}, $f$ is separated and universally closed if and only if so is $f'$.
\end{proof}

\begin{prop}
Let $f:X\to Y$ be a morphism between two pfp perfect algebraic spaces. Then $f$ is perfectly proper if and only if the valuative criterion holds for every \emph{perfect} valuation ring $R$ over $k$.
\end{prop}
We note that the perfection of a valuation ring is a valuation ring.
\begin{proof}
In fact, if $f$ is perfectly proper, then it is the perfection of a proper morphism $f':X'\to Y'$. Therefore, the map  $\Spec R\to Y\to Y'$ lifts to $\Spec R\to X'$. As $R$ is perfect, it factors through $\Spec R\to X$ by \eqref{perfection}. 
To prove the converse, note that every perfect local ring $A$ in a perfect field $K$ is dominated by a perfect valuation ring.  In addition, to check that $f: X\to Y$ is universally closed, it is enough to check that for every perfect ring $R$, the base change $X_R\to Y_R$ is closed. Then the usual arguments of valuative criterion for properness go through with obvious modifications.
\end{proof}

The following lemma is not used in the paper.
\begin{lem}\label{correct5}
Let $f:X\to Y$ be a perfectly proper morphism between two pfp perfect algebraic spaces, with geometrically connected fibers. Then the natural map $\mO_Y\to f_*\mO_X$ is an isomorphism.
\end{lem}
\begin{proof}Let $f':X'\to Y'$ be a model of $f$ and let $Z'$ be the relative spectrum of $f'_*\mO_{X'}$ over $Y'$. Since $f'$ has geometrically connected fibers, $Z'\to Y'$ is a universal homeomorphism. By Corollary \ref{correction6}, ${Z'}^{\pf}\to Y$ is an isomorphism. But since ${Z'}^\pf$ is the relative spectrum of $f_*\mO_X$, the lemma follows.
\end{proof}

\begin{lem}\label{fin:loc free}
Let $\mE$ be a locally free sheaf of finite rank on a pfp perfect algebraic $X$ over $k$. Then there exists a model $(X',\mE')$ of $(X,\mE)$, i.e algebraic space $X'$, of finite presentation over $k$, and a locally free sheaf $\mE'$ of finite rank such that $(X,\mE)= ({X'}^\pf, \varepsilon^*\mE')$, where $\varepsilon:X\to X'$ is the tautological map.
\end{lem}
\begin{proof}Let $\varepsilon: X\to X'$ be a model. As $X$ is quasi-compact, we can find a finite \'{e}tale cover $\{U_i\}$ of $X$ such that $\mE|_{U_i}\simeq \mO_{U_i}^r$. Then we obtain a \v{C}ech cocycle $f_{ij}: U_{ij}:=U_i\times_XU_j\to \GL_r$. By Proposition \ref{et top}, the \'{e}tale cover $\{U_i\}$ descend to an \'{e}tale cover $\{U'_i\}$ of $X'$. Let  $U'_{ij}=U'_i\times_{X'}U'_j $, and $U'_{ijk}=U'_i\times_{X'}U'_{j}\times_{X'}U'_k$. The map $f_{ij}$ factors as $f'_{ij}: {U'_{ij}}^{(m)} \to \GL_r$ for some $m$ large enough, where as in Proposition \ref{fin:morphism}, ${U'_{ij}}^{(m)}={U'_{ij}}$ with the $k$-structure given by ${U'_{ij}}\stackrel{\sigma^m}{\to} {U'_{ij}}\to \Spec k$. Let $h_{ijk}=f'_{ij}f'_{jk}f'_{ki}: U_{ijk}^{(m)}\to \GL_r$. Then $\varepsilon^*h_{ijk}=1: U_{ijk}\to \GL_r$. Then there is some $n$ big enough such that $(\sigma^n)^*h_{ijk}=1$ for all $i,j,k$. Therefore, we can define a locally free sheaf $\mE$ on ${X'}^{(m+n)}$ by the \v{C}ech cocycle $(\sigma^n)^*f'_{ij}$ (with respect to the \'{e}tale cover $\{{U'_i}^{(m+n)}\}$). By construction, $({X'}^{(m+n)},\mE')$ is a desired model.
\end{proof}

\begin{cor}\label{perf Grass}
Let $X$ be a pfp perfect algebraic space over $k$. Let $\mE$ be a locally free sheaf of rank $n$ over $X$. Then the perfect space which assigns every $f:\Spec R\to X$ the set of rank $i$ quotients $\mQ$ of $f^*\mE$ is represented by a perfect algebraic space $\Gr^\pf(i,\mE)$ perfectly proper over $X$. In particular, if $X$ is perfectly proper, so is $\Gr^\pf(i,\mE)$.
\end{cor}
\begin{proof}Let $(X',\mE')$ be as in Lemma \ref{fin:loc free}. Then $\Gr^\pf(i,\mE)$ is the perfection of the usual Grassmannian $\Gr(i,\mE')$ of rank $i$ quotients of $\mE'$.
\end{proof}

In the sequel, we denote $\Gr^\pf(i,\mE)$ by $\Gr^\pf(i,n)$ if $X=\Spec k$ and $\mE=k^n$ is the standard $n$-dimensional vector space. We denote $\Gr^\pf(1,\mE)$ by $\bP^\pf(\mE)$ and $\Gr^\pf(1,n+1)$ by $\bP^{n,\pf}$.

\begin{rmk}\label{Pf proj}
On $\bP^{n,\pf}$, there is the following tautological rank one quotient $$\mO^{n+1}_{\bP^{n,\pf}}\to \mO_{\bP^{n,\pf}}(1),$$ and therefore a distinguished element $\mO(1):=\mO_{\bP^{n,\pf}}(1)$ in $\Pic(\bP^{n,\pf})$. However, $\mO(1)$ is not the generator of the Picard group. Namely there exists the invertible sheaf $\mO(1/p)=(\sigma^{-1})^*\mO(1)$,
and the Picard group is isomorphic to $\bZ[1/p]$. 
\end{rmk}

We will also need the following definition.
\begin{dfn}Let $f:X\to Y$ be a map between two pfp perfect algebraic spaces. We say that $f$ is perfectly smooth at $x\in X$ if there exists an \'{e}tale atlas $U\to X$ at $x$ and an \'etale atlas $V\to Y$ at $f(x)$, such that the map $U\to Y$ factors as $U\stackrel{h}{\to} V\to Y$ and $h$ factors as $h=\pr \circ h'$, where $h':U\to V\times(\bA^n)^\pf$ is \'{e}tale and $\pr:V\times (\bA^n)^\pf\to V$ is the projection. We say that $f$ is perfectly smooth if it is perfectly smooth at every point of $X$. We say that $X$ is perfectly smooth (at $x$) if $X\to \Spec k$ is perfectly smooth (at $x$).
\end{dfn}
Note that every pfp perfect algebraic space $X$ contains a perfectly smooth open dense subspace.

\subsubsection{}We need to construct the quotient  for an action of a perfect group scheme on a perfect scheme in some cases. In \cite{Se}, a perfect group variety is defined as a group object in the category of perfect varieties. 

\begin{lem}\label{model for group}
Let $H$ be a pfp perfect group scheme over $k$. Then there exists a smooth algebraic group $H_0$ over $k$ such that $H=H_0^\pf$.
\end{lem}
\begin{proof}A priori, $H$ is the perfection of a scheme of finite type over $k$. But as was shown in \cite[\S 1.4, Proposition 10]{Se}, $H$ is the perfection of a group scheme $H'$ of finite type over $k$.
Let $H_0={H'}_{\on{red}}$ be the reduced subscheme. As $k$ is a perfect field, $H_0$ is closed subgroup scheme of $H'$ and is smooth. In addition, $H_0^\pf=H$.
\end{proof}

\begin{cor}\label{et triv}
Let $H$ be an affine pfp perfect group scheme over $k$. Then every $H$-torsor on a perfect algebraic space $X$ can be trivialized \'{e}tale locally on $X$.
\end{cor}
\begin{proof}
This is the combination of Lemma \ref{model for group} and Lemma \ref{equiv of torsors}.
\end{proof}

\begin{lem}\label{approx}
Let $H$ be a pfp perfect affine group scheme over $k$ acting on a pfp perfect affine scheme $X$ over $k$. Then this action arises as the perfection of an action of a smooth affine algebraic group $H'$ over $k$  on an affine scheme $X'$ of finite type over $k$.
\end{lem}
\begin{proof}Recall the following basic fact: let $H$ be an affine group scheme over $k$, with $A$ its ring of functions. Let $\rho: V\to A\otimes_k V$ be a representation of $H$. Then $V$ is a union of finite dimensional $H$-modules. 

Now let $B$ denote the ring of regular functions on $X$. Let $B_0\subset B$ be a finitely generated subalgebra whose perfection is $B$ and let $V_0\subset B_0$ be a finite dimensional subspace containing a set of generators of $B_0$. Let $V'$ be a finite dimensional $H$-invariant subspace of $B$ that contains $V_0$. Let $B'\subset B$ be the subalgebra generated by $V'$. Then $B'$ is an $H$-invariant subalgebra of $B$ and ${B'}^\pf=B$ (since ${B'}^\pf\supset B_0^\pf=B$). Let $X'=\Spec B'$.
Then $X'$ is of finite type over $k$, and admits an action of $H$, which induces the action of $H$ on $X$.  In addition, as $H=\underleftarrow\lim H_0$, the action of $H$ on $V'$ is induced from some action of $H'$ on $V'$ with $H'$ being smooth. Then the action $H\times X'\to X'$ is induced from some action $H'\times X'\to X'$ as required.
\end{proof}

Now the main result of this appendix is as follows.

\begin{thm}\label{quotient}
Let $H,X$ be as above. Furthermore, we assume that the action is free, i.e. the map $(\on{act},\pr_2):H\times X\to X\times X,\ (g,x)\mapsto (gx,x)$ is a monomorphism. Then $X/H$ is represented by a pfp perfect algebraic space over $k$. In addition, if $(\on{act},\pr_2)$ is a closed embedding, then $X/H$ is separated as well.
\end{thm}
\begin{proof}
First note that the diagonal $X/H\to X/H\times X/H$ is always representable. So it is enough to show that $X/H$ admits an \'etale cover by a scheme.
Therefore in order to prove the theorem, we are free to replace $X$ by an $H$-equivariant \'{e}tale cover $Y\to X$ by a perfect affine scheme $Y$ perfectly of finite type. Indeed, if $Y/H$ is representable by a perfect pfp algebraic space over $k$, then by Lemma \ref{et triv}, we can find an \'{e}tale  cover of $Y/H$ by an affine scheme that trivializes the $H$-torsor $Y\to Y/H$. I.e. after a further \'etale localization, we can assume $Y=H\times U$ where $U$ is a scheme. Then $U\to X/H$ is an \'etale atlas of $X/H$.

By Lemma \ref{approx}, there is an action $\on{act}':H'\times X'\to X'$ that induces $\on{act}:H\times X\to X$. 
Now the action may not be free, but it is quasi-finite. Therefore, the stack $\frakX'=[X'/H']$ is of finite presentation over $k$ and has quasi-finite diagonal. It follows that there is an \'etale cover $\frakY'\to \frakX'$ by an Artin stack $\frakY'$ which is separated and of finite presentation over $k$, such that there exists a finite flat cover $\frakZ'\to \frakY'$ with $\frakZ'$ being a quasi-projective scheme over $k$ (see \cite[Lemma 3.3, Proposition 4.2]{KM} or \cite[Lemma 2.1, 2.2]{C}). Let $Y'=X'\times_{\frakX'}\frakY'$. Then $Y'\to X'$ is an $H'$-equivariant \'etale cover, and after passing to the perfection the action of $H$ on $Y={Y'}^\pf$ is free. As explained above, it is then enough to show that $Y/H$ is representable. 
Let $V'=\frakZ'\times_{\frakY'}\frakZ'$. Then we have a finite flat groupoid $V'\rightrightarrows \frakZ'$ such that $\frakY'=[\frakZ'/V']$. Let $V\rightrightarrows\frakZ$ denote the perfection of this groupoid. It gives rise to an equivalence relation of $\frakZ$ (i.e. $V\to \frakZ\times\frakZ$ is a monomorphism). By Lemma \ref{perfection of stack}, $Y/H=\frakZ/V$. The theorem then follows from the next statement. 
\end{proof}

\begin{thm}\label{quotient by finite groupoid}
Let $V\rightrightarrows U$ be an equivalence relation of pfp perfect schemes. Assume that it is the perfection of
 a finite locally free groupoid $V'\rightrightarrows U'$ of quasi-projective schemes over $k$. Then $U/V$ is represented by a pfp perfect scheme.
\end{thm}

\begin{proof}
By the standard reduction (\cite[Tag03JE]{St}), we can assume that $U'=\Spec A$ and $V'=\Spec B$ are affine. Let $s,t: A\rightrightarrows B$ be the source and target map.

Let $C$ be the equalizer $A\rightrightarrows B$.
Then $C^\pf$ is the equalizer  for $A^\pf\rightrightarrows B^\pf$. The theorem follows if we can show that $U/V\simeq \Spec C^\pf$.
Let us write $X'=\Spec C$ and $X={X'}^\pf$.
We claim that
$$V \to U\times_{X} U$$
is an isomorphism. First it is a monomorphism since $V\to U\times U$ is an equivalence relation. It is also an integral morphism since it is the perfection of a finite morphism (cf. Lemma \ref{perfection of morphism}). In addition, it follows from the classical theory on quotients by finite locally free groupoids that the map is surjective (cf. \cite[Tag03BL]{St}).
Therefore,  $V\to U\times_{X}U$ is a universal homeomorphism. By Corollary \ref{correction6}, it is an isomorphism.

Therefore, it remains to show that $A^\pf$ is flat over $C^\pf$. It follows by applying the following lemma to the $C^\pf$-module $A^\pf$ and the injective integral map $C^\pf\to A^\pf$.
\end{proof}
\begin{lem}
Let $R\to S$ be an injective integral map of perfect rings, which is the perfection of a finite map. Let $M$ be an $R$-module. If $M\otimes_RS$ if flat over $S$, then $M$ is flat over $R$. 
\end{lem}
\begin{proof}This is a perfect version of a result of Ferrand. For the completeness, we repeat the argument as in \cite[Tag0533]{St}. Assume that $R\to S$ is the perfection of a finite map $R'\to S'$. By \cite[Tag0531]{St}, there is a finite free ring extension $R'\to R''$ such that $S''=S'\otimes_{R'}R''=R''[T_1,\ldots,T_n]/J$, where $J$ contains
$$(P_1(T_1),\ldots, P_n(T_n)),\ P_i(T)=\prod_{j=1,\ldots,d_i}(T-a_{ij}), \ a_{ij}\in R''.$$ Then $R\to {R''}^\pf$ is faithfully flat and it is enough to prove the flatness of $M\otimes_R{R''}^\pf$. Therefore, we may replace $R',S'$ by $R'',S''$ and rename them as $R'$ and $S'$. Now for $\underline k=(k_1,\ldots,k_n), \ 1\leq k_i\leq d_i$, we define the ideal $J_{\underline k}\subset R'$ as the image of $J$ under the map $R''[T_1,\ldots, T_n]\to R'',\ T_i\mapsto a_{i k_i}$. Then the quotient map $R'\to R'/J_{\underline k}$ factors through $R'\to S'\to R'/J_{\underline k}$.  Therefore, $M/J^\pf_{\underline k}M$ is flat over $R/J^\pf_{\underline k}$.

Since $R\to S$ is injective integral, $\Spec S\to \Spec R$ is surjective. Therefore $\Spec S'\to \Spec R'$ is surjective.  Then by \cite[Tag0532]{St}, $I=\cap J_{\underline k}\subset R'$ is in the nilradical. Passing to the perfection, $\cap J_{\underline k}^\pf=(0)$. It then follows from \cite[Tag0522]{St} that $M$ is flat over $R$.
\end{proof}

Finally, let us discuss the orbits for the group action.
\begin{prop}\label{perfect orbit}
Let $H$ be a connected pfp perfect affine group scheme acting on a separated pfp perfect algebraic space $X$. Let $x\in X$ be a closed point and let $H_x$ be the stabilizer of $x$ in $H$. Then the induced map $i:H/H_x\to X$ is a locally closed embedding.
\end{prop}
One can regard $H/H_x$ as the $H$-orbit through $x$.
\begin{proof}
Let us denote $H/H_x$ by $Y$ for simplicity.
By definition, $i:Y\to X$ is a monomorphism. In the classical theory, since the induced map of tangent spaces is injective, it follows easily that $i$ is a locally closed embedding. In our setting, we need an alternative argument. First, as in the classical situation, $|i(Y)|\subset |X|$ is locally closed. Namely, $|i(Y)|$ contains an open subset of its closure in $|X|$. Then by the group action, $|i(Y)|$ is open in its closure. Therefore, we may replace $X$ by a locally closed subspace and assume that $i:Y\to X$ is a bijective monomorphism. 
Let us choose some model $i':Y'\to X'$. Note that $i'$ is quasi-finite, and therefore by Zariski's main theorem (\cite[Tag05W7]{St}), it factors as $Y'\stackrel{j'}{\to} Z'\stackrel{q'}{\to} X'$ where $j':Y'\to Z'$ is open and $q':Z'\to X'$ is finite. By replacing $Z'$ by the closure of $Y'$ in it, we can assume that $Z'$ is irreducible.  Then $\dim (Z'\setminus Y')<\dim X'$, and therefore, there is an open subspace $U'\subset X'$ such that $q'^{-1}(U')\subset j'(Y')$. In other words, $i': {i'}^{-1}(U')\to U'$ is finite. Finite morphisms are integral and therefore after passing to the perfection, $i:i^{-1}(U)\to U$ is integral  by Lemma \ref{perfection of morphism}. Since $H$ acts transitively on points, we see that $i: Y\to X$ is integral. Then, $i$ is an integral, bijective monomorphism and therefore is a universal homeomorphism. By Corollary \ref{correction6}, $i$ is an isomorphism.
\end{proof}

\subsection{$\ell$-adic sheaves}
\subsubsection{}
The notion of (constructible) \'{e}tale sheaves makes sense for separated pfp perfect algebraic spaces. Indeed, let $X$ be such an algebraic space, and let $\varepsilon:X\to X'$ be a model of $X$. Note that $\varepsilon$ is a universal homeomorphism. Therefore by Lemma \ref{et top}, for an \'{e}tale sheaf $\mF$ on $X$ and an \'{e}tale sheaf $\mF'$ on $X'$, the natural maps
\begin{equation}\label{et model}
\varepsilon^*\varepsilon_*\mF\to \mF,\quad \mF'\to \varepsilon_*\varepsilon^*\mF'
\end{equation}
are isomorphisms.
In particular, for such $X$, one can define the corresponding $\ell$-adic derived category $D_c^b(X,\Ql)$ ($\ell\neq p$) as usual, with a pair of adjoint functors that are equivalences
$$\varepsilon^*:D_c^b(X',\Ql)\simeq D_c^b(X,\Ql):\varepsilon_*.$$

One can define six operations between these categories, thanks to Proposition \ref{fin:morphism}. The usual proper base change or smooth base change holds for perfectly proper or perfectly smooth maps. The definition of perverse sheaves works in this setting without change.  In particular, we have the notion of the Goresky-Macpherson intermediate extension, and for every $X$, the intersection cohomology sheaf $\IC_X$. The restriction of $\IC_X$ to any perfectly smooth open subset $U$ is canonically isomorphic to $\Ql[2\dim X](\dim X)$. We will denote by $\on{P}(X)$ the category of perverse sheaves on $X$.

\subsubsection{}\label{Chern class}
Let $X$ be a separated pfp perfect algebraic space over $k$. One can define Chern classes for locally free sheaves on $X$ as usual.
For an invertible sheaf $\mL$ on $X$, corresponding to a class $[\mL]\in H^1(X,\bG_m)$, we define its Chern class $c_1(\mL)$ as the image of $[\mL]$ under
$$H^1(X,\bG_m)\to H^1(X_{\bar{k}},\bG_m)\to H^2(X_{\bar{k}},\mu_{\ell^m}).$$
In general, for a locally free sheaf $\mE$ of rank $n$ over $X$,  let $\mO(1)=\mO_{\bP^\pf(\mE)}(1)$ denote the tautological line bundle on $\bP^\pf(\mE)$ and let $\xi=c_1(\mO(1))\in H^2(\bP^\pf(\mE)_{\bar k},\mu_{\ell^m})$. Then there are unique cohomology classes $c_i(\mE)\in H^{2i}(X_{\bar k},\mu_{\ell^m}^{\otimes i})$ such that
\[\xi^n-c_1(\mE)\xi^{n-1}+\cdots+(-1)^nc_n(\mE)=0.\]
Passing to the inverse limit and inverting $\ell$, we obtain $\Ql$-coefficient Chern classes.
The usual properties of Chern classes hold in this setting. 

One can also define characteristic classes for general principal homogeneous spaces. 
Let $G$ be a (connected) reductive group over $k$ and let $G_{\Ql}$ be the corresponding split group over $\Ql$. Let $R_{G,\ell}=\on{Sym}(\frakg^*_{\Ql}(-1))^{G_{\Ql}}$ denote the algebra of invariant polynomials on the Lie algebra $\frakg_{\Ql}(1)$. Then given a $G$-torsor $E$ on $X$ (equivalently a $G^\pf$-torsor on $X$ by \S~\ref{torsor}), its characteristic classes can be regarded as a ring homomorphism 
\begin{equation}\label{char classes}
c(E): R_{G,\ell}\to \on{H}^*(X_{\bar k},\Ql),
\end{equation}
which can be constructed as follows: let $B$ be a Borel subgroup of $G_{\bar k}$, with the unipotent radical $U$ and $T=B/U$. Let $W$ be the Weyl group. Then the $T$-torsor $E_{\bar k}/U\to E_{\bar k}/B$ induces a map $c(E_{\bar k/U}):R_{T,\ell}\to \on{H}^*(E_{\bar k}/B,\Ql)$ via the above construction of the Chern classes. Passing to some models, we see that $R_{G,\ell}(=R_{T,\ell}^W)$ maps to $\on{H}^*(X_{\bar k},\Ql)\subset \on{H}^*(E_{\bar k}/B,\Ql)$, giving the desired map \eqref{char classes}. For a general connected algebraic group $G$, let $G^{\on{red}}$ denote its reductive quotient. Then a $G$-torsor $E$ induces a $G^{\on{red}}$-torsor $E^{\on{red}}$, and we define $c(E)$ as $c(E^{\on{red}})$.

\begin{rmk}Alternatively, one can define the Chern classes (or general characteristic classes) of $E\to X$ by first passing to some model $E'\to X'$ (Lemma \ref{fin:loc free}) and then set $c(E)=c(E')$ using the identification $\on{H}^*(X_{\bar k},\Ql)=\on{H}^*(X'_{\bar k},\Ql)$. Then one
shows that this definition is independent of the choice of the model.
\end{rmk}

\subsubsection{}\label{cycle class}
 We can also define the cycle class map in the current setting. First, if $X$ is irreducible of dimension $d$, there is a canonical isomorphism
\[c_X:\on{H}^{2d}_c(X_{\bar k},\Ql(d))\simeq \Ql,\]
given as follows: Choose a model $\varepsilon:X\to X'$, which induces the canonical isomorphism 
$$\varepsilon^*: \on{H}^{2d}_c(X'_{\bar k},\Ql(d))\simeq \on{H}^{2d}_c(X_{\bar k},\Ql(d)).$$ 
Then we define $c_X= c_{X'}\circ(\varepsilon^*)^{-1}$.
Note that if  $f:X'\to X''$ is a morphism of two $d$-dimensional irreducible algebraic spaces of finite presentation, the canonical isomorphism $f^*:\on{H}^{2d}_c(X''_{\bar k},\Ql(d))\simeq \on{H}^{2d}_c(X'_{\bar k},\Ql(d))$ is compatible with $c_{X'}$ and $c_{X''}$. Therefore, $c_X$ is well-defined.
Alternatively, one can build $c_X$ directly, starting from the canonical isomorphism $$c_{\bP^{1,\pf}}:\on{H}^2(\bP^{1,\pf}_{\bar k},\mu_{\ell^n})\simeq \on{coker}(\Pic(\bP^{1,\pf})\stackrel{\ell^n}{\to}\Pic(\bP^{1,\pf}))\simeq \bZ/\ell^n$$ and then using the functoriality of the six operations.

Now let $X$ be a separated pfp perfect algebraic space.  Let $\omega_X=f^!\Ql$ denote the dualizing sheaf. We define the Borel-Moore homology of $X$ as
\[\on{H}^{\on{BM}}_i(X_{\bar k})=\on{H}^{-i}(X_{\bar{k}},\omega_X(-i/2)).\]
The usual properties of the Borel-Moore homology hold (by the functoriality of the six operations). We list a few.
\begin{itemize}
\item If $f:X\to Y$ is a perfectly proper morphism, there is a canonical map
\[f_*: \on{H}^{\on{BM}}_*(X_{\bar k})\to \on{H}^{\on{BM}}_*(Y_{\bar k}).\]

\item There is a canonical isomorphism $\on{H}^{\on{BM}}_{i}(X_{\bar k})\simeq \on{H}^{i}_c(X_{\bar k},\Ql(i/2))^*$. Therefore if $X$ is irreducible of dimension $d$, there exists the fundamental class $$[X]:=c_X\in \on{H}^{\on{BM}}_{2d}(X_{\bar k}).$$ In general, if $X$ is $d$-dimensional, with $X_1,\ldots,X_n$ its irreducible components of dimension $d$, then the natural map $\bigoplus_i \on{H}^{\on{BM}}_{2d}((X_i)_{\bar k})\simeq \on{H}^{\on{BM}}_{2d}(X_{\bar k})$ is an isomorphism. We set $[X]=\sum_i [X_i]$.

\item If $X$ is irreducible and perfectly smooth, the fundamental class $[X]$, regarded as a map of sheaves $\Ql\to \omega_X[-2d](-d)$ is an isomorphism. Therefore, $\on{H}^{BM}_{i}(X_{\bar k})\simeq \on{H}^{2d-i}(X_{\bar k},\Ql(d-i/2))$.
\end{itemize}
Finally, let $Z$ be a closed subset of codimension $r$. We define
the cycle class $\on{cl}(Z)$ of $Z$ as the image of $[Z]$ in $\on{H}^{\on{BM}}_{2(d-r)}(X_{\bar k})$. If $X$ is perfectly smooth and perfectly proper, we can regard $\on{cl}(Z)$ as a class in $\on{H}^{2r}(X_{\bar k},\Ql(r))$.

\subsubsection{}\label{trace formula}
We have the Lefschetz trace formula in this setting. Let $\mF$ be an $\ell$-adic complex with constructible cohomology on a separated pfp perfect algebraic space $X$ over $\bF_q$. As usual, one can attach a function 
$$f_\mF: X(\bF_{q})\to \Ql, \quad x\mapsto \tr(\sigma_x,\mF_{\bar{x}})=\sum_i (-1)^i \tr(\sigma_x,(\mH^i\mF)_{\bar{x}}),$$
where $x\in X(\bF_{q^r})$, $\bar{x}$ a geometric point over $x$, $(\mH^i\mF)_{\bar{x}}$ the stalk cohomology of $\mF$ at $\bar{x}$, and $\sigma_x$ is the \emph{geometric} Frobenius at $x$.

Let $f:X\to Y$ be a morphism of separated pfp perfect algebraic spaces over $\bF_q$. Let $\mF$ be an $\ell$-adic complex with constructible cohomology on $X$. Then the usual trace formula holds in this setting. Namely,
\[f_{f_!\mF}(y)= \sum_{x\in f^{-1}(y)(\bF_q)} f_{\mF}(x).\]
To prove this, one can replace $f:X\to Y$ by a model $f':X'\to Y'$, and replace $\mF$ by $\varepsilon_*\mF$.

\subsubsection{}\label{equivariant category}
 In the paper, we also need some basic facts about equivariant category and equivariant cohomology. Let $J$ be an affine pfp perfect group scheme over $k$. Recall that by Lemma \ref{model for group}, $J$ is perfectly smooth. Let $X$ be a separated pfp perfect algebraic space over $k$ with an action of $J$.
Then it makes sense to talk about the category of $J$-equivariant perverse sheaves on $X$, denoted by $\on{P}_J(X)$\footnote{One can also define the equivariant derived category.}. I.e., an object in $\on{P}_J(X)$ is a perverse sheaf on $X$ together with an isomorphism of the pull-backs along the two maps $J\times X\rightrightarrows X$, satisfying the usual compatibility conditions. 

We have the following two properties of the equivariant category. Let $J_1\subset J$ be a closed normal subgroup. 

(i) If the action of $J_1$ on $X$ is free and $[X/J_1]$ is represented by an algebraic space $\bar X$, then the pull back along $q:X\to \bar X$ induces an equivalence of categories 
\begin{equation}\label{equivariance free action}
q^*[\dim J_1]:\on{P}_{J/J_1}(\bar X)\simeq \on{P}_J(X)
\end{equation}

(ii) Assume that $J_1$ is connected and that the action of $J_1$ on $X$ is trivial. Then the forgetful functor
\begin{equation}\label{unipotence}
\on{P}_J(X)\to \on{P}_{J/J_1}(X)
\end{equation}
is an equivalence of categories.

The proof is the same as the classical (i.e. non-perfect) situation (e.g. see \cite[Lemma A.1.4]{Z16}).

For $\mA\in \on{P}_J(X)$ (or more generally a $J$-equivariant complex), it makes sense to talk about the $J$-equivariant cohomology $\on{H}_J^*(X_{\bar k},\mA)$. Namely, let $J_0$ be a smooth model of $J$ as in Lemma \ref{model for group}. Let  $\{E_n\to B_n\}$ denote a sequence of $J_0$-torsors over $\{B_n\}$, which approximate of the classifying space of $J_0$ in the sense that $B_n\subset B_{n+1}$ is a closed embedding, and $\on{H}^*(\underrightarrow\lim_n B_n)= \on{H}^*(BJ_0)$. E.g. we can embed $J_0$ into some $\GL_r$ such that $\GL_r/J_0$ is quasi-affine. Then for $n$ large, let $E_n:=S_{n,r}$ be the Stiefel variety, i.e. the tautological $\GL_r$-torsor over $\Gr(r,n)$. Then $B_n:=E_n/J_0$ is represented by a scheme, and the ind-scheme $\underrightarrow\lim_n B_n$ satisfies the required property. Then the sheaf $\Ql\boxtimes \mA$ on $E^\pf_n\times X$ is $J$-equivariant with respect to the diagonal action, and therefore by \eqref{equivariance free action} descends to a sheaf $\Ql\tilde\boxtimes \mA$ on $E_n\tilde\times X$, which is a separated pfp perfect algebraic space by the discussion in \S \ref{torsor}. Then 
$$\on{H}_J^*(X_{\bar k},\mA):=\on{H}^*(\underrightarrow\lim_n (E^\pf_n\tilde\times X)_{\bar k}, \Ql\tilde\boxtimes \mA).$$
From the construction $\on{H}^*_J(X_{\bar k},\mA)$ is a module over $\on{H}^*(\underrightarrow\lim_n B_n)=\on{H}^*(BJ_0)$. Let us also recall the Lie theoretical description of $\on{H}^*(BJ_0)$ in the case when $J_0$ is connected: if we denote by $G=J_0^{\on{red}}$ the reductive quotient of $J_0$ over $k$, then $\on{H}^*(BJ_0)=R_{G,\ell}$. 

It is clear from the definition that if $J_1\subset J$ acts freely on $X$ with $\bar X$ the quotient as above, then
\begin{equation}\label{coh free action}
\on{H}^*_J(X_{\bar k},q^*\mA)=\on{H}^*_{J/J_1}(\bar X_{\bar k},\mA).
\end{equation}
On the other hand, if $J_1$ is the perfection of a unipotent group and acts trivially on $X$, then
\begin{equation}\label{coh unip}
\on{H}^*_J(X_{\bar k},\mA)=\on{H}^*_{J/J_1}(X_{\bar k},\mA).
\end{equation}

Let $J$ be a perfect affine group scheme acting on a php perfect algebraic space $X$ satisfying the following condition (which is always the case in the paper): 
there exists a closed normal subgroup $J_1\subset J$ of finite codimension acting trivially on $X$, and $J_1$ is the perfection of a pro-unipotent pro-algebraic group.
Then we can define $\on{P}_J(X)$ as $\on{P}_{J/J_1}(X)$ and define $\on{H}^*_J(X,\mA)$ for $\mA\in\on{P}_J(X)$ as $\on{H}^*_{J/J_1}(X,\mA)$. By \eqref{unipotence} and \eqref{coh unip}, both definitions are independent of the choice of $J_1\subset J$.

\section{More on mixed characteristic affine Grassmannians}\label{rep as sch}
In this section, we discuss some unsolved questions related to mixed characteristic affine Grassmannians\footnote{Conjecture I and II have been recently proved by Bhatt-Scholze \cite{BS}. Conjecture III is still open.}. We also give an example of our construction. Proofs are generally omitted in this section.

\subsection{The determinant line bundle}\label{det line}

We continue to use the notations as in \S \ref{aff Grass}.
Recall that in equal characteristic, there is the important determinant line bundle $\mL^\flat_{\det}$ on $\Gr^\flat$. Its fiber at a point  $(\mE,\beta)\in\bGr_N^\flat$ (the equal characteristic analogue of $\bGr_N$) is $\wedge^{N}(\mE_0/\mE)$ (recall that $\mE_0/\mE$ is a projective $R$-module of rank $N$). In mixed characteristic, $\mE_0/\mE$ is not an $R$-module, and therefore $\wedge^{N}(\mE_0/\mE)$ does not make sense a priori. Alternatively, one can try to define this line bundle  by choosing a filtration of $\mE/\mE_0$ such that the associated graded is a projective $R$-module and then taking its top exterior power. This idea leads to a line bundle on $\wGr_N$.
Let us formulate it more generally for  $\wGr_{\mu_\bullet}$.
Given a quasi-isogenies $\mE\stackrel{\beta}{\dasharrow}\mE'$ with $\inv(\beta)\in\{\omega_1,\ldots,\omega_n,\omega_1^*,\ldots,\omega_n^*\}$, we can define a line bundle
\[\mL_i=\left\{\begin{array}{ll}
 \wedge^j\mE'/\mE  & \mu_i=\omega_j,\\  
\wedge^j\mE/\mE'  & \mu_i=\omega_j^*.
\end{array}
\right.\]
As $\Gr_{\mu_\bullet}$ classifies all chains of quasi isogenies $\mE_n\dasharrow \mE_{n-1}\dasharrow\cdots\dasharrow \mE_0$ with each step being of the above form,  we can define line bundles $\mL_i$ as above, and set
\begin{equation}\label{aml}
\tilde\mL_{\det}=\bigotimes_i \mL_i.
\end{equation}

\medskip

\noindent\bf Conjecture I. \rm
There is a unique line bundle $\mL_{\det}$ on $\bGr_N$ such that its pullback along $\wGr_N\to \bGr_N$ is $\tilde\mL_{\det}$ in \eqref{aml}.

\medskip

An evidence of Conjecture I is the following (we ignore the Tate twist).
\begin{prop}\label{bpf}
The map $\on{H}^*((\bGr_N)_{\bar k},\Ql)\to \on{H}^*((\wGr_N)_{\bar k},\Ql)$ is injective, and there exists a class $c\in \on{H}^2((\bGr_N)_{\bar k},\Ql)$, whose image in $\on{H}^2((\wGr_N)_{\bar k},\Ql)$ is the Chern class $c(\tilde\mL_{\det})$.
\end{prop}

We can reduce Conjecture I to Conjecture IV stated in the sequel. It is based on the following assertion.
\begin{prop}\label{base point free}
If $\Ga(\wGr_{N},\tilde{\mL}_{\det})$ is base point free, i.e. for every closed point $x\in \wGr_{N}$ there exists a section $s\in \Ga(\wGr_{N},\tilde{\mL}_{\det})$ such that $s(x)\neq 0$, then Conjecture I holds.
\end{prop}
Indeed, if the assumption in the above proposition holds, then the pushforward of $\tilde{\mL}_{\det}$ along $\wGr_N\to\bGr_N$ will give $\mL_{\det}$.

\begin{rmk}
More generally, given a perfect ring $R$, one can define the following category $\mC_R$ of triples $(\mE_1,\mE_2,\beta)$, where $\mE_1$ and $\mE_2$ are two finite projective $W(R)$-modules, and $\beta:\mE_1\dasharrow\mE_2$ is a quasi-isogeny. This is an exact category. It is an interesting question to see whether the algebraic K-theory of $\mC_R$ is related to the K-theory of $R$.
\end{rmk}

\medskip

\noindent\bf Conjecture II. \rm A model of the line bundle $\mL_{\det}$ (in the sense as Lemma \ref{fin:loc free}) induces an embedding of a model of $\bGr_N$ into some projective space. In particular, $\bGr_N$ is the perfection of a projective variety.

\medskip

Note that assuming Conjecture I, Conjecture II holds if the space of global sections of $\mL_{\det}$ separate points. I.e. for every two different points $x,y\in \bGr_{N}$, there exists $s\in \Ga(\bGr_{N},\mL_{\det})$ such that $s(x)=0$ and $s(y)\neq 0$.
As an evidence, we see in Remark \ref{wGr2}, and in particular in the sequel \S~\ref{sample cal} that in some cases when $\mu$ is (very) small, $\Gr_{\leq \mu}$ is the perfection of some projective variety.

\subsection{Deperfection}
Recall that from the proof of Proposition \ref{sph sch var}, the perfect scheme $\Gr_\mu=L^+G/ L^+G\cap \varpi^\mu L^+G\varpi^{-\mu}$ has a canonical model $\Gr'_\mu$, which is a smooth quasi-projective variety.  As $\Gr_\mu$ is open dense in $\Gr_{\leq \mu}$, it gives rise to a weakly normal model $\Gr'_{\leq \mu}$ of $\Gr_{\leq \mu}$, which is a proper algebraic space over $k$, containing $\Gr'_{\mu}$ as an open subset (see Proposition \ref{finite model}).

\medskip

\noindent\bf Conjecture III. \rm The proper algebraic space $\Gr'_{\leq \mu}$ is normal and Cohen-Macaulay.

\medskip

An evidence of this conjecture is the following lemma.
\begin{lem}\label{lci}
The affine scheme $V'_{N,h}$ defined via \eqref{notation} is normal and a locally complete intersection.
\end{lem}
We give a hint of the proof. First, note that $V'_{N,h}$ is defined by $N$-equations in an affine space of dimension $n^2h$. On the other hand, it is not hard to show that $\dim V'_{N,h}=\dim V_{N,h}=n^2h-N$ (e.g. by calculating the dimension of the fiber of $J$ at $\on{diag}\{p^{m_1}, p^{m_2}, \ldots , p^{m_n}\}$).  So it is a complete intersection.

To show that it is normal, it is then enough to show the smoothness in codimension one.
First, one shows that $V'_{N,h}$ is non-singular at $A=\on{diag}\{p^N,1,1,\ldots,1\}$ by calculating the dimension of the tangent space at $A$. Next using the action of $L^h_p\GL_n\times L^h_p\GL_n$ by left and right multiplication, $V'_{N,h}$ is non-singular at every point of $L^h_p\GL_n\cdot A\cdot L^h_p\GL_n$. On the other hand, topologically, this double coset is exactly the preimage of $\Gr_N$ under $V_{N,h}\to \bGr_N$. By Lemma \ref{fiber of pi}, $\bGr_N$ is of dimension $N(n-1)$, and the codimension of $\bGr_N-\Gr_N$ is at least two. Then $V'_{N,h}$ is non-singular away from a codimension two closed subset, and therefore is normal.

\begin{rmk}As mentioned in Remark \ref{wGr2}, $\Gr'_{\leq \mu}$ is probably not isomorphic to its equal characteristic counterpart $\Gr^\flat_{\leq \mu}$ for general $\mu$. However, according to the geometric Satake, their intersection cohomology (or more generally, their motives) are isomorphic.
\end{rmk}

\begin{rmk}If $\la<\mu$, we have the closed embedding $i:\Gr_{\leq\la}\subset\Gr_{\leq\mu}$, which by Proposition \ref{fin:morphism}, is induced from some map $i':{\Gr'}^{(m)}_{\leq\la}\to\Gr'_{\leq\mu}$, where ${\Gr'}^{(m)}_{\leq\la}=\Gr'_{\leq \la}$ but with  the $k$-structure given by $\Gr'_{\leq\la}\stackrel{\sigma^m}{\to}\Gr'_{\leq\la}\to\Spec k$. However, $i'$ is \emph{not} a closed embedding and therefore we do not have a deperfection of the whole affine Grassmannian. 
\end{rmk}

A natural question is whether $\Gr'_{\leq \mu}$ has a natural moduli interpretation. We have no idea how to answer this question. The following discussion provides some hints that this might be an interesting question.

We consider $G=\GL_n$. As mentioned above, we do not know a moduli interpretation of $\bGr'_N:=\Gr'_{\leq N\omega_1}$. On the other hand, recall that there is the ``Demazure resolution" $\wGr_N\to \bGr_N$. By Lemma \ref{fiber of pi}, $\Gr_N\subset \wGr_N$ is open dense, so by Proposition \ref{finite model} we also have a canonical model $\wGr'_N$. 

We do have a moduli interpretation of $\wGr'_N$, as suggested by L. Xiao. For simplicity, we assume that $k=\overline\bF_p$. Fix $h\geq N$. Let $E/\bQ_p$ be an unramified extension of degree $2h$, with ring of integers $\mO_E=\bZ_{p^{2h}}$. We fix an embedding $\tau_0: E\to \overline\bQ_p$, and let $\tau_i=\sigma^i\tau_0$, where $\sigma: \overline\bQ_p\to \overline\bQ_p$ is (a lift of) the Frobenius automorphism. Then $\tau_{i+2h}=\tau_i$.

Let $\bX_0$ be a $p$-divisible group of height $2hn$ and dimension $hn$ over $\overline\bF_p$, with an action $\iota:\mO_E\to \End\bX_0$. We assume that the signature of $(\bX_0,\iota)$ is $(0,\ldots,0,n,\ldots,n)$. I.e. $\on{rk} (\Lie \bX_0)_{\tau_i}=0$ for $i=0,\ldots,h-1$, where $(\Lie \bX_0)_{\tau_i}=\Lie\bX_0\otimes_{\mO_E,\tau_i}\overline\bF_p$. 

We define a space $M_{N,h}\in \on{Sp}_{\overline \bF_p}$ as follows. For an $\overline\bF_p$-algebra $R$, the set $M_{N,h}(R)$ classifies chains of isogenies of $p$-divisible groups on $R$ with an $\mO_E$-action,
\[(\bX_0)_R\stackrel{\phi_1}{\to}\bX_1\to\cdots\stackrel{\phi_N}{\to} \bX_N,\]
satisfying
\begin{enumerate}
\item $(\bX_i,\iota)$ has signature $(1,\ldots,1,0\ldots,0,n,\ldots,n,n-1\ldots,n-1)$, where the first $i$ positions are $1$s and the last $i$ positions are $(n-1)$s;
\item $\deg\phi_i=p^{2i-1}$ for $i=1,\ldots,N$;
\item the differential $(d\phi_i)_j: (\Lie \bX_{i-1})_{\tau_j}\to (\Lie \bX_{i})_{\tau_j}$ is zero for $j=0,\ldots,h-1$; and
\item the dual of the differential $(d\phi^*_i)_j: (\Lie \bX_{i}^*)_{\tau_j}\to (\Lie \bX_{i-1}^*)_{\tau_j}$ is zero for $j=h,\ldots,2h-1$.
\end{enumerate}
Note that the first two conditions imply that $\ker \phi_i\subset \bX_{i-1}[p]$. 

The first two conditions define a moduli scheme closely related to Rapoport-Zink spaces. However, it is not irreducible in general (not even equidimensional) and the last two conditions cut out $M_{N,h}$ inside it as an irreducible component. More precisely,
by induction we have
\begin{lem}The space
$M_{i+1,h}$ is a $\bP^{n-1}$-bundle over $M_{i,h}$, for $i=1,\ldots,h$. Therefore, $M_{N,h}$ is represented by a smooth projective variety.
\end{lem}

Indeed, let $\bX_{i}^{\on{univ}}$ denote the universal $p$-divisible group on $M_{i,h}$ appearing at the end of the chain. Let $M(\bX_{i}^{\on{univ}})^*$ denote the dual of the Lie algebra of the universal extension of $\bX^{\on{univ}}_i$ by vector groups. Let $M(\bX_{i}^{\on{univ}})^*_{\tau_j}=M(\bX_{i}^{\on{univ}})^*\otimes_{\mO_E,\tau_j}\overline\bF_p$.
Then we have a rank $n$ vector bundle $E_i=M(\bX_{i}^{\on{univ}})^*_{\tau_0}$ on $M_{i,h}$ and $M_{i+1,h}=\bP(E_i)$.

\begin{prop}\label{deperf of Demazure}
There is an isomorphism $\wGr'_N\simeq M_{N,h}$.
\end{prop}
Indeed, choosing a trivialization $\bD(\bX_0)\simeq W(k)^n=\mE_0$, and using the argument as in Proposition \ref{ADLV=RZ}, one shows that there is an isomorphism $M_{N,h}^\pf\simeq \wGr_N$ given by sending $\bX_0\to\cdots\to\bX_N$ to $\bD(\bX_N)_{\tau_0}\to\cdots\to\bD(\bX_0)_{\tau_0}$. On the other hand, there exists an open subscheme $\mathring{M}_{N,h}\subset M_{N,h}$ parameterizing those  chains such that the kernel of the map $\phi_N\cdots\phi_1:\bX_0\to \bX_N$ is \emph{not} contained in $\bX_0[p^{N-1}]$ and one can show that $\mathring{M}_{N,h}\simeq \Gr'_N$, compatible with $\wGr_N\simeq M_{N,h}^\pf$. Then by \eqref{push-out}, we have a projective birational universal homeomorphism $\wGr'_N\to M_{N,h}$. As $M_{N,h}$ is a smooth projective variety, we have $\wGr'_N\simeq M_{N,h}$. Details are left to readers.

\medskip

There is a line bundle on $M_{N,h}$ given by 
$$\tilde \mL_{\det}':=\bigotimes_{i=1}^N \omega_{\bX_{N},\tau_{-i}}^{-p^{i}},$$
where $\omega_{\bX_N,\tau_j}=\wedge^{\on{top}}(\Lie \bX_N)^*_{\tau_{j}}$. 
\begin{lem}
Under the map $\wGr_N\simeq M_{N,h}^\pf\stackrel{\varepsilon}{\to} M_{N,h}$, there is a canonical isomorphism $\tilde\mL_{\det}\simeq \varepsilon^*\tilde\mL'_{\det}$. 
\end{lem}
Now, in view of Proposition \ref{base point free}, Conjecture I would be a consequence of the following conjecture.

\medskip

\noindent\bf Conjecture IV. \rm The line bundle $\tilde \mL_{\det}'$ on $M_{N,h}$ is semi-ample.

\subsection{Example: $\wGr_2\to \bGr_2$.}\label{sample cal}

We assume that $p>2$ so that the Teichm\"{u}ller lifting of $-1$ is $-1$.
The purpose here is to illustrate the construction of \S~\ref{aff Grass for GLn} and \S~\ref{Dem Res} in the simplest but non-trivial case: $G=\GL_2$, and $N=2$. 
We hope to convince the readers that the bizarre-looking construction is indeed reasonable. We follow the notations there. We will see that $\bGr_2$ is a scheme, and a model of $\wGr_2\to \bGr_2$ can be regarded as a resolution of the singularities of $\bGr_2$. 

First, $\wGr_1=\bP^{1,\pf}$, over which there is a sequence of maps of locally free sheaves (notations from the proof of Proposition \ref{rep of dem res})
$$\mE/p\to \mO^2_{\bP^{1,\pf}}\to \mO_{\bP^{1,\pf}}(1).$$ Then $\wGr_2=\bP^\pf(\mE/p)$. From \eqref{switch}, we know that $\mE/p$ fits into the following exact sequence 
\[0\to \mO(1)\to \mE/p\to \mO(-1)\to 0.\]
Therefore, $\mE/p\simeq \mO(1)\oplus\mO(-1)$ (but this isomorphism is non-canonical), and $\wGr_2$ is isomorphic to $\bP^\pf(\mO(1)\oplus\mO(-1))$.

Next we consider $\bGr_2$. According to \S~\ref{representability},
\[\bGr_2=\bGr_{2,3}/L^3\GL_2.\]
Consider the scheme $U=\Spec R$ where $R=k[x,y,z]/x^2-yz$. There is a natural map $U^\pf\to \bGr_2$ given by
\[(x,y,z)\mapsto (\mE_0, A), \quad \mbox{where } A=\begin{pmatrix}p+[x] & -[y] \\ [z] & p-[x]\end{pmatrix}\in \GL_2(W(R^\pf)[1/p]).\]

\begin{lem}\label{chart Gr2}
This is an open embedding. 
\end{lem}
Note that in particular $\bGr_2$ has an open cover by $\Gr_2$ and $U^\pf$, and therefore is a scheme.
\begin{proof}
We write $G$ for $L^3\GL_2$ for simplicity. We lift $U^\pf\to\bGr_2$ to a map $U^\pf\to \bGr_{2,3}$ by sending $A\mapsto (\mE_0,A, \id)$. Then we need to show that the action map $ U^\pf\times G\to \bGr_{2,3}$ is an open embedding.

We need the following lemma whose proof is based some linear algebra calculation. Recall that $V_{2,3}$ is the scheme classifying pairs $(X,\la)\in M_2(W_3(R))\times R^\times$ satisfying $[\la]\det X=p^2$. We define a subfunctor $W\subset  V_{2,3}$, whose values at a perfect $k$-algebra $R$ consist of those $X\in V_{2,3}(R)$ such that there exist a decomposition $X=Ag$ with
$$g\in G(R),\quad A=\begin{pmatrix} p+[x] & -[y]  \\ [z] & p-[x] \end{pmatrix}\in M_{2}(W_3(R)),\ \det A=p^2.$$
\begin{lem}The functor $W$ is represented by an open subscheme of $V_{2,3}$. In addition, given $X\in W(R)$, the matrix $A$ is unique in the decomposition $X=Ag$ as above.
\end{lem}
\begin{proof}Let $X=\begin{pmatrix} a & b \\ c & d\end{pmatrix}\in V_{2,3}$, and let $X^*=\begin{pmatrix}d & -b \\ -c & a \end{pmatrix}$ be its adjugate matrix.
We write $a=[a_0]+p[a_1]+p^2[a_2]$, and similarly expand $b,c,d$. Note that $\det X$ is divisible by $p^2$ if and only if
\begin{equation}\label{eq for X}
a_0d_0=b_0c_0, \quad  \begin{pmatrix}a_1 & b_1 \\ c_1 & d_1\end{pmatrix}\begin{pmatrix}
d_0 & -b_0 \\ -c_0 & a_0 
\end{pmatrix}+ \begin{pmatrix}a_0 & b_0 \\ c_0 & d_0\end{pmatrix}\begin{pmatrix}
d_1 & -b_1\\ -c_1 & a_1 
\end{pmatrix}=0.
\end{equation}
In other words, $V_{2,3}$ is represented by an open subscheme of the affine scheme defined by equations in \eqref{eq for X}.

Let $\tilde W\subset V_{2,3}$ be the open subscheme defined by the condition $a_1d_1-b_1c_1\in R^*$. We claim that every $X\in \tilde{W}$ admits a decomposition $X=Ag$ with $(A,g)$ as in the definition of $W$. In addition, such $A$ is unique. Assuming this claim, we first finish the proof of the lemma.
Let $\tilde{W}G$ be the minimal $G$-invariant open subset of $V_{2,3}$ containing $\tilde{W}$. Then it follows from the claim that every $X\in \tilde{W}G$ admits such a decomposition $X=Ag$, i.e. $\tilde{W}G\subset W$. 
Conversely, let $X\in W$, with a decomposition $X=Ag$ as required. Since $A\in \tilde{W}$, we have $X\in \tilde{W}G$. Therefore, $W=\tilde{W}G$, and in the decomposition $X=Ag$, the matrix $A$ is unique. The lemma follows.

It remains to prove the above claim. Let $X\in \tilde W$ and suppose that $X=Ag$ is a required decomposition. Multiplying this equation by $X^*$ on the left and $g^{-1}$ on the right gives
\[p^2[\la]g^{-1}= \begin{pmatrix}d & -b \\ -c & a \end{pmatrix}\begin{pmatrix} p+[x] & -[y]  \\ [z] & p-[x] \end{pmatrix}.\]
It follows that the triple $(x,y,z)$ must satisfy 
\begin{equation}\label{two eq}
\begin{pmatrix}
d_0 & -b_0 \\ -c_0 & a_0 
\end{pmatrix}
\begin{pmatrix}x & -y \\ z & -x
\end{pmatrix}=0,
\quad 
\begin{pmatrix}
d_1 & -b_1 \\ -c_1 & a_1 
\end{pmatrix}
\begin{pmatrix}x & -y \\ z & -x
\end{pmatrix}=-\begin{pmatrix}d_0 & -b_0 \\ -c_0 & a_0\end{pmatrix}.
\end{equation}
When $X\in \tilde W$, it is easy to see that there is a unique triple $(x,y,z)$ satisfying these two equations. Namely,
\begin{equation}\label{solution of xyz}
\begin{pmatrix}
x & -y \\ z & -x
\end{pmatrix}=\frac{1}{b_1c_1-a_1d_1} \begin{pmatrix}a_1 & b_1 \\ c_1 & d_1\end{pmatrix}\begin{pmatrix}
d_0 & -b_0 \\ -c_0 & a_0 
\end{pmatrix}.
\end{equation}
Therefore, such $A$ in the decomposition $X=Ag$ is unique.

Conversely, for $X\in \tilde W$, let $A=\begin{pmatrix} p+[x] & -[y]  \\ [z] & p-[x] \end{pmatrix}$, where $(x,y,z)$ is given by \eqref{solution of xyz}. Using the lifting $V_{2,3}\to L^+V_2$ as in Remark \ref{set sect}, we regard $X$ and $A$ as elements in $L\GL_2$, denoted by $\tilde{X}$ and $\tilde{A}$. Then the determinant of
\[\tilde{g}:=\tilde{X}\tilde{A}^{-1}=p^{-2}\tilde{X}\tilde{A}^* \in L\GL_2\]
is an element in $W(R)^*$. But since \eqref{two eq} holds, the entries of $\tilde{g}$ are in fact in $W(R)$. Therefore, $\tilde{g}\in L^+\GL_2$. If we set $g= (\tilde{g} \mod p^3)$, then
$X=Ag$. The claim follows.
\end{proof}
Now let $V$ be the preimage of $W$ under the projection $\bGr_{2,3}\to V_{2,3}$. Then the action map $U^\pf\times G\to V$ is an isomorphism. In fact, the uniqueness of $A$ as in the above lemma implies that this map is a monomorphism. On the other hand, the definition of $W$ together with an argument as in Lemma \ref{key} implies the surjectivity of $R$-points. Therefore the lemma holds.
\end{proof}

Finally, one can also verify Conjecture I in this case. Namely, 
Let $j:\Gr_2\to \bGr_2$ be the open inclusion. One can restrict $\tilde\mL_{\det}$ to $\Gr_2\subset\wGr_2$. Then $\mL_{\det}=j_*(\mL_{\det}|_{\Gr_2})$.

\end{document}